\documentclass[]{lmcs}
\pdfoutput=1

\usepackage{lastpage}
\lmcsdoi{21}{2}{9}
\lmcsheading{}{\pageref{LastPage}}{}{}%
{Apr.~10,~2024}{Apr.~29,~2025}{}

\pdfoutput=1
\RequirePackage{hyperref}

\definecolor{nordred}{HTML}{bf616a}
\definecolor{nordblue}{HTML}{81a1c1}
\definecolor{norddarkblue}{HTML}{5e81ac}
\definecolor{nordgreen}{HTML}{a3be8c}
\definecolor{nordnight}{HTML}{4c566a}

\AtEndPreamble{%
\hypersetup{
  breaklinks = true,
  linktocpage,
  colorlinks = true,
  linkcolor = nordnight,
  citecolor = black,
  urlcolor = black,
}
}

\newcommand{\colorPre}[1]{{\color{nordred}{#1}}}
\newcommand{\colorMon}[1]{{#1}}
\newcommand{\pure}{\colorMon{pure}}
\newcommand{\effectful}{\colorPre{effectful}}

\usepackage{macros}
\usepackage[leftmargin=1em,rightmargin=1ex,vskip=1ex]{quoting}
\usepackage{lscape}
\usepackage{unicode-mario}
\usepackage{mathpartir}

\newcommand\MonCatStr{\hyperlink{linkStrictMonoidalFunctor}{\ensuremath{\mathbf{MonCat}_{\mathsf{Str}}}}}

\newcommand\strictMonoidalCategory{\hyperlink{linkStrictMonoidalCategory}{strict monoidal category}}
\newcommand\strictMonoidalCategories{\hyperlink{linkStrictMonoidalCategory}{strict monoidal categories}}

\newcommand\stringDiagrams{\hyperlink{linkStringDiagram}{string diagrams}}

\newcommand{\Run}{\colorPre{\operatorname{Run}}}
\newcommand{\PolyGraph}{\hyperlink{linkPolygraph}{\mathbf{PolyGr}}}
\newcommand{\EffPolyGraph}{\hyperlink{linkpolygraphcouple}{\mathbf{EffPolyGr}}}

\newcommand{\Put}{\mathsf{put}}
\newcommand{\Get}{\mathsf{get}} 

\newcommand{\IconPure}[1]{\scalebox{0.5}{
\begin{tikzpicture}[x=0.75pt,y=0.75pt,yscale=-1,xscale=1,baseline=-26,line width=1.25pt]
\draw    (36,8) -- (36,20) ;
\draw    (36,36) -- (36,48) ;
\draw   (8,20) -- (40,20) -- (40,36) -- (8,36) -- cycle ;
\draw    (12,8) -- (12,20) ;
\draw    (12,36) -- (12,48) ;
\draw (24,28) node  [font=\footnotesize]  {$\pmb{#1}$};
\draw (24,12) node  [font=\footnotesize]  {$\pmb{\dots}$};
\draw (24,42) node  [font=\footnotesize]  {$\pmb{\dots}$};
\end{tikzpicture}}}

\newcommand{\IconId}[1]{\scalebox{0.5}{
\begin{tikzpicture}[x=0.75pt,y=0.75pt,yscale=-1,xscale=1,line width=1.25pt, baseline=-26]
\draw    (13.5,14) -- (13.5,46) ;
\draw (13.5,10.6) node [anchor=south] [inner sep=0.75pt]  [font=\tiny]  {$\pmb{#1}$};
\end{tikzpicture}}}

\newcommand{\IconRun}{\scalebox{0.5}{
\begin{tikzpicture}[x=0.75pt,y=0.75pt,yscale=-1,xscale=1,line width=1.25pt, baseline=-26]
\draw  [color={rgb, 255:red, 191; green, 97; blue, 106 }  ,draw opacity=1 ] (13.5,14) -- (13.5,46) ;
\draw [color={rgb, 255:red, 191; green, 97; blue, 106 }  ,draw opacity=1 ] (13.5,10.6) node [anchor=south] [inner sep=0.75pt]  [font=\tiny]  {$\pmb{R}$};
\end{tikzpicture}}}

\newcommand{\EmptyBox}{\scalebox{0.5}{
\begin{tikzpicture}[x=0.75pt,y=0.75pt,yscale=-1,xscale=1,line width=1.25pt, baseline=-26]
\draw  [dash pattern={on 4.5pt off 4.5pt}] (0,16) -- (24,16) -- (24,36) -- (0,36) -- cycle ;
\end{tikzpicture}}}

\newcommand{\IconEff}[1]{\scalebox{0.5}{
\begin{tikzpicture}[x=0.75pt,y=0.75pt,yscale=-1,xscale=1,baseline=-26, line width=1.25pt]
\draw [color={rgb, 255:red, 191; green, 97; blue, 106 }  ,draw opacity=1 ]   (24,8) -- (24,20) ;
\draw    (36,8) -- (36,20) ;
\draw [color={rgb, 255:red, 191; green, 97; blue, 106 }  ,draw opacity=1 ]   (24,36) -- (24,48) ;
\draw    (36,36) -- (36,48) ;
\draw   (8,20) -- (40,20) -- (40,36) -- (8,36) -- cycle ;
\draw    (12,8) -- (12,20) ;
\draw    (12,36) -- (12,48) ;

\draw (24,28) node  [font=\footnotesize]  {$\pmb{#1}$};
\end{tikzpicture}}}

\newcommand{\IconComp}[2]{\scalebox{0.5}{
\begin{tikzpicture}[x=0.75pt,y=0.75pt,yscale=-1,xscale=1,line width=1.25pt,baseline=-34]
\draw [color={rgb, 255:red, 191; green, 97; blue, 106 }  ,draw opacity=1 ]   (24,12) -- (24,20) ;
\draw    (36,12) -- (36,20) ;
\draw [color={rgb, 255:red, 191; green, 97; blue, 106 }  ,draw opacity=1 ]   (24,36) -- (24,44) ;
\draw    (36,36) -- (36,44) ;
\draw   (8,20) -- (40,20) -- (40,36) -- (8,36) -- cycle ;
\draw    (12,12) -- (12,20) ;
\draw    (12,36) -- (12,44) ;
\draw [color={rgb, 255:red, 191; green, 97; blue, 106 }  ,draw opacity=1 ]   (24,60) -- (24,68) ;
\draw    (36,60) -- (36,68) ;
\draw   (8,44) -- (40,44) -- (40,60) -- (8,60) -- cycle ;
\draw    (12,60) -- (12,68) ;

\draw (24,28) node  [font=\footnotesize]  {$\pmb{#1}$};
\draw (24,52) node  [font=\footnotesize]  {$\pmb{#2}$};
\end{tikzpicture}}}

\newcommand{\IconThree}{\scalebox{0.5}{
\begin{tikzpicture}[x=0.75pt,y=0.75pt,yscale=-1,xscale=1,baseline=-26,line width=1.25]
\draw [color={rgb, 255:red, 191; green, 97; blue, 106 }  ,draw opacity=1 ]   (24,12) -- (24,20) ;
\draw    (36,12) -- (36,20) ;
\draw [color={rgb, 255:red, 191; green, 97; blue, 106 }  ,draw opacity=1 ]   (24,20) -- (24,44) ;
\draw    (36,20) -- (36,44) ;
\draw    (12,12) -- (12,20) ;
\draw    (12,20) -- (12,44) ;
\end{tikzpicture}}}

\newcommand{\IconSwapOne}{\scalebox{0.5}{\begin{tikzpicture}[x=0.75pt,y=0.75pt,yscale=-1,xscale=1,baseline=-26,line width=1.45pt]
    \draw    (8,8) -- (8,16) ;
    \draw [color={rgb, 255:red, 191; green, 97; blue, 106 }  ,draw opacity=1 ]   (8,32) -- (8,40) ;
    \draw [color={rgb, 255:red, 191; green, 97; blue, 106 }  ,draw opacity=1 ]   (24,16) .. controls (23.72,24.14) and (8,24.43) .. (8,32) ;
    \draw  [color={rgb, 255:red, 255; green, 255; blue, 255 }  ,draw opacity=1 ][fill={rgb, 255:red, 255; green, 255; blue, 255 }  ,fill opacity=1 ] (14,24.25) .. controls (14,23.15) and (14.9,22.25) .. (16,22.25) .. controls (17.11,22.25) and (18,23.15) .. (18,24.25) .. controls (18,25.35) and (17.11,26.25) .. (16,26.25) .. controls (14.9,26.25) and (14,25.35) .. (14,24.25) -- cycle ;
    \draw [color={rgb, 255:red, 0; green, 0; blue, 0 }  ,draw opacity=1 ]   (8,16) .. controls (7.72,24.14) and (24,24.43) .. (24,32) ;
    \draw    (24,32) -- (24.01,40) ;
    \draw [color={rgb, 255:red, 191; green, 97; blue, 106 }  ,draw opacity=1 ]   (24.01,8) -- (24,16) ;
    \end{tikzpicture}}}
\newcommand{\IconSwapTwo}{\scalebox{0.5}{
    \begin{tikzpicture}[x=0.75pt,y=0.75pt,yscale=-1,xscale=-1,baseline=-26,line width=1.45pt]
        \draw    (8,8) -- (8,16) ;
        \draw [color={rgb, 255:red, 191; green, 97; blue, 106 }  ,draw opacity=1 ]   (8,32) -- (8,40) ;
        \draw [color={rgb, 255:red, 191; green, 97; blue, 106 }  ,draw opacity=1 ]   (24,16) .. controls (23.72,24.14) and (8,24.43) .. (8,32) ;
        \draw  [color={rgb, 255:red, 255; green, 255; blue, 255 }  ,draw opacity=1 ][fill={rgb, 255:red, 255; green, 255; blue, 255 }  ,fill opacity=1 ] (14,24.25) .. controls (14,23.15) and (14.9,22.25) .. (16,22.25) .. controls (17.11,22.25) and (18,23.15) .. (18,24.25) .. controls (18,25.35) and (17.11,26.25) .. (16,26.25) .. controls (14.9,26.25) and (14,25.35) .. (14,24.25) -- cycle ;
        \draw [color={rgb, 255:red, 0; green, 0; blue, 0 }  ,draw opacity=1 ]   (8,16) .. controls (7.72,24.14) and (24,24.43) .. (24,32) ;
        \draw    (24,32) -- (24.01,40) ;
        \draw [color={rgb, 255:red, 191; green, 97; blue, 106 }  ,draw opacity=1 ]   (24.01,8) -- (24,16) ;
        \end{tikzpicture}}}

\newcommand{\IconThreeToOne}{\scalebox{0.5}{
\begin{tikzpicture}[x=0.75pt,y=0.75pt,yscale=-1,xscale=1,line width=1.25,baseline=-26]
\draw    (8,8) -- (8,16) ;
\draw [color={rgb, 255:red, 191; green, 97; blue, 106 }  ,draw opacity=1 ]   (8,32) -- (8,40) ;
\draw [color={rgb, 255:red, 191; green, 97; blue, 106 }  ,draw opacity=1 ]   (40,16) .. controls (39.71,24.14) and (8,24.43) .. (8,32) ;
\draw  [color={rgb, 255:red, 255; green, 255; blue, 255 }  ,draw opacity=1 ][fill={rgb, 255:red, 255; green, 255; blue, 255 }  ,fill opacity=1 ] (16.6,25.65) .. controls (16.6,24.55) and (17.5,23.65) .. (18.6,23.65) .. controls (19.71,23.65) and (20.6,24.55) .. (20.6,25.65) .. controls (20.6,26.75) and (19.71,27.65) .. (18.6,27.65) .. controls (17.5,27.65) and (16.6,26.75) .. (16.6,25.65) -- cycle ;
\draw    (24,32) -- (24.01,40) ;
\draw [color={rgb, 255:red, 191; green, 97; blue, 106 }  ,draw opacity=1 ]   (40,8) -- (40,16) ;
\draw    (24,8) -- (24,16) ;
\draw    (40,32) -- (40.01,40) ;
\draw  [color={rgb, 255:red, 255; green, 255; blue, 255 }  ,draw opacity=1 ][fill={rgb, 255:red, 255; green, 255; blue, 255 }  ,fill opacity=1 ] (27.2,22.67) .. controls (27.2,21.56) and (28.1,20.67) .. (29.2,20.67) .. controls (30.31,20.67) and (31.2,21.56) .. (31.2,22.67) .. controls (31.2,23.77) and (30.31,24.67) .. (29.2,24.67) .. controls (28.1,24.67) and (27.2,23.77) .. (27.2,22.67) -- cycle ;
\draw [color={rgb, 255:red, 0; green, 0; blue, 0 }  ,draw opacity=1 ]   (24,16) .. controls (23.72,24.14) and (40,24.43) .. (40,32) ;
\draw [color={rgb, 255:red, 0; green, 0; blue, 0 }  ,draw opacity=1 ]   (8,16) .. controls (7.72,24.14) and (24,24.43) .. (24,32) ;
\end{tikzpicture}}}

\newcommand{\IconOneToThree}{\scalebox{0.5}{
\begin{tikzpicture}[x=0.75pt,y=0.75pt,yscale=-1,xscale=-1,line width=1.25,baseline=-26]
\draw    (8,8) -- (8,16) ;
\draw [color={rgb, 255:red, 191; green, 97; blue, 106 }  ,draw opacity=1 ]   (8,32) -- (8,40) ;
\draw [color={rgb, 255:red, 191; green, 97; blue, 106 }  ,draw opacity=1 ]   (40,16) .. controls (39.71,24.14) and (8,24.43) .. (8,32) ;
\draw  [color={rgb, 255:red, 255; green, 255; blue, 255 }  ,draw opacity=1 ][fill={rgb, 255:red, 255; green, 255; blue, 255 }  ,fill opacity=1 ] (16.6,25.65) .. controls (16.6,24.55) and (17.5,23.65) .. (18.6,23.65) .. controls (19.71,23.65) and (20.6,24.55) .. (20.6,25.65) .. controls (20.6,26.75) and (19.71,27.65) .. (18.6,27.65) .. controls (17.5,27.65) and (16.6,26.75) .. (16.6,25.65) -- cycle ;
\draw    (24,32) -- (24.01,40) ;
\draw [color={rgb, 255:red, 191; green, 97; blue, 106 }  ,draw opacity=1 ]   (40,8) -- (40,16) ;
\draw    (24,8) -- (24,16) ;
\draw    (40,32) -- (40.01,40) ;
\draw  [color={rgb, 255:red, 255; green, 255; blue, 255 }  ,draw opacity=1 ][fill={rgb, 255:red, 255; green, 255; blue, 255 }  ,fill opacity=1 ] (27.2,22.67) .. controls (27.2,21.56) and (28.1,20.67) .. (29.2,20.67) .. controls (30.31,20.67) and (31.2,21.56) .. (31.2,22.67) .. controls (31.2,23.77) and (30.31,24.67) .. (29.2,24.67) .. controls (28.1,24.67) and (27.2,23.77) .. (27.2,22.67) -- cycle ;
\draw [color={rgb, 255:red, 0; green, 0; blue, 0 }  ,draw opacity=1 ]   (24,16) .. controls (23.72,24.14) and (40,24.43) .. (40,32) ;
\draw [color={rgb, 255:red, 0; green, 0; blue, 0 }  ,draw opacity=1 ]   (8,16) .. controls (7.72,24.14) and (24,24.43) .. (24,32) ;
\end{tikzpicture}}} 
\usepackage{cleveref} %

\usepackage{hyperref}
\usepackage{cleveref}

\newcommand{\renewtheorem}[1]{%
  \expandafter\let\csname #1\endcsname\relax
  \expandafter\let\csname c@#1\endcsname\relax
  \expandafter\let\csname end#1\endcsname\relax
  \newtheorem{#1}%
}

\theoremstyle{plain}

\renewtheorem{thm}{Theorem}[section]
\renewtheorem{cor}[thm]{Corollary}
\renewtheorem{lem}[thm]{Lemma}
\renewtheorem{prop}[thm]{Proposition}
\renewtheorem{asm}[thm]{Assumption}

\theoremstyle{definition}

\renewtheorem{rem}[thm]{Remark}
\renewtheorem{rems}[thm]{Remarks}
\renewtheorem{exa}[thm]{Example}
\renewtheorem{exas}[thm]{Examples}
\renewtheorem{defi}[thm]{Definition}
\renewtheorem{conv}[thm]{Convention}
\renewtheorem{conj}[thm]{Conjecture}
\renewtheorem{prob}[thm]{Problem}
\renewtheorem{oprob}[thm]{Open Problem}
\renewtheorem{algo}[thm]{Algorithm}
\renewtheorem{obs}[thm]{Observation}
\renewtheorem{qu}[thm]{Question}
\renewtheorem{fact}[thm]{Fact}
\renewtheorem{pty}[thm]{Property}

\newcommand\p[1]{#1} 
\renewcommand{\Run}{\colorPre{\operatorname{Run}}}

\newcommand\blackComultiplication{
\tikzset{every picture/.style={line width=0.85pt}} %
\begin{tikzpicture}[x=0.75pt,y=0.75pt,yscale=-0.4,xscale=0.4,baseline=-15pt,rotate=90,transform shape]
\draw    (40,39.92) -- (27,40) ;
\draw    (50,50.01) .. controls (40.47,50.32) and (41.08,47.82) .. (40,39.92) ;
\draw    (40,39.92) .. controls (42.58,30.82) and (40.87,29.92) .. (50,30.01) ;
\draw  [fill={rgb, 255:red, 0; green, 0; blue, 0 }  ,fill opacity=1 ] (36,39.92) .. controls (36,37.71) and (37.79,35.92) .. (40,35.92) .. controls (42.21,35.92) and (44,37.71) .. (44,39.92) .. controls (44,42.13) and (42.21,43.92) .. (40,43.92) .. controls (37.79,43.92) and (36,42.13) .. (36,39.92) -- cycle ;
\end{tikzpicture}
}

\newcommand\blackUnit{
\tikzset{every picture/.style={line width=0.85pt}} %
\begin{tikzpicture}[x=0.75pt,y=0.75pt,yscale=-0.4,xscale=0.4,baseline=8pt,rotate=270,transform shape]
\draw    (40,39.92) -- (27,40) ;
\draw  [fill={rgb, 255:red, 0; green, 0; blue, 0 }  ,fill opacity=1 ] (36,39.92) .. controls (36,37.71) and (37.79,35.92) .. (40,35.92) .. controls (42.21,35.92) and (44,37.71) .. (44,39.92) .. controls (44,42.13) and (42.21,43.92) .. (40,43.92) .. controls (37.79,43.92) and (36,42.13) .. (36,39.92) -- cycle ;
\end{tikzpicture}
}

\newcommand\whiteComonoidUnit{
\tikzset{every picture/.style={line width=0.85pt}} %
\begin{tikzpicture}[x=0.75pt,y=0.75pt,yscale=-0.4,xscale=0.4,baseline=-15pt,rotate=90,transform shape]
  \tikzset{every picture/.style={line width=0.85pt}} %
  \draw  (40,39.92) -- (27,40);
  \draw  [fill={rgb, 255:red, 255; green, 255; blue, 255}  ,fill opacity=1 ] (36,39.92) .. controls (36,37.71) and (37.79,35.92) .. (40,35.92) .. controls (42.21,35.92) and (44,37.71) .. (44,39.92) .. controls (44,42.13) and (42.21,43.92) .. (40,43.92) .. controls (37.79,43.92) and (36,42.13) .. (36,39.92) -- cycle ;
\end{tikzpicture}
}

\setlist[itemize]{leftmargin=10mm}
\setlist[enumerate]{leftmargin=10mm}

\title{String Diagrams for Premonoidal Categories}
\author[Román]{Mario Rom\'an\lmcsorcid{0000-0003-3158-1226}}[a]
\address{University of Oxford and Tallinn University of Technology}
\email{mario.roman.garcia@cs.ox.ac.uk}
\thanks{Mario Rom\'an and Pawe{\l} Sobociński were supported by the
Estonian Research Council grant PRG1210 and
 by the Advanced Research + Invention Agency
(ARIA) Safeguarded AI Programme. Pawe{\l} Sobociński was additionally supported by
European Union and Estonian Research Council via project TEM-TA5, by the Estonian Research Council funded Estonian Centre of Excellence in Artificial Intelligence project TK213U9 and by the European Union under Grant Agreement No. 101087529.
Mario Román was additionally supported
by the Air Force Office of Scientific Research (AFOSR) award number
FA9550-21-1-0038. This manuscript is an extended version of
``Promonads and String Diagrams for Effectful Categories'' by the first-named
author, presented at Applied Category Theory 2022.} %
\author[Sobociński]{Pawe{\l} Sobociński\lmcsorcid{0000-0002-7992-9685}}[b]
\address{Tallinn University of Technology}
\email{pawel.sobocinski@taltech.ee}

\makeatletter
\def\@copyrightspace{\relax}
\makeatother

\begin{document}

\begin{abstract}
  Premonoidal categories are monoidal categories without the interchange law while effectful categories are premonoidal categories with a chosen monoidal subcategory of interchanging morphisms. In the same sense that string diagrams—pioneered by Joyal and Street—are an internal language for monoidal categories, we show that string diagrams with an added ``runtime object''—pioneered by Alan Jeffrey—are an internal language for effectful categories and can be used as string diagrams for effectful, premonoidal and Freyd categories.
\end{abstract}

\maketitle

\section{Introduction}

\p{ Graphical programming languages have been of perennial interest to
  programmers: they are intuitive even for end-users without software
  development training, and they are closer to the digrammatic descriptions
  found across multiple branches of engineering \cite{kuhail21:visual}. While
  classical programming semantics employs categorical structures such as
  \emph{monads} and \emph{premonoidal categories} \cite{moggi91,power97},
  graphical programming semantics is traditionally a semi-formal field. This
  manuscript defends that graphical programming languages can be given semantics
  with the same formal techniques as their textual counterparts. }

\p{ Indeed, category theory has fostered two successful applications that are
  rarely combined: monoidal string diagrams \cite{joyal91} and functional
  programming semantics \cite{moggi91}. String diagrams are used to describe
  about quantum transformations \cite{abramsky2009categorical}, relational
  queries \cite{bonchi18}, and even computability \cite{pavlovic13}; at the same
  time, proof nets and the geometry of interaction \cite{girard89,blute96} have
  been widely applied in computer science \cite{abramsky02,hoshino14}. On the
  other hand, monads and comonads, Kleisli categories and
  \premonoidalCategories{} are used to explain effectful functional programming
  \cite{hughes00,jacobs09,moggi91,power99:freyd,uustalu2008comonadic}. Although
  Freyd categories with a cartesian base \cite{power02} are commonly used, it is
  also possible to consider non-cartesian Freyd categories \cite{staton13},
  which are called \emph{effectful categories}. }

\p{ These applications are well-known. However, some foundational results in the
  intersection between string diagrams, \premonoidalCategories{} and
  \effectfulCategories{} are missing in the literature. This manuscript
  contributes one such result: \emph{we introduce string diagrams for
  premonoidal and effectful categories}. Jeffrey \cite{jeffrey97} was the first
  to employ string diagrams of \premonoidalCategories{}, and we are indebted to
  his and Schweimeier's expositions \cite{schweimeier01}. Jeffrey's technique
  consisted in introducing an extra wire---which we call the
  \emph{runtime}---that prevents some morphisms from interchanging.

  We promote this technique into a result about the construction of free
  premonoidal and free \effectfulCategories{}: the free \premonoidalCategory{}
  can be constructed in terms of the free monoidal category with an additional wire.
  For this construction, we rely on Joyal and Street's  analogue result on
  string diagrams for monoidal categories \cite{joyal91}. Our slogan, which
  constitutes the statement of Theorem \ref{theorem:runtime-as-a-resource}, says
  that
  \begin{quote}
    \emph{Premonoidal categories are monoidal categories with a runtime.}
  \end{quote}
}

\subsection{Synopsis}

\p{ \Cref{sec:premonoidalcats,sec:effectfulcats} contain mostly preliminary
  material on premonoidal, Freyd and \effectfulCategories{}. \Cref{sec:runtime}
  recalls string diagrams for monoidal categories and constructs string diagrams
  for \effectfulCategories{}. \Cref{sec:stringPremonoidal} describes string
  diagrams for \premonoidalCategories{}. }

\subsection{Contributions}
Our main original contribution is in \Cref{sec:runtime}; we prove that
\effectfulCategories{} are monoidal categories with runtime via an adjunction
(Theorem \ref{theorem:runtime-as-a-resource}). We immediately obtain string diagrams
for \effectfulCategories{} (\Cref{cor:string-effectfuls}). Our second
contribution is an adjunction constructing string diagrams for
\premonoidalCategories{} (\Cref{thm:adjunction-string-premonoidals}).

\section{Premonoidal and Effectful Categories}
\label{sec:premonoidal}

\subsection{Premonoidal categories}
\label{sec:premonoidalcats}
\PremonoidalCategories{} are monoidal categories without the \emph{interchange law},
$(\fm \tensor \im) \comp (\im \tensor \gm) \neq (\im \tensor \gm) \comp (\fm \tensor \im)$.
This means that we cannot tensor any two arbitrary morphisms, $(\fm \tensor \gm)$, without explicitly stating which one is to be composed first, $(\fm \tensor \im) \comp (\im \tensor \gm)$ or $(\im \tensor \gm) \comp (\fm \tensor \im)$, and the two compositions are not equivalent (\Cref{fig:noninterchange}).
  \begin{figure}[H]
    \centering

\tikzset{every picture/.style={line width=0.75pt}} %

\begin{tikzpicture}[x=0.75pt,y=0.75pt,yscale=-1,xscale=1]
\draw   (4,20) -- (28,20) -- (28,32) -- (4,32) -- cycle ;
\draw    (16,4) -- (16,20) ;
\draw   (36,40) -- (60,40) -- (60,52) -- (36,52) -- cycle ;
\draw    (48,4) -- (48,40) ;
\draw    (48,52) -- (48,68) ;
\draw    (16,32) -- (16,68) ;
\draw   (100,40) -- (124,40) -- (124,52) -- (100,52) -- cycle ;
\draw    (144,4) -- (144,20) ;
\draw    (112,4) -- (112,40) ;
\draw    (112,52) -- (112,68) ;
\draw    (144,32) -- (144,68) ;
\draw  [draw opacity=0] (64,28) -- (96,28) -- (96,44) -- (64,44) -- cycle ;
\draw   (132,20) -- (156,20) -- (156,32) -- (132,32) -- cycle ;

\draw (16,26) node  [font=\footnotesize]  {$f$};
\draw (48,46) node  [font=\footnotesize]  {$g$};
\draw (112,46) node  [font=\footnotesize]  {$f$};
\draw (80,36) node  [font=\footnotesize]  {$\neq$};
\draw (144,26) node  [font=\footnotesize]  {$g$};

\end{tikzpicture}
     \caption{The interchange law does not hold in a premonoidal category.}
    \label{fig:noninterchange}
  \end{figure}

In technical terms, the tensor of a \premonoidalCategory{} $(\tensor) \colon \catC \times \catC \to \catC$ is not a functor, but only what is called a \emph{\sesquifunctor{}}: independently functorial on each variable. Indeed, tensoring with any identity, which is sometimes referred to as \emph{whiskering}, is a functor $(\bullet \tensor \im) \colon \catC \to \catC$, but there is no functor $(\bullet \tensor \bullet) \colon \catC \times \catC \to \catC$.

A good motivation for dropping the interchange law can be found when describing transformations that affect some global state.
These effectful processes should not interchange in general, because the order in which we modify the global state is meaningful.
For instance, in the Kleisli category of the \emph{writer monad}, $(\Sigma^{\ast} \times \bullet) ፡ \Set \to \Set$ for some alphabet $\Sigma \in \Set$, we can consider the function $\mathsf{print} \colon \Sigma^{\ast} \to \Sigma^{\ast} \times 1$, which is an effectful map from $\Sigma^\ast$ to $1$. The order in which we ``print'' does matter (\Cref{fig:writermonad}).
  \begin{figure}[H]
    \centering

\tikzset{every picture/.style={line width=0.75pt}} %

\begin{tikzpicture}[x=0.75pt,y=0.75pt,yscale=-1,xscale=1]
\draw   (4,4) -- (52,4) -- (52,20) -- (4,20) -- cycle ;
\draw   (56,4) -- (104,4) -- (104,20) -- (56,20) -- cycle ;
\draw   (4,28) -- (52,28) -- (52,44) -- (4,44) -- cycle ;
\draw   (56,44) -- (104,44) -- (104,60) -- (56,60) -- cycle ;
\draw    (28,20) -- (28,28) ;
\draw    (80,20) -- (80,44) ;
\draw  [draw opacity=0] (104,24) -- (136,24) -- (136,40) -- (104,40) -- cycle ;
\draw   (132,4) -- (204,4) -- (204,20) -- (132,20) -- cycle ;
\draw   (144,36) -- (192,36) -- (192,52) -- (144,52) -- cycle ;
\draw    (168,20) -- (168,36) ;
\draw  [draw opacity=0] (200,24) -- (232,24) -- (232,40) -- (200,40) -- cycle ;
\draw   (228,4) -- (300,4) -- (300,20) -- (228,20) -- cycle ;
\draw   (240,36) -- (288,36) -- (288,52) -- (240,52) -- cycle ;
\draw    (264,20) -- (264,36) ;
\draw  [draw opacity=0] (300,24) -- (332,24) -- (332,40) -- (300,40) -- cycle ;
\draw   (332,4) -- (380,4) -- (380,20) -- (332,20) -- cycle ;
\draw   (384,4) -- (432,4) -- (432,20) -- (384,20) -- cycle ;
\draw   (384,28) -- (432,28) -- (432,44) -- (384,44) -- cycle ;
\draw   (332,44) -- (380,44) -- (380,60) -- (332,60) -- cycle ;
\draw    (408,20) -- (408,28) ;
\draw    (356,20) -- (356,44) ;

\draw (28,12) node  [font=\scriptsize]  {``$hello$''};
\draw (80,12) node  [font=\scriptsize]  {``$world$''};
\draw (28,36) node  [font=\scriptsize]  {$print$};
\draw (80,52) node  [font=\scriptsize]  {$print$};
\draw (120,32) node  [font=\footnotesize]  {$=$};
\draw (168,12) node  [font=\scriptsize]  {``$helloworld$''};
\draw (168,44) node  [font=\scriptsize]  {$print$};
\draw (216,32) node  [font=\footnotesize]  {$\neq$};
\draw (264,12) node  [font=\scriptsize]  {``$worldhello$''};
\draw (264,44) node  [font=\scriptsize]  {$print$};
\draw (316,32) node  [font=\footnotesize]  {$=$};
\draw (356,12) node  [font=\scriptsize]  {``$hello$''};
\draw (408,12) node  [font=\scriptsize]  {``$world$''};
\draw (408,36) node  [font=\scriptsize]  {$print$};
\draw (356,52) node  [font=\scriptsize]  {$print$};

\end{tikzpicture}
     \caption{Writing does not interchange.}
    \label{fig:writermonad}
  \end{figure}

Not surprisingly, the paradigmatic examples of \premonoidalCategories{} are the
Kleisli categories of Set-based monads $T \colon \Set \to \Set$  (more
generally, of strong monads), which fail to be monoidal unless the monad itself
is commutative \cite{guitart1980tenseurs,power97,power99:freyd}.
Intuitively, the morphisms are ``effectful'', and these effects do not always
commute.

However, we may still want to allow some morphisms to interchange.
For instance, apart from asking the same associators and unitors of monoidal categories to exist, we ask them to be \emph{central}: that means that they interchange with any other morphism.
This notion of centrality forces us to write the definition of \premonoidalCategory{} in two different steps: first, we introduce the minimal setting in which centrality can be considered (\emph{binoidal} categories \cite{power99:freyd}) and then we use that setting to bootstrap the full definition of \premonoidalCategory{} with central coherence morphisms.

\begin{defi}[Binoidal category]
  \defining{linknoidalcategory}{}
  A \emph{binoidal category} is a category $\catC$ endowed with an object $I \in \catC$ and an object $A \tensor B$ for each $A \in \catC$ and $B \in \catC$.
  There are whiskering functors
  $(A \tensor \bullet) \colon \catC \to \catC
    \mbox{, and }
    (\bullet \tensor B) \colon \catC \to \catC$
  that coincide on $(A \tensor B)$, even if $(\bullet \tensor \bullet)$ is not necessarily itself a functor.
\end{defi}

Again, this means that we can tensor with identities (whiskering), functorially; but we cannot tensor two arbitrary morphisms: the interchange law stops being true in general.

\begin{defi}[Centre, central morphism]
The \emph{centre} of a binoidal category, $\zentre(\catC)$, is the wide
subcategory of morphisms that do satisfy the interchange law with any other
morphism. That is, $f \colon A \to B$ is \emph{central} if, for each $g \colon
A' \to B'$,
\begin{align*}
  & (f \tensor \im_{A'}) \comp (\im_{B} \tensor g)
   = (\im_{A} \tensor g) \comp (f \tensor \im_{B'}), \mbox{ and } \\
  & (\im_{A'} \tensor f) \comp (g \tensor \im_{B})
   = (g \tensor \im_{A}) \comp (\im_{B'} \tensor f).
\end{align*}
\end{defi}

\begin{defi}[Premonoidal category]
  \defining{linkpremonoidalcategory}
  A \emph{premonoidal category} is a \noidalCategory{} $(\catC,\tensor,I)$ together with the following coherence isomorphisms
  $\alpha_{A,B,C} \colon A \tensor (B \tensor C) \to (A \tensor B) \tensor C$, $\rho_{A} \colon A \tensor I \to A$ and $\lambda_{A} \colon I \tensor A \to A$ which are central, natural \emph{separately at each given component}, and satisfy the pentagon and triangle equations.

  A \premonoidalCategory{} is \emph{strict} when these coherence morphisms are identities.
  A \premonoidalCategory{} is moreover \emph{symmetric} when it is endowed with a coherence isomorphism $\sigma_{A,B} \colon A \tensor B \to B \tensor A$ that is central and natural at each given component, and satisfies the symmetry condition and hexagon equations.
\end{defi}

\begin{rem}
  The coherence theorem of monoidal categories still holds for \premonoidalCategories{}: every premonoidal is equivalent to a strict one.
  We will construct the free strict \premonoidalCategory{} using string diagrams.
  However, the usual string diagrams for monoidal categories need to be restricted: in \premonoidalCategories{}, we cannot consider two morphisms in parallel unless one of the two is \emph{central}.
\end{rem}

\subsection{Effectful and Freyd categories}
\label{sec:effectfulcats}

\PremonoidalCategories{} immediately present a problem: what are the strong
premonoidal functors? If we want them to compose, they should preserve
centrality of the coherence morphisms (so that the central coherence morphisms
of $F \comp G$ are these of $F$ after applying $G$), but naively asking them to
preserve all central morphisms would rule out otherwise valid
functors\footnote{For instance, the inclusion of a commutative monoid into a
non-commutative monoid does not necessarily preserve the monoid centre; in turn, this
induces an inclusion between the Kleisli categories of their writer monads that
does not necessarily preserve the premonoidal centre.}~\cite{staton13}. The solution is to
explicitly choose some central morphisms that represent ``pure'' computations.
These do not need to form the whole centre: it could be that some morphisms
considered \emph{effectful} just ``happen'' to fall in the centre of the
category, while we do not ask our functors to preserve them. This is the
well-studied notion of a \emph{non-cartesian Freyd category}, which we shorten
to \emph{effectful monoidal category} or \emph{effectful category}.

Thus, \effectfulCategories{} are \premonoidalCategories{} endowed with a chosen
family of central morphisms. These central morphisms are called \pure{}
morphisms, constrasting with the general, non-central, morphisms that fall
outside this family, which we call \effectful{}.

\begin{defi}
  \defining{linkeffectful}{}
  \defining{linkfreyd}{}
  An \emph{effectful category}
  is an identity-on-objects functor $\colorMon{\baseV} \to \colorPre{\catC}$ from a monoidal category $\baseV$ (the \pure{} morphisms, or ``values'') to a \premonoidalCategory{} $\catC$ (the \effectful{} morphisms, or ``computations''), that strictly preserves all of the premonoidal structure and whose image is central. It is \emph{strict} when both are. A \emph{Freyd category} \cite{levy2004} is an \effectfulCategory{} where the \pure{} morphisms form a cartesian monoidal category.
\end{defi}

\begin{rem}[Terminology]
  The name ``Freyd category'' sometimes assumes cartesianity of the pure morphisms, but it is sometimes also used for the general case.
To avoid this clash of nomenclature we reserve ``effectful categories'' for the general case and ``Freyd categories'' for the cartesian case.
There exists also the more fine-grained notion of ``Cartesian effect category'' \cite{dumas11}, which generalizes Freyd categories.
\end{rem}

\EffectfulCategories{} solve the problem of defining premonoidal functors: a functor between \effectfulCategories{} only needs to preserve the \pure{} morphisms.
We are not losing expressivity: \premonoidalCategories{} are effectful with their own centre, $\colorMon{\zentre}(\catC) \to \catC$.
From now on, we focus on \effectfulCategories{}.

\begin{defi}[Effectful functor]
  \defining{linkeffectfulfunctor}
  Let $\cbaseV \to \ccatC$ and $\cbaseW \to \ccatD$ be \effectfulCategories{}.
  An \emph{effectful functor} is a quadruple  $(F,F_{0},\varepsilon, \mu)$ consisting of a functor $F \colon \ccatC \to \ccatD$ and a functor $F_{0} \colon \cbaseV \to \cbaseW$ making the square commute,
  and two natural and pure isomorphisms $\varepsilon \colon J \cong F(I)$ and $\mu \colon F(A \otimes B) \cong F(A) \otimes F(B)$ such that they make $F_{0}$ a monoidal functor. It is \emph{strict} if these are identities.
  Strict \effectfulCategories{} and functors form a category, $\EffCatStr$.
\end{defi}

When drawing string diagrams in an \effectfulCategory{}, we could use two
different colours to declare if we are depicting either a value or a computation
(\Cref{fig:prehelloworld}).

  \begin{figure}[H]
    \centering

\tikzset{every picture/.style={line width=0.75pt}} %

\begin{tikzpicture}[x=0.75pt,y=0.75pt,yscale=-1,xscale=1]
\draw   (4,4) -- (52,4) -- (52,20) -- (4,20) -- cycle ;
\draw   (56,4) -- (104,4) -- (104,20) -- (56,20) -- cycle ;
\draw  [color={rgb, 255:red, 191; green, 97; blue, 106 }  ,draw opacity=1 ] (4,28) -- (52,28) -- (52,44) -- (4,44) -- cycle ;
\draw  [color={rgb, 255:red, 191; green, 97; blue, 106 }  ,draw opacity=1 ] (56,44) -- (104,44) -- (104,60) -- (56,60) -- cycle ;
\draw    (28,20) -- (28,28) ;
\draw    (80,20) -- (80,44) ;
\draw  [draw opacity=0] (104,24) -- (136,24) -- (136,40) -- (104,40) -- cycle ;
\draw   (136,4) -- (184,4) -- (184,20) -- (136,20) -- cycle ;
\draw   (188,4) -- (236,4) -- (236,20) -- (188,20) -- cycle ;
\draw  [color={rgb, 255:red, 191; green, 97; blue, 106 }  ,draw opacity=1 ] (188,28) -- (236,28) -- (236,44) -- (188,44) -- cycle ;
\draw  [color={rgb, 255:red, 191; green, 97; blue, 106 }  ,draw opacity=1 ] (136,44) -- (184,44) -- (184,60) -- (136,60) -- cycle ;
\draw    (212,20) -- (212,28) ;
\draw    (160,20) -- (160,44) ;

\draw (28,12) node  [font=\scriptsize]  {``$hello$''};
\draw (80,12) node  [font=\scriptsize]  {``$world$''};
\draw (28,36) node  [font=\scriptsize,color={rgb, 255:red, 191; green, 97; blue, 106 }  ,opacity=1 ]  {$print$};
\draw (80,52) node  [font=\scriptsize,color={rgb, 255:red, 191; green, 97; blue, 106 }  ,opacity=1 ]  {$print$};
\draw (120,32) node  [font=\footnotesize]  {$\neq$};
\draw (160,12) node  [font=\scriptsize]  {``$hello$''};
\draw (212,12) node  [font=\scriptsize]  {``$world$''};
\draw (212,36) node  [font=\scriptsize,color={rgb, 255:red, 191; green, 97; blue, 106 }  ,opacity=1 ]  {$print$};
\draw (160,52) node  [font=\scriptsize,color={rgb, 255:red, 191; green, 97; blue, 106 }  ,opacity=1 ]  {$print$};

\end{tikzpicture}
    \caption{``Hello world'' is not ``world hello''.}
    \label{fig:prehelloworld}
  \end{figure}

Here, the values \colorMon{``hello''} and \colorMon{``world''} satisfy the
interchange law as in an ordinary monoidal category. However, the effectful
computation \colorPre{``print''} does not need to satisfy the interchange law.
String diagrams like these can be found in the work of Alan Jeffrey
\cite{jeffrey97}. Jeffrey presents a clever mechanism to graphically depict the
failure of interchange: all effectful morphisms need to have a control wire as
an input and output. This control wire needs to be passed around to all the
computations in order, and it prevents them from interchanging.

  \begin{figure}[H]
    \centering

\tikzset{every picture/.style={line width=0.75pt}} %

\begin{tikzpicture}[x=0.75pt,y=0.75pt,yscale=-1,xscale=1]
\draw    (72,56) -- (72,68) ;
\draw [color={rgb, 255:red, 0; green, 0; blue, 0 }  ,draw opacity=1 ]   (46,28) .. controls (46.17,36.17) and (39.83,31.83) .. (40,40) ;
\draw  [color={rgb, 255:red, 0; green, 0; blue, 0 }  ,draw opacity=1 ] (24,12) -- (68,12) -- (68,28) -- (24,28) -- cycle ;
\draw  [color={rgb, 255:red, 0; green, 0; blue, 0 }  ,draw opacity=1 ] (72,12) -- (116,12) -- (116,28) -- (72,28) -- cycle ;
\draw  [color={rgb, 255:red, 191; green, 97; blue, 106 }  ,draw opacity=1 ] (12,40) -- (52,40) -- (52,56) -- (12,56) -- cycle ;
\draw  [color={rgb, 255:red, 191; green, 97; blue, 106 }  ,draw opacity=1 ] (44,68) -- (84,68) -- (84,84) -- (44,84) -- cycle ;
\draw  [color={rgb, 255:red, 0; green, 0; blue, 0 }  ,draw opacity=1 ] (204,12) -- (248,12) -- (248,28) -- (204,28) -- cycle ;
\draw  [color={rgb, 255:red, 191; green, 97; blue, 106 }  ,draw opacity=1 ] (192,40) -- (232,40) -- (232,56) -- (192,56) -- cycle ;
\draw [color={rgb, 255:red, 191; green, 97; blue, 106 }  ,draw opacity=1 ]   (12,4) -- (12,20) ;
\draw [color={rgb, 255:red, 191; green, 97; blue, 106 }  ,draw opacity=1 ]   (12,20) .. controls (12.17,31.83) and (23.83,27.83) .. (24,40) ;
\draw [color={rgb, 255:red, 191; green, 97; blue, 106 }  ,draw opacity=1 ]   (32,56) .. controls (32.17,69.17) and (56.17,56.5) .. (56,68) ;
\draw [color={rgb, 255:red, 0; green, 0; blue, 0 }  ,draw opacity=1 ]   (94,40) .. controls (94.5,49.17) and (71.83,46.5) .. (72,56) ;
\draw [color={rgb, 255:red, 0; green, 0; blue, 0 }  ,draw opacity=1 ]   (94,28) -- (94,40) ;
\draw [color={rgb, 255:red, 191; green, 97; blue, 106 }  ,draw opacity=1 ]   (64,84) -- (64,96) ;
\draw [color={rgb, 255:red, 191; green, 97; blue, 106 }  ,draw opacity=1 ]   (144,8) -- (144,28) ;
\draw [color={rgb, 255:red, 191; green, 97; blue, 106 }  ,draw opacity=1 ]   (144,28) .. controls (144.17,41.17) and (204.17,28.5) .. (204,40) ;
\draw [color={rgb, 255:red, 191; green, 97; blue, 106 }  ,draw opacity=1 ]   (204,56) .. controls (204,72.07) and (160.2,51.27) .. (160,68) ;
\draw [color={rgb, 255:red, 0; green, 0; blue, 0 }  ,draw opacity=1 ]   (226,28) .. controls (226.17,36.17) and (219.83,31.83) .. (220,40) ;
\draw [color={rgb, 255:red, 0; green, 0; blue, 0 }  ,draw opacity=1 ]   (178,44) .. controls (178.17,52.17) and (175.83,47.83) .. (176,56) ;
\draw [color={rgb, 255:red, 191; green, 97; blue, 106 }  ,draw opacity=1 ]   (168,84) -- (168,96) ;
\draw  [color={rgb, 255:red, 255; green, 255; blue, 255 }  ,draw opacity=1 ][fill={rgb, 255:red, 255; green, 255; blue, 255 }  ,fill opacity=1 ] (176,34.66) .. controls (176,33.56) and (176.9,32.66) .. (178,32.66) .. controls (179.1,32.66) and (180,33.56) .. (180,34.66) .. controls (180,35.77) and (179.1,36.66) .. (178,36.66) .. controls (176.9,36.66) and (176,35.77) .. (176,34.66) -- cycle ;
\draw    (178,28) -- (178,44) ;
\draw  [color={rgb, 255:red, 0; green, 0; blue, 0 }  ,draw opacity=1 ] (156,12) -- (200,12) -- (200,28) -- (156,28) -- cycle ;
\draw  [color={rgb, 255:red, 255; green, 255; blue, 255 }  ,draw opacity=1 ][fill={rgb, 255:red, 255; green, 255; blue, 255 }  ,fill opacity=1 ] (174,61.45) .. controls (174,60.35) and (174.9,59.45) .. (176,59.45) .. controls (177.1,59.45) and (178,60.35) .. (178,61.45) .. controls (178,62.56) and (177.1,63.45) .. (176,63.45) .. controls (174.9,63.45) and (174,62.56) .. (174,61.45) -- cycle ;
\draw    (176,56) -- (176,68) ;
\draw  [color={rgb, 255:red, 191; green, 97; blue, 106 }  ,draw opacity=1 ] (148,68) -- (188,68) -- (188,84) -- (148,84) -- cycle ;

\draw (46,20) node  [font=\scriptsize,color={rgb, 255:red, 0; green, 0; blue, 0 }  ,opacity=1 ]  {``$hello$''};
\draw (94,20) node  [font=\scriptsize,color={rgb, 255:red, 0; green, 0; blue, 0 }  ,opacity=1 ]  {``$world$''};
\draw (32,48) node  [font=\scriptsize,color={rgb, 255:red, 191; green, 97; blue, 106 }  ,opacity=1 ]  {$print$};
\draw (64,76) node  [font=\scriptsize,color={rgb, 255:red, 191; green, 97; blue, 106 }  ,opacity=1 ]  {$print$};
\draw (128,40) node  [font=\footnotesize]  {$\neq $};
\draw (178,20) node  [font=\scriptsize,color={rgb, 255:red, 0; green, 0; blue, 0 }  ,opacity=1 ]  {``$hello$''};
\draw (226,20) node  [font=\scriptsize,color={rgb, 255:red, 0; green, 0; blue, 0 }  ,opacity=1 ]  {``$world$''};
\draw (168,76) node  [font=\scriptsize,color={rgb, 255:red, 191; green, 97; blue, 106 }  ,opacity=1 ]  {$print$};
\draw (212,48) node  [font=\scriptsize,color={rgb, 255:red, 191; green, 97; blue, 106 }  ,opacity=1 ]  {$print$};

\end{tikzpicture}
    \caption{An additional wire prevents interchange.}
    \label{fig:helloworld}
  \end{figure}

A common interpretation of monoidal categories is as theories of resources. We
can interpret premonoidal categories as monoidal categories with an extra
resource---the \colorPre{``runtime''}---that needs to be passed to all
computations. The next section promotes Jeffrey's observation into a theorem.

\begin{rem}[Non-interchanging string diagrams]
  Could we avoid the extra runtime wire and simply keep track of which morphisms must not be interchanged?
  This is an easy solution and we employ it in \Cref{fig:prehelloworld}; unfortunately, it raises an important concern: we lose a ``locality of substitutions'' property that, intuitively, we expect string diagrams to satisfy.

  Explicitly, imagine that we know that a pure morphism arises as the composition of two effectful morphisms, $f = g ⨾ h$.
  We expect to be able to substitute $f$ by the pair of morphisms $g ⨾ h$ at whichever point in a string diagram without regarding the rest of the diagram, as we do with monoidal string diagrams. However, say that $f$ is tensored with an effectful morphism $k$; in this case, the relative order of the morphisms, $g$, $h$ and $k$, is meaningful and forces us to be careful about the position of the rest of the nodes on the diagram (see \Cref{fig:substituting-one}).
  \begin{figure}[ht]
    \centering

\tikzset{every picture/.style={line width=0.75pt}} %

\begin{tikzpicture}[x=0.75pt,y=0.75pt,yscale=-1,xscale=1]
\draw    (40,16) -- (40,36) ;
\draw    (40,48) -- (40,68) ;
\draw   (28,36) -- (52,36) -- (52,48) -- (28,48) -- cycle ;
\draw    (96,16) -- (96,24) ;
\draw    (96,36) -- (96,48) ;
\draw    (96,60) -- (96,68) ;
\draw  [draw opacity=0] (56,20) -- (80,20) -- (80,60) -- (56,60) -- cycle ;
\draw   (84,24) -- (108,24) -- (108,36) -- (84,36) -- cycle ;
\draw   (84,48) -- (108,48) -- (108,60) -- (84,60) -- cycle ;
\draw  [draw opacity=0] (108,20) -- (132,20) -- (132,60) -- (108,60) -- cycle ;
\draw  [draw opacity=0] (128,20) -- (228,20) -- (228,60) -- (128,60) -- cycle ;
\draw    (284,16) -- (284,36) ;
\draw    (284,48) -- (284,68) ;
\draw   (272,36) -- (296,36) -- (296,48) -- (272,48) -- cycle ;
\draw  [draw opacity=0] (296,20) -- (320,20) -- (320,60) -- (296,60) -- cycle ;
\draw  [draw opacity=0] (380,20) -- (404,20) -- (404,60) -- (380,60) -- cycle ;
\draw    (256,16) -- (256,36) ;
\draw    (256,48) -- (256,68) ;
\draw   (244,36) -- (268,36) -- (268,48) -- (244,48) -- cycle ;
\draw    (336,16) -- (336,36) ;
\draw    (336,48) -- (336,68) ;
\draw   (324,36) -- (348,36) -- (348,48) -- (324,48) -- cycle ;
\draw    (364,16) -- (364,24) ;
\draw    (364,36) -- (364,48) ;
\draw   (352,24) -- (376,24) -- (376,36) -- (352,36) -- cycle ;
\draw    (364,60) -- (364,68) ;
\draw   (352,48) -- (376,48) -- (376,60) -- (352,60) -- cycle ;

\draw (40,42) node  [font=\footnotesize]  {$f$};
\draw (96,30) node  [font=\footnotesize]  {$g$};
\draw (96,54) node  [font=\footnotesize]  {$h$};
\draw (68,40) node  [font=\footnotesize]  {$=$};
\draw (120,40) node  [font=\footnotesize]  {};
\draw (178,40) node  [font=\footnotesize]  {$does\ not\ imply$};
\draw (284,42) node  [font=\footnotesize]  {$f$};
\draw (308,40) node  [font=\footnotesize]  {$=$};
\draw (256,42) node  [font=\footnotesize]  {$k$};
\draw (336,42) node  [font=\footnotesize]  {$k$};
\draw (364,30) node  [font=\footnotesize]  {$g$};
\draw (364,54) node  [font=\footnotesize]  {$h$};

\end{tikzpicture}
    \caption{Substitution with non-interchanging string diagrams.}
    \label{fig:substituting-one}
  \end{figure}

  Considering runtime as an extra wire addresses this problem in two ways: \emph{(i)} it becomes obvious that pure and effectful morphisms are topologically different, and \emph{(ii)} it shows how the second substitution is not local and should not follow from the assumption (see \Cref{fig:substituting-two}).
  \begin{figure}[ht]
    \centering

\tikzset{every picture/.style={line width=0.75pt}} %

\begin{tikzpicture}[x=0.75pt,y=0.75pt,yscale=-1,xscale=1]
\draw [color={rgb, 255:red, 191; green, 97; blue, 106 }  ,draw opacity=1 ]   (348,16) -- (348,68) ;
\draw [color={rgb, 255:red, 191; green, 97; blue, 106 }  ,draw opacity=1 ]   (260,16) -- (260,68) ;
\draw [color={rgb, 255:red, 191; green, 97; blue, 106 }  ,draw opacity=1 ]   (92,16) -- (92,68) ;
\draw    (40,16) -- (40,36) ;
\draw    (40,48) -- (40,68) ;
\draw   (28,36) -- (52,36) -- (52,48) -- (28,48) -- cycle ;
\draw    (108,16) -- (108,24) ;
\draw    (108,36) -- (108,48) ;
\draw    (108,60) -- (108,68) ;
\draw  [draw opacity=0] (56,20) -- (80,20) -- (80,60) -- (56,60) -- cycle ;
\draw  [fill={rgb, 255:red, 255; green, 255; blue, 255 }  ,fill opacity=1 ] (84,24) -- (116,24) -- (116,36) -- (84,36) -- cycle ;
\draw  [fill={rgb, 255:red, 255; green, 255; blue, 255 }  ,fill opacity=1 ] (84,48) -- (116,48) -- (116,60) -- (84,60) -- cycle ;
\draw  [draw opacity=0] (128,20) -- (228,20) -- (228,60) -- (128,60) -- cycle ;
\draw    (284,16) -- (284,36) ;
\draw    (284,48) -- (284,68) ;
\draw   (272,36) -- (296,36) -- (296,48) -- (272,48) -- cycle ;
\draw  [draw opacity=0] (296,20) -- (320,20) -- (320,60) -- (296,60) -- cycle ;
\draw    (244,16) -- (244,36) ;
\draw    (244,48) -- (244,68) ;
\draw  [fill={rgb, 255:red, 255; green, 255; blue, 255 }  ,fill opacity=1 ] (236,36) -- (268,36) -- (268,48) -- (236,48) -- cycle ;
\draw    (332,16) -- (332,36) ;
\draw    (332,48) -- (332,68) ;
\draw  [fill={rgb, 255:red, 255; green, 255; blue, 255 }  ,fill opacity=1 ] (324,36) -- (356,36) -- (356,48) -- (324,48) -- cycle ;
\draw    (364,16) -- (364,20) ;
\draw    (364,32) -- (364,52) ;
\draw  [fill={rgb, 255:red, 255; green, 255; blue, 255 }  ,fill opacity=1 ] (340,20) -- (372,20) -- (372,32) -- (340,32) -- cycle ;
\draw    (364,64) -- (364,68) ;
\draw  [fill={rgb, 255:red, 255; green, 255; blue, 255 }  ,fill opacity=1 ] (340,52) -- (372,52) -- (372,64) -- (340,64) -- cycle ;
\draw [color={rgb, 255:red, 191; green, 97; blue, 106 }  ,draw opacity=1 ]   (20,16) -- (20,68) ;

\draw (40,42) node  [font=\footnotesize]  {$f$};
\draw (100,30) node  [font=\footnotesize]  {$g$};
\draw (100,54) node  [font=\footnotesize]  {$h$};
\draw (68,40) node  [font=\footnotesize]  {$=$};
\draw (178,40) node  [font=\footnotesize]  {$does\ not\ imply$};
\draw (284,42) node  [font=\footnotesize]  {$f$};
\draw (308,40) node  [font=\footnotesize]  {$=$};
\draw (252,42) node  [font=\footnotesize]  {$k$};
\draw (340,42) node  [font=\footnotesize]  {$k$};
\draw (356,26) node  [font=\footnotesize]  {$g$};
\draw (356,58) node  [font=\footnotesize]  {$h$};
\end{tikzpicture}     %
    \caption{Substitution with runtime string diagrams.}
    \label{fig:substituting-two}
  \end{figure}
\end{rem}

\section{String Diagrams for Effectful Categories}
\label{sec:runtime}

\subsection{String Diagrams for Monoidal Categories}
Soundness and completeness of string diagrams can be proven by showing that the morphisms of the monoidal category freely generated over a \polygraph{} of generators are string diagrams on these generators, quotiented by topological deformations \cite{joyal91}.
We will justify string diagrams for \premonoidalCategories{} by proving that the freely generated effectful category over a pair of polygraphs (for pure and effectful generators, respectively) can be constructed as the freely generated monoidal category over a particular polygraph that includes an extra wire. Let us start by recalling string diagrams for monoidal categories, following the work of Joyal and Street.

\begin{defi}
  \defining{linkpolygraph}{}\defining{linkPolygraph}{} A \emph{\polygraph{}}
  $\hyG$ (analogue of a \emph{multigraph} \cite{shulman2016categorical}) is
  given by a set of objects, $\hyGobj$, and a set of generators
  $\hyG(A_{0},\dots,A_{n};B_{0},\dots,B_{m})$ for any two sequences of objects
  $A_{0},\dots,A_{n}$ and $B_{0},\dots,B_{m}$. A morphism of \polygraphs{} $f
  \colon \hyG \to \hyH$ is a function between their object sets,
  $f_{\mathrm{obj}} \colon \hyGobj \to \hyHobj$, and a function between their
  corresponding morphism sets,
  \begin{align*}
     f_{A_{0},\dots,A_{n}; B_{0},\dots,B_{n}} & \colon  \hyG(A_{0},\dots,A_{n};B_{0},\dots,B_{m}) \\ & \to \hyH(f_{\mathrm{obj}}(A_{0}),\dots,f_{\mathrm{obj}}(A_{n});f_{\mathrm{obj}}(B_{0}),\dots,f_{\mathrm{obj}}(B_{m})).
  \end{align*}
  Polygraphs and polygraph morphisms form a category, $\PolyGraph$.
\end{defi}

  There exists an adjunction between \polygraphs{} and strict monoidal
  categories. Any monoidal category $\catC$ can be seen as a \polygraph{},
  $\Forget(\catC)$; the objects are those of the monoidal category; the edges,
  $\Forget(\catC)(A_{0},\dots, A_{n};B_{0}, \dots, B_{m})$, are the morphisms
  $\catC(A_{0} \tensor \dots \tensor A_{n},B_{0} \tensor \dots \tensor B_{m})$;
  we forget about composition and tensoring. Given a \polygraph{} $\hyG$, the
  free strict monoidal category $\MON( \hyG)$ is the strict monoidal category
  that has as morphisms the string diagrams over the generators in the
  \polygraph{}.

\begin{defi}
  \defining{linkStringDiagram}{}
  A \emph{string diagram} over a \emph{polygraph} $𝓖$ (or \emph{progressive
  plane graph} in the work of Joyal and Street \cite[Definition 1.1]{joyal91})
  is a graph $Γ$ embedded in the squared interval, up to planar isotopy, and
  such that
  \begin{enumerate}
    \item the boundary of the graph touches only the top and the bottom of the
    square, $δ Γ \subseteq \{0,1\} × [0,1]$;
    \item and the second projection is injective on each component of the graph
    without its vertices, $Γ - Γ₀$; this makes it acyclic and progressive.
  \end{enumerate}
  We call the components of $Γ - Γ₀$ \emph{wires}, $W$; we call the vertices of
  the graph \emph{nodes}, $Γ₀$. Wires must be labeled by the objects of the
  \polygraph{}, $o ፡ W → 𝓖_{obj}$, nodes must be labeled by the generators of
  the \polygraph{}, $m ፡ Γ₀ → 𝓖$; and each node must be connected to wires
  exactly typed by the objects of its generator---a string diagram must be
  well-typed.
\end{defi}

\begin{rem}[The symmetric case]
  Compare these with the string diagrams for symmetric monoidal categories \cite{selinger2010survey}. String diagrams for symmetric monoidal categories have a practical combinatorial characterization in terms of acyclic hypergraphs \cite{bonchi22:rewritingStrings}. This characterization would simplify our results, but we prefer to work for the full generality of (non-symmetric) monoidal categories.
\end{rem}

\begin{lem}
  String diagrams over a \polygraph{} $𝓖$ form a \monoidalCategory{}, which we call $\StringMon(𝓖)$. This determines a functor,
  $$\StringMon ፡ \PolyGraph → \MonCatStr.$$
\end{lem}
\begin{proof}[Proof sketch]
  The objects of the category are lists of objects of the \polygraph{}, which we write as $[X₀ , \dots , Xₙ]$, for $Xᵢ ∈ 𝓖_{obj}$. These form a (free) monoid with concatenation and the empty list.

  Morphisms $[X₀ , \dots , Xₙ] → [Y₀ , \dots , Yₘ]$ are \stringDiagrams{} over the \polygraph{} $𝓖$ such that \emph{(i)} the ordered list of wires that touches the upper boundary is typed by $[X₀ , \dots , Xₙ]$, and \emph{(ii)} the ordered list of wires that touches the lower boundary is typed by $[Y₀ , \dots , Yₘ]$.
  \begin{figure}[ht]
    \centering

\tikzset{every picture/.style={line width=0.75pt}} %

\begin{tikzpicture}[x=0.75pt,y=0.75pt,yscale=-1,xscale=1]
\draw    (56,20) -- (56,32) ;
\draw    (56,48) -- (56,60) ;
\draw   (28,32) -- (60,32) -- (60,48) -- (28,48) -- cycle ;
\draw    (32,20) -- (32,32) ;
\draw    (32,48) -- (32,60) ;
\draw    (128,16) -- (128,24) ;
\draw    (128,56) -- (128,64) ;
\draw   (100,24) -- (132,24) -- (132,56) -- (100,56) -- cycle ;
\draw    (104,16) -- (104,24) ;
\draw    (104,56) -- (104,64) ;
\draw    (168,16) -- (168,24) ;
\draw    (168,56) -- (168,64) ;
\draw   (140,24) -- (172,24) -- (172,56) -- (140,56) -- cycle ;
\draw    (144,16) -- (144,24) ;
\draw    (144,56) -- (144,64) ;
\draw    (240,16) -- (240,20) ;
\draw   (212,20) -- (244,20) -- (244,36) -- (212,36) -- cycle ;
\draw    (216,16) -- (216,20) ;
\draw    (240,36) -- (240,44) ;
\draw   (212,44) -- (244,44) -- (244,60) -- (212,60) -- cycle ;
\draw    (216,36) -- (216,44) ;
\draw    (240,60) -- (240,64) ;
\draw    (216,60) -- (216,64) ;
\draw    (304,28) -- (304,52) ;
\draw    (284,28) -- (284,52) ;
\draw  [draw opacity=0] (96,4) -- (180,4) -- (180,80) -- (96,80) -- cycle ;
\draw  [draw opacity=0] (72,20) -- (100,20) -- (100,60) -- (72,60) -- cycle ;
\draw  [draw opacity=0] (208,4) -- (248,4) -- (248,80) -- (208,80) -- cycle ;
\draw  [draw opacity=0] (276,4) -- (316,4) -- (316,80) -- (276,80) -- cycle ;
\draw  [draw opacity=0] (184,20) -- (212,20) -- (212,60) -- (184,60) -- cycle ;
\draw  [draw opacity=0] (256,20) -- (284,20) -- (284,60) -- (256,60) -- cycle ;
\draw  [draw opacity=0] (304,20) -- (332,20) -- (332,60) -- (304,60) -- cycle ;
\draw  [draw opacity=0] (28,88) -- (60,88) -- (60,104) -- (28,104) -- cycle ;
\draw  [draw opacity=0] (116,88) -- (156,88) -- (156,104) -- (116,104) -- cycle ;
\draw  [draw opacity=0] (208,88) -- (248,88) -- (248,104) -- (208,104) -- cycle ;
\draw  [draw opacity=0] (276,88) -- (316,88) -- (316,104) -- (276,104) -- cycle ;

\draw (44,40) node  [font=\footnotesize]  {$f$};
\draw (116,40) node  [font=\footnotesize]  {$\alpha\phantom{'}$};
\draw (156,40) node  [font=\footnotesize]  {$\alpha'$};
\draw (228,28) node  [font=\footnotesize]  {$\alpha$};
\draw (228,52) node  [font=\footnotesize]  {$\beta $};
\draw (32,16.6) node [anchor=south] [inner sep=0.75pt]  [font=\tiny]  {$A_{0}$};
\draw (56,16.6) node [anchor=south] [inner sep=0.75pt]  [font=\tiny]  {$A_{n}$};
\draw (32,63.4) node [anchor=north] [inner sep=0.75pt]  [font=\tiny]  {$B_{0}$};
\draw (56,63.4) node [anchor=north] [inner sep=0.75pt]  [font=\tiny]  {$B_{m}$};
\draw (104,12.6) node [anchor=south] [inner sep=0.75pt]  [font=\tiny]  {$A_{0}$};
\draw (128,12.6) node [anchor=south] [inner sep=0.75pt]  [font=\tiny]  {$A_{n}$};
\draw (144,12.6) node [anchor=south] [inner sep=0.75pt]  [font=\tiny]  {$A'_{0}$};
\draw (168,12.6) node [anchor=south] [inner sep=0.75pt]  [font=\tiny]  {$A'_{n}$};
\draw (104,67.4) node [anchor=north] [inner sep=0.75pt]  [font=\tiny]  {$B_{0}$};
\draw (128,67.4) node [anchor=north] [inner sep=0.75pt]  [font=\tiny]  {$B_{m}$};
\draw (144,67.4) node [anchor=north] [inner sep=0.75pt]  [font=\tiny]  {$B'_{0}$};
\draw (168,67.4) node [anchor=north] [inner sep=0.75pt]  [font=\tiny]  {$B'_{m}$};
\draw (216,67.4) node [anchor=north] [inner sep=0.75pt]  [font=\tiny]  {$B_{0}$};
\draw (240,67.4) node [anchor=north] [inner sep=0.75pt]  [font=\tiny]  {$B_{m}$};
\draw (216,12.6) node [anchor=south] [inner sep=0.75pt]  [font=\tiny]  {$A_{0}$};
\draw (240,12.6) node [anchor=south] [inner sep=0.75pt]  [font=\tiny]  {$A_{n}$};
\draw (304,24.6) node [anchor=south] [inner sep=0.75pt]  [font=\tiny]  {$A_{n}$};
\draw (284,24.6) node [anchor=south] [inner sep=0.75pt]  [font=\tiny]  {$A_{0}$};
\draw (284,55.4) node [anchor=north] [inner sep=0.75pt]  [font=\tiny]  {$A_{0}$};
\draw (304,55.4) node [anchor=north] [inner sep=0.75pt]  [font=\tiny]  {$A_{n}$};
\draw (74,40) node  [font=\footnotesize]  {$;$};
\draw (186,40) node  [font=\footnotesize]  {$;$};
\draw (258,40) node  [font=\footnotesize]  {$;$};
\draw (318,40) node  [font=\footnotesize]  {$;$};
\draw (44,96) node  [font=\footnotesize]  {$f$};
\draw (136,96) node  [font=\footnotesize]  {$\alpha \otimes \alpha'$};
\draw (228,96) node  [font=\footnotesize]  {$\alpha ⨾ \beta $};
\draw (296,96) node  [font=\footnotesize]  {$id$};

\end{tikzpicture}

    \caption{Strict monoidal category of string diagrams.}
    \label{fig:string-diagrams-mon-cat}
  \end{figure}

  \Cref{fig:string-diagrams-mon-cat} describes the operations of the category.
  The parallel composition of two diagrams $α ፡ [X₀, … , Xₙ] → [Y₀, … , Yₘ]$ and $α' ፡ [X'_0, … , X'_{n'}] → [Y'_0, … , Y'_{m'}]$ is their horizontal juxtaposition.
  The sequential composition of two diagrams $α ፡ [X₀, …, Xₙ] → [Y₀, …, Yₘ]$ and $β ፡ [Y₀, …, Yₘ] → [Z₀, …, Zₖ]$ is the diagram obtained by vertical juxtaposition linking the outputs of the first to the inputs of the second.
  The identity on the object $[X₀, …, Xₙ]$ is given by a diagram containing $n$ identity wires labeled by these objects.
\end{proof}

\defining{linkForgetMon}{}
\begin{lem}
  Forgetting about the sequential and parallel composition defines a functor from \monoidalCategories{} to \polygraphs{}, $$\ForgetMon{} ፡ \MonCatStr → \PolyGraph.$$
\end{lem}
\begin{proof}
  Any \monoidalCategory{} $ℂ$ can be seen as a \polygraph{}, $\ForgetMon(ℂ)$, where the edges are determined by the morphisms,
  $$\ForgetMon(ℂ)(A_{0},\dots, A_{n};B_{0}, \dots, B_{m}) = ℂ(A₀ ⊗ … ⊗ Aₙ, B₀ ⊗ … ⊗ Bₘ),$$
  and we forget about composition and tensoring. It can be checked, by its definition, that any strict monoidal functor induces a homomorphism on the underlying \polygraphs{}.
\end{proof}

\begin{thm}[Joyal and Street, {{\cite[Theorem 2.3]{joyal91}}}]
  There exists an adjunction between \polygraphs{} and \strictMonoidalCategories{}, $\StringMon ⊣ \ForgetMon$.
  Given a \polygraph{} $𝓖$, the free \strictMonoidalCategory{}, $\StringMon(𝓖)$, is the \strictMonoidalCategory{} that has as morphisms the \stringDiagrams{} over the generators of the \polygraph{}; the underlying \polygraph{} determines the right adjoint.
\end{thm}

\subsection{String Diagrams for Effectful Categories}
String diagrams \emph{for monoidal categories} arise from the adjunction between \polygraphs{} and monoidal categories. Let us translate this story to the effectful case:
we will construct a similar adjunction between \polygraphCouples{} and \effectfulCategories{}. We start by formally adding the runtime to a free monoidal category.

\begin{defi}[Effectful polygraph]
  \defining{linkpolygraphcouple}{}
  An \emph{\polygraphCouple{}} is a pair of \polygraphs{} $(\hyV,\hyG)$ sharing the same objects, $\hyV_{\mathrm{obj}} = \hyG_{\mathrm{obj}}$.
  A morphism of \polygraphCouples{} $(u,f) \colon (\hyV,\hyG) \to (\hyW,\hyH)$ is a pair of morphisms of \polygraphs{}, $u \colon \hyV \to \hyW$ and $f \colon \hyG \to \hyH$, such that they coincide on objects, $f_{\mathrm{obj}} = u_{\mathrm{obj}}$. Effectful polygraphs form a category, $\EffPolyGraph$.
\end{defi}

\begin{defi}[Runtime monoidal category]
  \defining{linkruntimepolygraph}{}
  Let $(\hyV,\hyG)$ be an \polygraphCouple{}. Its \emph{runtime monoidal
  category}, $\MONRUN(\hyV,\hyG)$, is the monoidal category freely generated
  from adding an extra object---the runtime, $\R$---to the input and output of
  every effectful generator in $\hyG$ (but not to those in $\hyV$), and letting
  that extra object be braided with respect to every other object of the
  category.

  In other words, it is the monoidal category freely generated by the following \polygraph{}, $\Run(\hyV,\hyG)$,
  (\Cref{fig:rungen}),
  assuming $A_{0},\dots,A_{n}$ and $B_{0},\dots,B_{m}$ are distinct from $\R$,
  \begin{itemize}
    \item $\obj{\Run(\hyV,\hyG)} = \obj{\hyG} + \{ \R \} = \obj{\hyV} + \{ \R \}$,
    \item $\Run(\hyV,\hyG)(\R,A_{0},\dots,A_{n};\R, B_{0},\dots,B_{n}) = \hyG(A_{0},\dots,A_{n}; B_{0},\dots,B_{n})$,
    \item $\Run(\hyV,\hyG)(A_{0},\dots,A_{n}; B_{0},\dots,B_{n}) = \hyV(A_{0},\dots,A_{n}; B_{0},\dots,B_{n})$,
    \item $\Run(\hyV,\hyG)(\R,A_{0};A_{0},\R) = \Run(\hyV,\hyG)(A_{0},\R;\R,A_{0}) = \{\sigma\}$,
  \end{itemize}
  with $\Run(\hyV,\hyG)$ empty in any other case, and quotiented by the braiding axioms for $\R$ (\Cref{fig:runaxiom}).
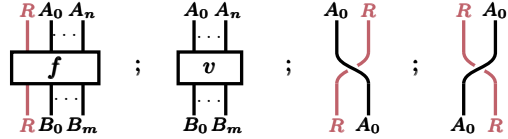
  \begin{figure}[H]
    \centering
    \vspace{-1ex}

\tikzset{every picture/.style={line width=0.75pt}} %

\begin{tikzpicture}[x=0.75pt,y=0.75pt,yscale=-1,xscale=1]
\draw    (28,16) -- (28,32) ;
\draw    (44,16) -- (44,32) ;
\draw [color={rgb, 255:red, 191; green, 97; blue, 106 }  ,draw opacity=1 ]   (16,16) -- (16,32) ;
\draw [color={rgb, 255:red, 191; green, 97; blue, 106 }  ,draw opacity=1 ]   (16,48) -- (16,64) ;
\draw    (28,48) -- (28,64) ;
\draw    (44,48) -- (44,64) ;
\draw    (100,16) -- (100,32) ;
\draw    (116,16) -- (116,32) ;
\draw    (100,48) -- (100,64) ;
\draw    (116,48) -- (116,64) ;
\draw    (172,16) -- (172,32) ;
\draw [color={rgb, 255:red, 191; green, 97; blue, 106 }  ,draw opacity=1 ]   (188,16) -- (188,32) ;
\draw [color={rgb, 255:red, 191; green, 97; blue, 106 }  ,draw opacity=1 ]   (172,48) -- (172,64) ;
\draw    (188,48) -- (188,64) ;
\draw [color={rgb, 255:red, 191; green, 97; blue, 106 }  ,draw opacity=1 ]   (188,32) .. controls (187.72,40.14) and (172,40.43) .. (172,48) ;
\draw  [color={rgb, 255:red, 255; green, 255; blue, 255 }  ,draw opacity=1 ][fill={rgb, 255:red, 255; green, 255; blue, 255 }  ,fill opacity=1 ] (178,40.25) .. controls (178,39.15) and (178.9,38.25) .. (180,38.25) .. controls (181.11,38.25) and (182,39.15) .. (182,40.25) .. controls (182,41.35) and (181.11,42.25) .. (180,42.25) .. controls (178.9,42.25) and (178,41.35) .. (178,40.25) -- cycle ;
\draw [color={rgb, 255:red, 0; green, 0; blue, 0 }  ,draw opacity=1 ]   (172,32) .. controls (171.72,40.14) and (188,40.43) .. (188,48) ;
\draw    (252,16) -- (252,32) ;
\draw [color={rgb, 255:red, 191; green, 97; blue, 106 }  ,draw opacity=1 ]   (236,16) -- (236,32) ;
\draw [color={rgb, 255:red, 191; green, 97; blue, 106 }  ,draw opacity=1 ]   (252,48) -- (252,64) ;
\draw    (236,48) -- (236,64) ;
\draw [color={rgb, 255:red, 191; green, 97; blue, 106 }  ,draw opacity=1 ]   (236,32) .. controls (236.29,40.14) and (252,40.43) .. (252,48) ;
\draw  [color={rgb, 255:red, 255; green, 255; blue, 255 }  ,draw opacity=1 ][fill={rgb, 255:red, 255; green, 255; blue, 255 }  ,fill opacity=1 ] (246,40.25) .. controls (246,39.15) and (245.1,38.25) .. (244,38.25) .. controls (242.9,38.25) and (242,39.15) .. (242,40.25) .. controls (242,41.35) and (242.9,42.25) .. (244,42.25) .. controls (245.1,42.25) and (246,41.35) .. (246,40.25) -- cycle ;
\draw [color={rgb, 255:red, 0; green, 0; blue, 0 }  ,draw opacity=1 ]   (252,32) .. controls (252.28,40.14) and (236,40.43) .. (236,48) ;
\draw   (8,32) -- (52,32) -- (52,48) -- (8,48) -- cycle ;
\draw   (92,32) -- (124,32) -- (124,48) -- (92,48) -- cycle ;
\draw  [draw opacity=0] (12,6) -- (20,6) -- (20,16) -- (12,16) -- cycle ;
\draw  [draw opacity=0] (24,6) -- (32,6) -- (32,16) -- (24,16) -- cycle ;
\draw  [draw opacity=0] (40,6) -- (48,6) -- (48,16) -- (40,16) -- cycle ;
\draw  [draw opacity=0] (12,64) -- (20,64) -- (20,74) -- (12,74) -- cycle ;
\draw  [draw opacity=0] (24,64) -- (32,64) -- (32,74) -- (24,74) -- cycle ;
\draw  [draw opacity=0] (40,64) -- (48,64) -- (48,74) -- (40,74) -- cycle ;
\draw  [draw opacity=0] (28,48) -- (48,48) -- (48,62) -- (28,62) -- cycle ;
\draw  [draw opacity=0] (104,82) -- (124,82) -- (124,96) -- (104,96) -- cycle ;
\draw  [draw opacity=0] (100,16) -- (120,16) -- (120,30) -- (100,30) -- cycle ;
\draw  [draw opacity=0] (28,16) -- (48,16) -- (48,30) -- (28,30) -- cycle ;
\draw  [draw opacity=0] (97.5,6) -- (105.5,6) -- (105.5,16) -- (97.5,16) -- cycle ;
\draw  [draw opacity=0] (113.5,6) -- (121.5,6) -- (121.5,16) -- (113.5,16) -- cycle ;
\draw  [draw opacity=0] (97,64) -- (105,64) -- (105,74) -- (97,74) -- cycle ;
\draw  [draw opacity=0] (113,64) -- (121,64) -- (121,74) -- (113,74) -- cycle ;
\draw  [draw opacity=0] (164,6) -- (180,6) -- (180,16) -- (164,16) -- cycle ;
\draw  [draw opacity=0] (180,6) -- (196,6) -- (196,16) -- (180,16) -- cycle ;
\draw  [draw opacity=0] (228,64) -- (244,64) -- (244,74) -- (228,74) -- cycle ;
\draw  [draw opacity=0] (244,64) -- (260,64) -- (260,74) -- (244,74) -- cycle ;
\draw  [draw opacity=0] (180,64) -- (196,64) -- (196,74) -- (180,74) -- cycle ;
\draw  [draw opacity=0] (164,64) -- (180,64) -- (180,74) -- (164,74) -- cycle ;
\draw  [draw opacity=0] (244,6) -- (260,6) -- (260,16) -- (244,16) -- cycle ;
\draw  [draw opacity=0] (228,6) -- (244,6) -- (244,16) -- (228,16) -- cycle ;

\draw (30,40) node  [font=\footnotesize]  {$f$};
\draw (72,40) node  [font=\footnotesize]  {$;$};
\draw (108,40) node  [font=\footnotesize]  {$v$};
\draw (148,40) node  [font=\footnotesize]  {$;$};
\draw (212,40) node  [font=\footnotesize]  {$;$};
\draw (16,11) node  [font=\tiny,color={rgb, 255:red, 191; green, 97; blue, 106 }  ,opacity=1 ]  {$R$};
\draw (28,11) node  [font=\tiny]  {$A_{0}$};
\draw (44,11) node  [font=\tiny]  {$A_{n}$};
\draw (16,69) node  [font=\tiny,color={rgb, 255:red, 191; green, 97; blue, 106 }  ,opacity=1 ]  {$R$};
\draw (28,69) node  [font=\tiny]  {$B_{0}$};
\draw (44,69) node  [font=\tiny]  {$B_{m}$};
\draw (36.5,55.5) node  [font=\tiny]  {$\dotsc $};
\draw (108.5,56.5) node  [font=\tiny]  {$\dotsc $};
\draw (108.5,24.5) node  [font=\tiny]  {$\dotsc $};
\draw (35.5,23.5) node  [font=\tiny]  {$\dotsc $};
\draw (101.5,11) node  [font=\tiny]  {$A_{0}$};
\draw (117.5,11) node  [font=\tiny]  {$A_{n}$};
\draw (101,69) node  [font=\tiny]  {$B_{0}$};
\draw (117,69) node  [font=\tiny]  {$B_{m}$};
\draw (172,11) node  [font=\tiny]  {$A_{0}$};
\draw (188,11) node  [font=\tiny,color={rgb, 255:red, 191; green, 97; blue, 106 }  ,opacity=1 ]  {$R$};
\draw (236,69) node  [font=\tiny]  {$A_{0}$};
\draw (252,69) node  [font=\tiny,color={rgb, 255:red, 191; green, 97; blue, 106 }  ,opacity=1 ]  {$R$};
\draw (188,69) node  [font=\tiny]  {$A_{0}$};
\draw (172,69) node  [font=\tiny,color={rgb, 255:red, 191; green, 97; blue, 106 }  ,opacity=1 ]  {$R$};
\draw (252,11) node  [font=\tiny]  {$A_{0}$};
\draw (236,11) node  [font=\tiny,color={rgb, 255:red, 191; green, 97; blue, 106 }  ,opacity=1 ]  {$R$};
\end{tikzpicture} %
\newline{}
    For each $f \in \hyG(A_0,\dots,A_n;B_0,\dots,B_m)$ and each $v \in \hyV(A_0,\dots,A_n;B_0,\dots,B_m)$.
    \caption{Generators for the runtime monoidal category.}
    \label{fig:rungen}
  \end{figure}
  \begin{figure}[H]
    \centering

\tikzset{every picture/.style={line width=0.75pt}} %

\begin{tikzpicture}[x=0.75pt,y=0.75pt,yscale=-1,xscale=1]
\draw    (212,12) -- (212,20) ;
\draw    (228,12) -- (228,20) ;
\draw    (212,36) -- (212,44) ;
\draw    (228,36) -- (228,44) ;
\draw   (204,20) -- (236,20) -- (236,36) -- (204,36) -- cycle ;
\draw  [draw opacity=0] (16,12) -- (44,12) -- (44,60) -- (16,60) -- cycle ;
\draw    (20,12) -- (20,20) ;
\draw [color={rgb, 255:red, 191; green, 97; blue, 106 }  ,draw opacity=1 ]   (4,12) -- (4,20) ;
\draw [color={rgb, 255:red, 191; green, 97; blue, 106 }  ,draw opacity=1 ]   (4,20) .. controls (4.29,28.14) and (20,28.43) .. (20,36) ;
\draw  [color={rgb, 255:red, 255; green, 255; blue, 255 }  ,draw opacity=1 ][fill={rgb, 255:red, 255; green, 255; blue, 255 }  ,fill opacity=1 ] (14,28.25) .. controls (14,27.15) and (13.1,26.25) .. (12,26.25) .. controls (10.9,26.25) and (10,27.15) .. (10,28.25) .. controls (10,29.35) and (10.9,30.25) .. (12,30.25) .. controls (13.1,30.25) and (14,29.35) .. (14,28.25) -- cycle ;
\draw [color={rgb, 255:red, 0; green, 0; blue, 0 }  ,draw opacity=1 ]   (20,20) .. controls (20.28,28.14) and (4,28.43) .. (4,36) ;
\draw [color={rgb, 255:red, 191; green, 97; blue, 106 }  ,draw opacity=1 ]   (4,52) -- (4,60) ;
\draw    (20,52) -- (20,60) ;
\draw [color={rgb, 255:red, 191; green, 97; blue, 106 }  ,draw opacity=1 ]   (20,36) .. controls (19.71,44.14) and (4,44.43) .. (4,52) ;
\draw  [color={rgb, 255:red, 255; green, 255; blue, 255 }  ,draw opacity=1 ][fill={rgb, 255:red, 255; green, 255; blue, 255 }  ,fill opacity=1 ] (10,44.25) .. controls (10,43.15) and (10.9,42.25) .. (12,42.25) .. controls (13.1,42.25) and (14,43.15) .. (14,44.25) .. controls (14,45.35) and (13.1,46.25) .. (12,46.25) .. controls (10.9,46.25) and (10,45.35) .. (10,44.25) -- cycle ;
\draw [color={rgb, 255:red, 0; green, 0; blue, 0 }  ,draw opacity=1 ]   (4,36) .. controls (3.72,44.14) and (20,44.43) .. (20,52) ;
\draw    (56,12) -- (56,60) ;
\draw [color={rgb, 255:red, 191; green, 97; blue, 106 }  ,draw opacity=1 ]   (40,12) -- (40,60) ;
\draw    (100,12) -- (100.01,20) ;
\draw [color={rgb, 255:red, 191; green, 97; blue, 106 }  ,draw opacity=1 ]   (116,12) -- (116,20) ;
\draw [color={rgb, 255:red, 191; green, 97; blue, 106 }  ,draw opacity=1 ]   (116,20) .. controls (115.72,28.14) and (100.01,28.43) .. (100.01,36) ;
\draw  [color={rgb, 255:red, 255; green, 255; blue, 255 }  ,draw opacity=1 ][fill={rgb, 255:red, 255; green, 255; blue, 255 }  ,fill opacity=1 ] (106.01,28.25) .. controls (106.01,27.15) and (106.9,26.25) .. (108,26.25) .. controls (109.11,26.25) and (110,27.15) .. (110,28.25) .. controls (110,29.35) and (109.11,30.25) .. (108,30.25) .. controls (106.9,30.25) and (106.01,29.35) .. (106.01,28.25) -- cycle ;
\draw [color={rgb, 255:red, 0; green, 0; blue, 0 }  ,draw opacity=1 ]   (100.01,20) .. controls (99.72,28.14) and (116,28.43) .. (116,36) ;
\draw [color={rgb, 255:red, 191; green, 97; blue, 106 }  ,draw opacity=1 ]   (116,52) -- (116,60) ;
\draw    (100.01,52) -- (100,60) ;
\draw [color={rgb, 255:red, 191; green, 97; blue, 106 }  ,draw opacity=1 ]   (100.01,36) .. controls (100.29,44.14) and (116,44.43) .. (116,52) ;
\draw  [color={rgb, 255:red, 255; green, 255; blue, 255 }  ,draw opacity=1 ][fill={rgb, 255:red, 255; green, 255; blue, 255 }  ,fill opacity=1 ] (110,44.25) .. controls (110,43.15) and (109.11,42.25) .. (108,42.25) .. controls (106.9,42.25) and (106.01,43.15) .. (106.01,44.25) .. controls (106.01,45.35) and (106.9,46.25) .. (108,46.25) .. controls (109.11,46.25) and (110,45.35) .. (110,44.25) -- cycle ;
\draw [color={rgb, 255:red, 191; green, 97; blue, 106 }  ,draw opacity=1 ]   (152,12) -- (152,60) ;
\draw [color={rgb, 255:red, 0; green, 0; blue, 0 }  ,draw opacity=1 ]   (136,12) -- (136,60) ;
\draw [color={rgb, 255:red, 191; green, 97; blue, 106 }  ,draw opacity=1 ]   (196,12) -- (196,36) ;
\draw [color={rgb, 255:red, 191; green, 97; blue, 106 }  ,draw opacity=1 ]   (196,36) .. controls (196.2,52.4) and (228.2,44.4) .. (228,60) ;
\draw  [color={rgb, 255:red, 255; green, 255; blue, 255 }  ,draw opacity=1 ][fill={rgb, 255:red, 255; green, 255; blue, 255 }  ,fill opacity=1 ] (212.37,48.13) .. controls (212.37,47.02) and (211.48,46.13) .. (210.38,46.13) .. controls (209.27,46.13) and (208.38,47.02) .. (208.38,48.13) .. controls (208.38,49.23) and (209.27,50.13) .. (210.38,50.13) .. controls (211.48,50.13) and (212.37,49.23) .. (212.37,48.13) -- cycle ;
\draw  [color={rgb, 255:red, 255; green, 255; blue, 255 }  ,draw opacity=1 ][fill={rgb, 255:red, 255; green, 255; blue, 255 }  ,fill opacity=1 ] (223.62,51.88) .. controls (223.62,50.77) and (222.73,49.88) .. (221.63,49.88) .. controls (220.52,49.88) and (219.63,50.77) .. (219.63,51.88) .. controls (219.63,52.98) and (220.52,53.88) .. (221.63,53.88) .. controls (222.73,53.88) and (223.62,52.98) .. (223.62,51.88) -- cycle ;
\draw [color={rgb, 255:red, 0; green, 0; blue, 0 }  ,draw opacity=1 ]   (228,44) .. controls (228.29,52.14) and (212.01,52.43) .. (212.01,60) ;
\draw    (263.99,52) -- (263.99,60) ;
\draw    (279.99,52) -- (280,60) ;
\draw    (263.99,28) -- (264,36) ;
\draw    (280,28) -- (280,36) ;
\draw   (255.99,36) -- (288,36) -- (288,52) -- (255.99,52) -- cycle ;
\draw [color={rgb, 255:red, 191; green, 97; blue, 106 }  ,draw opacity=1 ]   (296,36) -- (296,60) ;
\draw [color={rgb, 255:red, 191; green, 97; blue, 106 }  ,draw opacity=1 ]   (264,12) .. controls (264.2,28.4) and (296.2,20.4) .. (296,36) ;
\draw  [color={rgb, 255:red, 255; green, 255; blue, 255 }  ,draw opacity=1 ][fill={rgb, 255:red, 255; green, 255; blue, 255 }  ,fill opacity=1 ] (272.56,20.86) .. controls (272.56,19.75) and (271.67,18.86) .. (270.57,18.86) .. controls (269.46,18.86) and (268.57,19.75) .. (268.57,20.86) .. controls (268.57,21.96) and (269.46,22.86) .. (270.57,22.86) .. controls (271.67,22.86) and (272.56,21.96) .. (272.56,20.86) -- cycle ;
\draw  [color={rgb, 255:red, 255; green, 255; blue, 255 }  ,draw opacity=1 ][fill={rgb, 255:red, 255; green, 255; blue, 255 }  ,fill opacity=1 ] (282.85,24.29) .. controls (282.85,23.18) and (281.96,22.29) .. (280.85,22.29) .. controls (279.75,22.29) and (278.85,23.18) .. (278.85,24.29) .. controls (278.85,25.39) and (279.75,26.29) .. (280.85,26.29) .. controls (281.96,26.29) and (282.85,25.39) .. (282.85,24.29) -- cycle ;
\draw [color={rgb, 255:red, 0; green, 0; blue, 0 }  ,draw opacity=1 ]   (295.99,12) .. controls (296.27,20.14) and (280,20.43) .. (280,28) ;
\draw [color={rgb, 255:red, 0; green, 0; blue, 0 }  ,draw opacity=1 ]   (279.99,12) .. controls (280.27,20.14) and (264,20.43) .. (264,28) ;
\draw [color={rgb, 255:red, 0; green, 0; blue, 0 }  ,draw opacity=1 ]   (116,36) .. controls (116.29,44.14) and (100.01,44.43) .. (100.01,52) ;
\draw    (364,12) -- (364,20) ;
\draw    (348,12) -- (348,20) ;
\draw    (364,36) -- (364,44) ;
\draw    (348,36) -- (348,44) ;
\draw   (372,20) -- (340,20) -- (340,36) -- (372,36) -- cycle ;
\draw [color={rgb, 255:red, 191; green, 97; blue, 106 }  ,draw opacity=1 ]   (380,12) -- (380,36) ;
\draw [color={rgb, 255:red, 191; green, 97; blue, 106 }  ,draw opacity=1 ]   (380,36) .. controls (379.8,52.4) and (347.8,44.4) .. (348,60) ;
\draw  [color={rgb, 255:red, 255; green, 255; blue, 255 }  ,draw opacity=1 ][fill={rgb, 255:red, 255; green, 255; blue, 255 }  ,fill opacity=1 ] (363.63,48.13) .. controls (363.63,47.02) and (364.52,46.13) .. (365.62,46.13) .. controls (366.73,46.13) and (367.62,47.02) .. (367.62,48.13) .. controls (367.62,49.23) and (366.73,50.13) .. (365.62,50.13) .. controls (364.52,50.13) and (363.63,49.23) .. (363.63,48.13) -- cycle ;
\draw  [color={rgb, 255:red, 255; green, 255; blue, 255 }  ,draw opacity=1 ][fill={rgb, 255:red, 255; green, 255; blue, 255 }  ,fill opacity=1 ] (352.38,51.88) .. controls (352.38,50.77) and (353.27,49.88) .. (354.37,49.88) .. controls (355.48,49.88) and (356.37,50.77) .. (356.37,51.88) .. controls (356.37,52.98) and (355.48,53.88) .. (354.37,53.88) .. controls (353.27,53.88) and (352.38,52.98) .. (352.38,51.88) -- cycle ;
\draw [color={rgb, 255:red, 0; green, 0; blue, 0 }  ,draw opacity=1 ]   (348,44) .. controls (347.71,52.14) and (363.99,52.43) .. (363.99,60) ;
\draw [color={rgb, 255:red, 0; green, 0; blue, 0 }  ,draw opacity=1 ]   (364,44) .. controls (363.71,52.14) and (379.99,52.43) .. (379.99,60) ;
\draw    (432,52) -- (432,60) ;
\draw    (416,52) -- (416,60) ;
\draw    (432,28) -- (431.99,36) ;
\draw    (416,28) -- (416,36) ;
\draw   (440,36) -- (408,36) -- (408,52) -- (440,52) -- cycle ;
\draw [color={rgb, 255:red, 191; green, 97; blue, 106 }  ,draw opacity=1 ]   (400,36) -- (400,60) ;
\draw [color={rgb, 255:red, 191; green, 97; blue, 106 }  ,draw opacity=1 ]   (431.99,12) .. controls (431.79,28.4) and (399.8,20.4) .. (400,36) ;
\draw  [color={rgb, 255:red, 255; green, 255; blue, 255 }  ,draw opacity=1 ][fill={rgb, 255:red, 255; green, 255; blue, 255 }  ,fill opacity=1 ] (423.43,20.86) .. controls (423.43,19.75) and (424.32,18.86) .. (425.42,18.86) .. controls (426.53,18.86) and (427.42,19.75) .. (427.42,20.86) .. controls (427.42,21.96) and (426.53,22.86) .. (425.42,22.86) .. controls (424.32,22.86) and (423.43,21.96) .. (423.43,20.86) -- cycle ;
\draw  [color={rgb, 255:red, 255; green, 255; blue, 255 }  ,draw opacity=1 ][fill={rgb, 255:red, 255; green, 255; blue, 255 }  ,fill opacity=1 ] (413.14,24.29) .. controls (413.14,23.18) and (414.04,22.29) .. (415.14,22.29) .. controls (416.24,22.29) and (417.14,23.18) .. (417.14,24.29) .. controls (417.14,25.39) and (416.24,26.29) .. (415.14,26.29) .. controls (414.04,26.29) and (413.14,25.39) .. (413.14,24.29) -- cycle ;
\draw [color={rgb, 255:red, 0; green, 0; blue, 0 }  ,draw opacity=1 ]   (400.01,12) .. controls (399.72,20.14) and (416,20.43) .. (416,28) ;
\draw [color={rgb, 255:red, 0; green, 0; blue, 0 }  ,draw opacity=1 ]   (416,12) .. controls (415.72,20.14) and (431.99,20.43) .. (431.99,28) ;
\draw [color={rgb, 255:red, 0; green, 0; blue, 0 }  ,draw opacity=1 ]   (212,44) .. controls (212.29,52.14) and (196.01,52.43) .. (196.01,60) ;
\draw  [draw opacity=0] (112,12) -- (140,12) -- (140,60) -- (112,60) -- cycle ;
\draw  [draw opacity=0] (232,12) -- (260,12) -- (260,60) -- (232,60) -- cycle ;
\draw  [draw opacity=0] (376,12) -- (404,12) -- (404,60) -- (376,60) -- cycle ;
\draw  [draw opacity=0] (64,12) -- (92,12) -- (92,60) -- (64,60) -- cycle ;
\draw  [draw opacity=0] (160,12) -- (188,12) -- (188,60) -- (160,60) -- cycle ;
\draw  [draw opacity=0] (304,12) -- (332,12) -- (332,60) -- (304,60) -- cycle ;
\draw  [draw opacity=0] (448,12) -- (476,12) -- (476,60) -- (448,60) -- cycle ;

\draw (220,28) node  [font=\footnotesize]  {$v$};
\draw (220.5,40.5) node  [font=\tiny]  {$\dots$};
\draw (220.5,12.5) node  [font=\tiny]  {$\dots$};
\draw (271.99,44) node  [font=\footnotesize]  {$v$};
\draw (272.5,28.5) node  [font=\tiny]  {$\dots$};
\draw (272.49,52.5) node  [font=\tiny]  {$\dots$};
\draw (355.5,40.5) node  [font=\tiny,xscale=-1]  {$\dots$};
\draw (355.5,12.5) node  [font=\tiny,xscale=-1]  {$\dots$};
\draw (356,28) node  [font=\footnotesize]  {$v$};
\draw (423.5,28.5) node  [font=\tiny,xscale=-1]  {$\dots$};
\draw (423.5,52.5) node  [font=\tiny,xscale=-1]  {$\dots$};
\draw (424,44) node  [font=\footnotesize]  {$v$};
\draw (30,36) node  [font=\footnotesize]  {$\pmb{=}$};
\draw (126,36) node  [font=\footnotesize]  {$\pmb{=}$};
\draw (246,36) node  [font=\footnotesize]  {$\pmb{=}$};
\draw (390,36) node  [font=\footnotesize]  {$\pmb{=}$};
\draw (78,36) node  [font=\footnotesize]  {$\pmb{;}$};
\draw (174,36) node  [font=\footnotesize]  {$\pmb{;}$};
\draw (318,36) node  [font=\footnotesize]  {$\pmb{;}$};
\draw (462,36) node  [font=\footnotesize]  {$\pmb{;}$};

\end{tikzpicture}
     \caption{Axioms for the runtime monoidal category.}
    \label{fig:runaxiom}
  \end{figure}
\end{defi}

Some readers may prefer to understand this as asking the runtime $\R$ to be in
the Drinfeld centre \cite{drinfeld10} of the monoidal category.\footnote{The
Drinfeld centre of a monoidal category $(ℂ,⊗,I)$ is braided monoidal category
that has as objects the pairs $(R,σ)$, where $σ ፡ R ⊗ • → • ⊗ R$ is a natural
isomorphism satisfying the braiding axioms; morphisms in the Drinfeld centre are
those that preserve these braidings \cite{drinfeld10}.} The extra wire that $\R$
provides is only used to prevent interchange, and so it does not really matter
where it is placed in the input and the output. We can choose to always place it
on the left, for instance---and indeed we will be able to do so---but a better
solution is to just consider objects ``up to some runtime braidings''. This is
formalized by the notion of \emph{braid clique}.

\begin{defi}[Braid clique]
  Given any list of objects $A_{0},\dots,A_{n}$ in $\obj{\hyV} = \obj{\hyG}$,
  we construct a \emph{clique}  \cite{trimble:coherence,shulman20182} in the category $\MONRUN(\hyV,\hyG)$: we consider the objects, $A_{0} \tensor \dots \tensor \R_{(i)} \tensor \dots \tensor A_{n}$, created by inserting the runtime $\R$ in all of the possible $0 \leq i \leq n+1$ positions; and we consider the family of commuting isomorphisms constructed by braiding the runtime,
  \[ \sigma_{i,j} \colon A_{0} \tensor \dots \tensor \R_{(i)} \tensor \dots \tensor A_{n} \to  A_{0} \tensor \dots \tensor \R_{(j)} \tensor \dots \tensor A_{n}.\]
  We call this family of isomorphisms the \emph{braid clique}, $\Braid(A_{0},\dots,A_{n})$, on that list.
\end{defi}

\begin{defi}
  A \emph{braid clique morphism}, $$f \colon \Braid(A_{0},\dots,A_{n}) \to \Braid(B_{0},\dots,B_{m}),$$ is a family of morphisms in the runtime monoidal category, $\MONRUN(\hyV,\hyG)$, from each of the objects of first clique to each of the objects of the second clique,
  \[f_{ik} \colon A_{0} \tensor \dots \tensor \R_{(i)} \tensor \dots \tensor A_{n} \to
    B_{0} \tensor \dots \tensor \R_{(k)} \tensor \dots \tensor B_{m},\]
  that moreover commutes with all braiding isomorphisms, $f_{ij} \comp \sigma_{jk} = \sigma_{il} \comp f_{lk}$.
\end{defi}

A braid clique morphism  $f \colon \Braid(A_{0},\dots,A_{n}) \to \Braid(B_{0},\dots,B_{m})$ is fully determined by \emph{any} of its components: any component can be obtained uniquely from any other, simply by pre/post-composing it with braidings.
In particular, a braid clique morphism is always fully determined by its leftmost component, $$f_{00} \colon \R \tensor A_{0} \tensor \dots \tensor A_{n} \to \R \tensor B_{0} \tensor \dots \tensor B_{m}.$$

For the following two lemmas, we choose to always deal with the leftmost component of the braid clique morphism.
Given any clique $\Braid(A_{0},\dots,A_{n})$ we let $A = \tensordot{A}{n}$; so
clique morphisms $\Braid(A_{0},\dots,A_{n}) \to \Braid(B_{0},\dots,B_{m})$ are represented by morphisms $\R \tensor A \to \R \tensor B$.

We now use braid morphisms to construct a \premonoidalCategory{}. Braid cliques are the objects
of a category; however, it is interesting to note they it is \emph{not} a monoidal
category. Indeed, the tensor of morphisms cannot be defined in the naive way: a morphism $\R
\tensor A \to \R \tensor B$ cannot be tensored with a morphism $\R \tensor A'
\to \R \tensor B'$ to obtain a morphism $\R \tensor A \tensor A' \to \R \tensor
B \tensor B'$ (as the runtime would appear twice); they only form a premonoidal category.

\begin{lem}\label{ax:lemma:eff-is-premonoidal}\label{lemma:eff-is-premonoidal}
  Let $(\hyV,\hyG)$ be an \polygraphCouple{}.
  There exists a \premonoidalCategory{}, $\EFF(\hyV,\hyG)$, that has as objects the braid cliques, $\Braid(A_{0},\dots,A_{n})$, in $\MONRUN(\hyV,\hyG)$, and as morphisms the braid clique morphisms between them.
\end{lem}
\begin{proof}
  Let us first give $\EFF(\hyV,\hyG)$ category structure.
  The identity on $\Braid(A_{0},\dots,A_{n})$ is the identity on $\R \tensor A$.
  The composition of a morphism $\R  \tensor A \to \R \tensor B$ with a morphism $\R \tensor B \to \R \tensor C$ is their plain composition as string diagrams, in $\MONRUN(\hyV,\hyG)$.

  Let us now check that it is moreover a \premonoidalCategory{}. Tensoring of braid cliques is given by concatenation of lists, which coincides with the tensor of objects in $\MONRUN(\hyV,\hyG)$.
  Whiskering of a morphism $f \colon \R  \tensor A \to \R \tensor B$ is defined with braidings in the left case, $\R \tensor C \tensor A \to \R \tensor C \tensor B$, and by plain whiskering in the right case, $\R \tensor A \tensor C \to \R \tensor B \tensor C$, as depicted in \Cref{fig:whiskering}.
   \begin{figure}[H]
    \centering

\tikzset{every picture/.style={line width=0.75pt}} %

\begin{tikzpicture}[x=0.75pt,y=0.75pt,yscale=-1,xscale=1]
\draw    (60,20) -- (60,36) ;
\draw    (76,20) -- (76,36) ;
\draw [color={rgb, 255:red, 0; green, 0; blue, 0 }  ,draw opacity=1 ]   (44.01,8) -- (44.01,20) ;
\draw    (60,52) -- (60,80) ;
\draw    (76,52) -- (76,80) ;
\draw  [draw opacity=0] (60,20) -- (80,20) -- (80,34) -- (60,34) -- cycle ;
\draw [color={rgb, 255:red, 191; green, 97; blue, 106 }  ,draw opacity=1 ]   (20.01,20) .. controls (20.21,32.4) and (47.8,23.2) .. (48,36) ;
\draw [color={rgb, 255:red, 0; green, 0; blue, 0 }  ,draw opacity=1 ]   (32.01,8) -- (32.01,20) ;
\draw [color={rgb, 255:red, 191; green, 97; blue, 106 }  ,draw opacity=1 ]   (20.01,8) -- (20.01,20) ;
\draw [color={rgb, 255:red, 0; green, 0; blue, 0 }  ,draw opacity=1 ]   (32.01,20) .. controls (32.21,32.4) and (19.81,23.2) .. (20.01,36) ;
\draw [color={rgb, 255:red, 0; green, 0; blue, 0 }  ,draw opacity=1 ]   (44.01,20) .. controls (44.21,32.4) and (31.81,23.2) .. (32.01,36) ;
\draw [color={rgb, 255:red, 0; green, 0; blue, 0 }  ,draw opacity=1 ]   (76,8) -- (76,20) ;
\draw [color={rgb, 255:red, 0; green, 0; blue, 0 }  ,draw opacity=1 ]   (60,8) -- (60,20) ;
\draw [color={rgb, 255:red, 0; green, 0; blue, 0 }  ,draw opacity=1 ]   (32,36) -- (32,52) ;
\draw [color={rgb, 255:red, 0; green, 0; blue, 0 }  ,draw opacity=1 ]   (20,36) -- (20,52) ;
\draw [color={rgb, 255:red, 191; green, 97; blue, 106 }  ,draw opacity=1 ]   (48,52) .. controls (48.2,64.4) and (19.8,55.2) .. (20,68) ;
\draw [color={rgb, 255:red, 0; green, 0; blue, 0 }  ,draw opacity=1 ]   (20,52) .. controls (20.2,64.4) and (31.8,55.2) .. (32,68) ;
\draw [color={rgb, 255:red, 0; green, 0; blue, 0 }  ,draw opacity=1 ]   (32,52) .. controls (32.2,64.4) and (43.8,55.2) .. (44,68) ;
\draw   (40,36) -- (84,36) -- (84,52) -- (40,52) -- cycle ;
\draw [color={rgb, 255:red, 0; green, 0; blue, 0 }  ,draw opacity=1 ]   (44,68) -- (44,80) ;
\draw [color={rgb, 255:red, 0; green, 0; blue, 0 }  ,draw opacity=1 ]   (32,68) -- (32,80) ;
\draw [color={rgb, 255:red, 191; green, 97; blue, 106 }  ,draw opacity=1 ]   (20,68) -- (20,80) ;
\draw    (148,20) -- (148,36) ;
\draw    (164,20) -- (164,36) ;
\draw    (148,52) -- (148,80) ;
\draw    (164,52) -- (164,80) ;
\draw  [draw opacity=0] (148,20) -- (168,20) -- (168,34) -- (148,34) -- cycle ;
\draw [color={rgb, 255:red, 191; green, 97; blue, 106 }  ,draw opacity=1 ]   (136,8) -- (136,36) ;
\draw [color={rgb, 255:red, 0; green, 0; blue, 0 }  ,draw opacity=1 ]   (164,8) -- (164,20) ;
\draw [color={rgb, 255:red, 0; green, 0; blue, 0 }  ,draw opacity=1 ]   (148,8) -- (148,20) ;
\draw   (128,36) -- (172,36) -- (172,52) -- (128,52) -- cycle ;
\draw [color={rgb, 255:red, 191; green, 97; blue, 106 }  ,draw opacity=1 ]   (136,52) -- (136,80) ;
\draw [color={rgb, 255:red, 0; green, 0; blue, 0 }  ,draw opacity=1 ]   (192,8) -- (192,80) ;
\draw [color={rgb, 255:red, 0; green, 0; blue, 0 }  ,draw opacity=1 ]   (180,8) -- (180,80) ;
\draw  [draw opacity=0] (92,20) -- (120,20) -- (120,68) -- (92,68) -- cycle ;

\draw (62,44) node  [font=\footnotesize]  {$f$};
\draw (68.5,68.5) node  [font=\tiny]  {$\dots$};
\draw (38.5,68.5) node  [font=\tiny]  {$\dots$};
\draw (37.5,18.5) node  [font=\tiny]  {$\dots$};
\draw (150,44) node  [font=\footnotesize]  {$f$};
\draw (156.5,69.5) node  [font=\tiny]  {$\dots$};

\draw (186.5,69.5) node  [font=\tiny]  {$\dots$};

\draw (106,44) node  [font=\footnotesize]  {$\pmb{;}$};

\draw (67.5,18.5) node  [font=\tiny]  {$\dots$};
\draw (156.5,18.5) node  [font=\tiny]  {$\dots$};
\draw (186.5,18.5) node  [font=\tiny]  {$\dots$};

\end{tikzpicture}
     \caption{Whiskering in the runtime premonoidal category.}
    \label{fig:whiskering}
  \end{figure}

  Finally, the associators and unitors are identities, which are natural and central.
\end{proof}

\begin{lem}
  \label{ax:lemma:identity-mon-eff}\label{lemma:identity-mon-eff}
  Let $(\hyV,\hyG)$ be an \polygraphCouple{}.
  There exists an identity-on-objects functor \(\MON(\hyV) \to \EFF(\hyV,\hyG)\) that strictly preserves the premonoidal structure and whose image is central. This determines an \effectfulCategory{}.
\end{lem}
\begin{proof}
  A morphism $v \in \MON(\hyV)(A,B)$ induces a morphism $(\im_{\R} \tensor v) \in \MONRUN(\hyV,\hyG)(\R \tensor A,\R \tensor B)$, which can be read as a morphism of cliques $(\im_{\R} \tensor v) \in \EFF(\hyV,\hyG)(A,B)$.
  This is tensoring with an identity, which is indeed functorial.

  Let us now show that this functor strictly preserves the premonoidal structure.
  The fact that it preserves right whiskerings is immediate.
  The fact that it preserves left whiskerings follows from the axioms of symmetry (\Cref{fig:whiskeringeq}, left).
  Associators and unitors are identities, which are preserved by tensoring with an identity.
  \begin{figure}[H]
    \centering

\tikzset{every picture/.style={line width=0.75pt}} %

\begin{tikzpicture}[x=0.75pt,y=0.75pt,yscale=-1,xscale=1]
\draw    (516,64) -- (516,80) ;
\draw    (532,64) -- (532,80) ;
\draw [color={rgb, 255:red, 191; green, 97; blue, 106 }  ,draw opacity=1 ]   (500,64) -- (500,80) ;
\draw    (48,20) -- (48,36) ;
\draw    (64,20) -- (64,36) ;
\draw [color={rgb, 255:red, 0; green, 0; blue, 0 }  ,draw opacity=1 ]   (32.01,8) -- (32.01,20) ;
\draw    (48,52) -- (48,80) ;
\draw    (64,52) -- (64,80) ;
\draw  [draw opacity=0] (48,20) -- (68,20) -- (68,34) -- (48,34) -- cycle ;
\draw [color={rgb, 255:red, 191; green, 97; blue, 106 }  ,draw opacity=1 ]   (8.01,20) .. controls (8.21,32.4) and (31.8,23.2) .. (32,36) ;
\draw [color={rgb, 255:red, 0; green, 0; blue, 0 }  ,draw opacity=1 ]   (20.01,8) -- (20.01,20) ;
\draw [color={rgb, 255:red, 191; green, 97; blue, 106 }  ,draw opacity=1 ]   (8.01,8) -- (8.01,20) ;
\draw [color={rgb, 255:red, 0; green, 0; blue, 0 }  ,draw opacity=1 ]   (20.01,20) .. controls (20.21,32.4) and (7.81,23.2) .. (8.01,36) ;
\draw [color={rgb, 255:red, 0; green, 0; blue, 0 }  ,draw opacity=1 ]   (32.01,20) .. controls (32.21,32.4) and (19.81,23.2) .. (20.01,36) ;
\draw [color={rgb, 255:red, 0; green, 0; blue, 0 }  ,draw opacity=1 ]   (64,8) -- (64,20) ;
\draw [color={rgb, 255:red, 0; green, 0; blue, 0 }  ,draw opacity=1 ]   (48,8) -- (48,20) ;
\draw [color={rgb, 255:red, 0; green, 0; blue, 0 }  ,draw opacity=1 ]   (20,36) -- (20,52) ;
\draw [color={rgb, 255:red, 0; green, 0; blue, 0 }  ,draw opacity=1 ]   (8,36) -- (8,52) ;
\draw [color={rgb, 255:red, 191; green, 97; blue, 106 }  ,draw opacity=1 ]   (32,52) .. controls (32.2,64.4) and (7.8,55.2) .. (8,68) ;
\draw [color={rgb, 255:red, 0; green, 0; blue, 0 }  ,draw opacity=1 ]   (8,52) .. controls (8.2,64.4) and (19.8,55.2) .. (20,68) ;
\draw [color={rgb, 255:red, 0; green, 0; blue, 0 }  ,draw opacity=1 ]   (20,52) .. controls (20.2,64.4) and (31.8,55.2) .. (32,68) ;
\draw   (40,36) -- (72,36) -- (72,52) -- (40,52) -- cycle ;
\draw [color={rgb, 255:red, 0; green, 0; blue, 0 }  ,draw opacity=1 ]   (32,68) -- (32,80) ;
\draw [color={rgb, 255:red, 0; green, 0; blue, 0 }  ,draw opacity=1 ]   (20,68) -- (20,80) ;
\draw [color={rgb, 255:red, 191; green, 97; blue, 106 }  ,draw opacity=1 ]   (8,68) -- (8,80) ;
\draw [color={rgb, 255:red, 191; green, 97; blue, 106 }  ,draw opacity=1 ]   (32,36) -- (32,52) ;
\draw    (132,20) -- (132,36) ;
\draw    (148,20) -- (148,36) ;
\draw    (132,52) -- (132,80) ;
\draw    (148,52) -- (148,80) ;
\draw  [draw opacity=0] (132,20) -- (152,20) -- (152,34) -- (132,34) -- cycle ;
\draw [color={rgb, 255:red, 0; green, 0; blue, 0 }  ,draw opacity=1 ]   (104.01,8) -- (104,80) ;
\draw [color={rgb, 255:red, 0; green, 0; blue, 0 }  ,draw opacity=1 ]   (148,8) -- (148,20) ;
\draw [color={rgb, 255:red, 0; green, 0; blue, 0 }  ,draw opacity=1 ]   (132,8) -- (132,20) ;
\draw   (124,36) -- (156,36) -- (156,52) -- (124,52) -- cycle ;
\draw  [draw opacity=0] (68,20) -- (96,20) -- (96,68) -- (68,68) -- cycle ;
\draw [color={rgb, 255:red, 0; green, 0; blue, 0 }  ,draw opacity=1 ]   (116.01,8) -- (116,80) ;
\draw [color={rgb, 255:red, 191; green, 97; blue, 106 }  ,draw opacity=1 ]   (92.01,8) -- (92,80) ;
\draw  [draw opacity=0] (160,24) -- (188,24) -- (188,72) -- (160,72) -- cycle ;
\draw    (240,8) -- (240,48) ;
\draw    (256,8) -- (256,48) ;
\draw    (240,64) -- (240,80) ;
\draw    (256,64) -- (256,80) ;
\draw [color={rgb, 255:red, 191; green, 97; blue, 106 }  ,draw opacity=1 ]   (192,32) .. controls (192.2,44.4) and (227.8,35.2) .. (228,48) ;
\draw [color={rgb, 255:red, 0; green, 0; blue, 0 }  ,draw opacity=1 ]   (208,48) -- (208,64) ;
\draw [color={rgb, 255:red, 0; green, 0; blue, 0 }  ,draw opacity=1 ]   (192,48) -- (192,64) ;
\draw [color={rgb, 255:red, 191; green, 97; blue, 106 }  ,draw opacity=1 ]   (228,64) .. controls (228.2,76.4) and (191.8,67.2) .. (192,80) ;
\draw   (220,48) -- (264,48) -- (264,64) -- (220,64) -- cycle ;
\draw   (200,16) -- (232,16) -- (232,32) -- (200,32) -- cycle ;
\draw    (208,8) -- (208,16) ;
\draw    (224,8) -- (224,16) ;
\draw [color={rgb, 255:red, 191; green, 97; blue, 106 }  ,draw opacity=1 ]   (192.01,8) -- (192,32) ;
\draw  [color={rgb, 255:red, 255; green, 255; blue, 255 }  ,draw opacity=1 ][fill={rgb, 255:red, 255; green, 255; blue, 255 }  ,fill opacity=1 ] (205,39) .. controls (205,37.9) and (204.1,37) .. (203,37) .. controls (201.9,37) and (201,37.9) .. (201,39) .. controls (201,40.1) and (201.9,41) .. (203,41) .. controls (204.1,41) and (205,40.1) .. (205,39) -- cycle ;
\draw  [color={rgb, 255:red, 255; green, 255; blue, 255 }  ,draw opacity=1 ][fill={rgb, 255:red, 255; green, 255; blue, 255 }  ,fill opacity=1 ] (216,40.55) .. controls (216,39.45) and (215.1,38.55) .. (214,38.55) .. controls (212.9,38.55) and (212,39.45) .. (212,40.55) .. controls (212,41.66) and (212.9,42.55) .. (214,42.55) .. controls (215.1,42.55) and (216,41.66) .. (216,40.55) -- cycle ;
\draw  [color={rgb, 255:red, 255; green, 255; blue, 255 }  ,draw opacity=1 ][fill={rgb, 255:red, 255; green, 255; blue, 255 }  ,fill opacity=1 ] (216,71.43) .. controls (216,70.32) and (215.1,69.43) .. (214,69.43) .. controls (212.9,69.43) and (212,70.32) .. (212,71.43) .. controls (212,72.53) and (212.9,73.43) .. (214,73.43) .. controls (215.1,73.43) and (216,72.53) .. (216,71.43) -- cycle ;
\draw  [color={rgb, 255:red, 255; green, 255; blue, 255 }  ,draw opacity=1 ][fill={rgb, 255:red, 255; green, 255; blue, 255 }  ,fill opacity=1 ] (204,72.29) .. controls (204,71.18) and (203.1,70.29) .. (202,70.29) .. controls (200.9,70.29) and (200,71.18) .. (200,72.29) .. controls (200,73.39) and (200.9,74.29) .. (202,74.29) .. controls (203.1,74.29) and (204,73.39) .. (204,72.29) -- cycle ;
\draw [color={rgb, 255:red, 0; green, 0; blue, 0 }  ,draw opacity=1 ]   (208,32) .. controls (208.2,44.4) and (191.8,35.2) .. (192,48) ;
\draw [color={rgb, 255:red, 0; green, 0; blue, 0 }  ,draw opacity=1 ]   (224,32) .. controls (224.2,44.4) and (207.8,35.2) .. (208,48) ;
\draw [color={rgb, 255:red, 0; green, 0; blue, 0 }  ,draw opacity=1 ]   (192,64) .. controls (192.2,76.4) and (207.8,67.2) .. (208,80) ;
\draw [color={rgb, 255:red, 0; green, 0; blue, 0 }  ,draw opacity=1 ]   (208,64) .. controls (208.2,76.4) and (223.8,67.2) .. (224,80) ;
\draw    (332,8) -- (332,24) ;
\draw    (348,8) -- (348,24) ;
\draw    (332,40) -- (332,56) ;
\draw    (348,40) -- (348,56) ;
\draw [color={rgb, 255:red, 191; green, 97; blue, 106 }  ,draw opacity=1 ]   (284,8) .. controls (284.2,20.4) and (319.8,11.2) .. (320,24) ;
\draw [color={rgb, 255:red, 0; green, 0; blue, 0 }  ,draw opacity=1 ]   (300,24) -- (300,40) ;
\draw [color={rgb, 255:red, 0; green, 0; blue, 0 }  ,draw opacity=1 ]   (284,24) -- (284,40) ;
\draw [color={rgb, 255:red, 191; green, 97; blue, 106 }  ,draw opacity=1 ]   (320,40) .. controls (320.2,52.4) and (283.8,43.2) .. (284,56) ;
\draw   (312,24) -- (356,24) -- (356,40) -- (312,40) -- cycle ;
\draw  [color={rgb, 255:red, 255; green, 255; blue, 255 }  ,draw opacity=1 ][fill={rgb, 255:red, 255; green, 255; blue, 255 }  ,fill opacity=1 ] (297,15) .. controls (297,13.9) and (296.11,13) .. (295,13) .. controls (293.9,13) and (293.01,13.9) .. (293.01,15) .. controls (293.01,16.1) and (293.9,17) .. (295,17) .. controls (296.11,17) and (297,16.1) .. (297,15) -- cycle ;
\draw  [color={rgb, 255:red, 255; green, 255; blue, 255 }  ,draw opacity=1 ][fill={rgb, 255:red, 255; green, 255; blue, 255 }  ,fill opacity=1 ] (308,16.55) .. controls (308,15.45) and (307.11,14.55) .. (306,14.55) .. controls (304.9,14.55) and (304.01,15.45) .. (304.01,16.55) .. controls (304.01,17.66) and (304.9,18.55) .. (306,18.55) .. controls (307.11,18.55) and (308,17.66) .. (308,16.55) -- cycle ;
\draw  [color={rgb, 255:red, 255; green, 255; blue, 255 }  ,draw opacity=1 ][fill={rgb, 255:red, 255; green, 255; blue, 255 }  ,fill opacity=1 ] (308,47.43) .. controls (308,46.32) and (307.11,45.43) .. (306,45.43) .. controls (304.9,45.43) and (304.01,46.32) .. (304.01,47.43) .. controls (304.01,48.53) and (304.9,49.43) .. (306,49.43) .. controls (307.11,49.43) and (308,48.53) .. (308,47.43) -- cycle ;
\draw  [color={rgb, 255:red, 255; green, 255; blue, 255 }  ,draw opacity=1 ][fill={rgb, 255:red, 255; green, 255; blue, 255 }  ,fill opacity=1 ] (296,48.29) .. controls (296,47.18) and (295.11,46.29) .. (294,46.29) .. controls (292.9,46.29) and (292.01,47.18) .. (292.01,48.29) .. controls (292.01,49.39) and (292.9,50.29) .. (294,50.29) .. controls (295.11,50.29) and (296,49.39) .. (296,48.29) -- cycle ;
\draw [color={rgb, 255:red, 0; green, 0; blue, 0 }  ,draw opacity=1 ]   (300,8) .. controls (300.2,20.4) and (283.8,11.2) .. (284,24) ;
\draw [color={rgb, 255:red, 0; green, 0; blue, 0 }  ,draw opacity=1 ]   (316.01,8) .. controls (316.21,20.4) and (299.8,11.2) .. (300,24) ;
\draw [color={rgb, 255:red, 0; green, 0; blue, 0 }  ,draw opacity=1 ]   (284,40) .. controls (284.2,52.4) and (299.8,43.2) .. (300,56) ;
\draw [color={rgb, 255:red, 0; green, 0; blue, 0 }  ,draw opacity=1 ]   (300,40) .. controls (300.2,52.4) and (315.8,43.2) .. (316,56) ;
\draw    (332,56) -- (332,80) ;
\draw    (348,56) -- (348,80) ;
\draw   (292,56) -- (324,56) -- (324,72) -- (292,72) -- cycle ;
\draw    (300,72) -- (300,80) ;
\draw    (316,72) -- (316,80) ;
\draw [color={rgb, 255:red, 191; green, 97; blue, 106 }  ,draw opacity=1 ]   (284.01,56) -- (284,80) ;
\draw [color={rgb, 255:red, 191; green, 97; blue, 106 }  ,draw opacity=1 ]   (400,40) .. controls (400.2,52.4) and (435.8,43.2) .. (436,56) ;
\draw  [color={rgb, 255:red, 255; green, 255; blue, 255 }  ,draw opacity=1 ][fill={rgb, 255:red, 255; green, 255; blue, 255 }  ,fill opacity=1 ] (413,47) .. controls (413,45.9) and (412.11,45) .. (411,45) .. controls (409.9,45) and (409.01,45.9) .. (409.01,47) .. controls (409.01,48.1) and (409.9,49) .. (411,49) .. controls (412.11,49) and (413,48.1) .. (413,47) -- cycle ;
\draw  [color={rgb, 255:red, 255; green, 255; blue, 255 }  ,draw opacity=1 ][fill={rgb, 255:red, 255; green, 255; blue, 255 }  ,fill opacity=1 ] (424,48.55) .. controls (424,47.45) and (423.11,46.55) .. (422,46.55) .. controls (420.9,46.55) and (420.01,47.45) .. (420.01,48.55) .. controls (420.01,49.66) and (420.9,50.55) .. (422,50.55) .. controls (423.11,50.55) and (424,49.66) .. (424,48.55) -- cycle ;
\draw [color={rgb, 255:red, 0; green, 0; blue, 0 }  ,draw opacity=1 ]   (416,40) .. controls (416.2,52.4) and (399.8,43.2) .. (400,56) ;
\draw [color={rgb, 255:red, 0; green, 0; blue, 0 }  ,draw opacity=1 ]   (432.01,40) .. controls (432.21,52.4) and (415.8,43.2) .. (416,56) ;
\draw    (416,8) -- (416,24) ;
\draw    (432,8) -- (432,24) ;
\draw   (392,24) -- (440,24) -- (440,40) -- (392,40) -- cycle ;
\draw [color={rgb, 255:red, 191; green, 97; blue, 106 }  ,draw opacity=1 ]   (400,8) -- (400,24) ;
\draw [color={rgb, 255:red, 191; green, 97; blue, 106 }  ,draw opacity=1 ]   (436,56) .. controls (435.8,68.4) and (400.2,59.2) .. (400,72) ;
\draw  [color={rgb, 255:red, 255; green, 255; blue, 255 }  ,draw opacity=1 ][fill={rgb, 255:red, 255; green, 255; blue, 255 }  ,fill opacity=1 ] (420,63.33) .. controls (420,62.23) and (420.9,61.33) .. (422,61.33) .. controls (423.11,61.33) and (424,62.23) .. (424,63.33) .. controls (424,64.44) and (423.11,65.33) .. (422,65.33) .. controls (420.9,65.33) and (420,64.44) .. (420,63.33) -- cycle ;
\draw  [color={rgb, 255:red, 255; green, 255; blue, 255 }  ,draw opacity=1 ][fill={rgb, 255:red, 255; green, 255; blue, 255 }  ,fill opacity=1 ] (409,64.67) .. controls (409,63.56) and (409.9,62.67) .. (411,62.67) .. controls (412.11,62.67) and (413,63.56) .. (413,64.67) .. controls (413,65.77) and (412.11,66.67) .. (411,66.67) .. controls (409.9,66.67) and (409,65.77) .. (409,64.67) -- cycle ;
\draw [color={rgb, 255:red, 0; green, 0; blue, 0 }  ,draw opacity=1 ]   (400,56) .. controls (399.8,68.4) and (416.2,59.2) .. (416,72) ;
\draw [color={rgb, 255:red, 0; green, 0; blue, 0 }  ,draw opacity=1 ]   (416,56) .. controls (415.8,68.4) and (432.2,59.2) .. (432,72) ;
\draw   (440,48) -- (472,48) -- (472,64) -- (440,64) -- cycle ;
\draw    (448,64) -- (448,72) ;
\draw    (464,64) -- (464,72) ;
\draw    (448,8) -- (448,48) ;
\draw    (464,8) -- (464,48) ;
\draw    (416,72) -- (416,80) ;
\draw    (432,72) -- (432,80) ;
\draw [color={rgb, 255:red, 191; green, 97; blue, 106 }  ,draw opacity=1 ]   (400,72) -- (400,80) ;
\draw    (448,72) -- (448,80) ;
\draw    (464,72) -- (464,80) ;
\draw [color={rgb, 255:red, 191; green, 97; blue, 106 }  ,draw opacity=1 ]   (500,16) .. controls (500.2,28.4) and (535.8,19.2) .. (536,32) ;
\draw  [color={rgb, 255:red, 255; green, 255; blue, 255 }  ,draw opacity=1 ][fill={rgb, 255:red, 255; green, 255; blue, 255 }  ,fill opacity=1 ] (513,23) .. controls (513,21.9) and (512.11,21) .. (511,21) .. controls (509.9,21) and (509.01,21.9) .. (509.01,23) .. controls (509.01,24.1) and (509.9,25) .. (511,25) .. controls (512.11,25) and (513,24.1) .. (513,23) -- cycle ;
\draw  [color={rgb, 255:red, 255; green, 255; blue, 255 }  ,draw opacity=1 ][fill={rgb, 255:red, 255; green, 255; blue, 255 }  ,fill opacity=1 ] (524,24.55) .. controls (524,23.45) and (523.11,22.55) .. (522,22.55) .. controls (520.9,22.55) and (520.01,23.45) .. (520.01,24.55) .. controls (520.01,25.66) and (520.9,26.55) .. (522,26.55) .. controls (523.11,26.55) and (524,25.66) .. (524,24.55) -- cycle ;
\draw [color={rgb, 255:red, 0; green, 0; blue, 0 }  ,draw opacity=1 ]   (516,16) .. controls (516.2,28.4) and (499.8,19.2) .. (500,32) ;
\draw [color={rgb, 255:red, 0; green, 0; blue, 0 }  ,draw opacity=1 ]   (532.01,16) .. controls (532.21,28.4) and (515.8,19.2) .. (516,32) ;
\draw [color={rgb, 255:red, 191; green, 97; blue, 106 }  ,draw opacity=1 ]   (536,32) .. controls (535.8,44.4) and (500.2,35.2) .. (500,48) ;
\draw  [color={rgb, 255:red, 255; green, 255; blue, 255 }  ,draw opacity=1 ][fill={rgb, 255:red, 255; green, 255; blue, 255 }  ,fill opacity=1 ] (520,39.33) .. controls (520,38.23) and (520.9,37.33) .. (522,37.33) .. controls (523.11,37.33) and (524,38.23) .. (524,39.33) .. controls (524,40.44) and (523.11,41.33) .. (522,41.33) .. controls (520.9,41.33) and (520,40.44) .. (520,39.33) -- cycle ;
\draw  [color={rgb, 255:red, 255; green, 255; blue, 255 }  ,draw opacity=1 ][fill={rgb, 255:red, 255; green, 255; blue, 255 }  ,fill opacity=1 ] (509,40.67) .. controls (509,39.56) and (509.9,38.67) .. (511,38.67) .. controls (512.11,38.67) and (513,39.56) .. (513,40.67) .. controls (513,41.77) and (512.11,42.67) .. (511,42.67) .. controls (509.9,42.67) and (509,41.77) .. (509,40.67) -- cycle ;
\draw [color={rgb, 255:red, 0; green, 0; blue, 0 }  ,draw opacity=1 ]   (500,32) .. controls (499.8,44.4) and (516.2,35.2) .. (516,48) ;
\draw [color={rgb, 255:red, 0; green, 0; blue, 0 }  ,draw opacity=1 ]   (516,32) .. controls (515.8,44.4) and (532.2,35.2) .. (532,48) ;
\draw   (540,24) -- (572,24) -- (572,40) -- (540,40) -- cycle ;
\draw    (548,40) -- (548,48) ;
\draw    (564,40) -- (564,48) ;
\draw    (548,8) -- (548,24) ;
\draw    (564,8) -- (564,24) ;
\draw    (548,48) -- (548,80) ;
\draw    (564,48) -- (564,80) ;
\draw    (516,8) -- (516,16) ;
\draw    (532,8) -- (532,16) ;
\draw [color={rgb, 255:red, 191; green, 97; blue, 106 }  ,draw opacity=1 ]   (500,8) -- (500,16) ;
\draw   (492,48) -- (540,48) -- (540,64) -- (492,64) -- cycle ;
\draw  [draw opacity=0] (260,20) -- (288,20) -- (288,68) -- (260,68) -- cycle ;
\draw  [draw opacity=0] (360,20) -- (388,20) -- (388,68) -- (360,68) -- cycle ;
\draw  [draw opacity=0] (468,20) -- (496,20) -- (496,68) -- (468,68) -- cycle ;

\draw (56,44) node  [font=\footnotesize]  {$v$};
\draw (56.5,68.5) node  [font=\tiny]  {$\dotsc $};
\draw (26.5,68.5) node  [font=\tiny]  {$\dotsc $};
\draw (25.5,18.5) node  [font=\tiny]  {$\dotsc $};
\draw (55.5,18.5) node  [font=\tiny]  {$\dotsc $};
\draw (140,44) node  [font=\footnotesize]  {$v$};
\draw (140.5,68.5) node  [font=\tiny]  {$\dotsc $};
\draw (110.5,68.5) node  [font=\tiny]  {$\dotsc $};
\draw (109.5,18.5) node  [font=\tiny]  {$\dotsc $};
\draw (139.5,18.5) node  [font=\tiny]  {$\dotsc $};
\draw (82,44) node  [font=\footnotesize]  {$=$};
\draw (174,44) node  [font=\footnotesize]  {$;$};
\draw (242,56) node  [font=\footnotesize]  {$x$};
\draw (216,24) node  [font=\footnotesize]  {$v$};
\draw (334,32) node  [font=\footnotesize]  {$x$};
\draw (308,64) node  [font=\footnotesize]  {$v$};
\draw (418,32) node  [font=\footnotesize]  {$x$};
\draw (456,56) node  [font=\footnotesize]  {$v$};
\draw (556,32) node  [font=\footnotesize]  {$v$};
\draw (518,56) node  [font=\footnotesize]  {$x$};
\draw (274,44) node  [font=\footnotesize]  {$=$};
\draw (374,44) node  [font=\footnotesize]  {$;$};
\draw (482,44) node  [font=\footnotesize]  {$=$};

\end{tikzpicture}

    \caption{Preservation of whiskerings, and centrality.}
    \label{fig:centrality}    \label{fig:whiskeringeq}
  \end{figure}
  Finally, we can check by string diagrams that the image of this functor is central, interchanging with any given $x \colon \R \tensor C \to \R \tensor D$ (\Cref{fig:centrality}, center and right).
\end{proof}

\begin{lem}
  \label{lemma:freeness}
  Let $(\hyV,\hyG)$ be an \polygraphCouple{} and consider the \effectfulCategory{} determined by \(\MON(\hyV) \to \EFF(\hyV,\hyG)\).
  Let $\cbaseV \to \ccatC$ be a strict effectful category endowed with an \polygraphCouple{} morphism $F \colon (\hyV,\hyG) \to \mathcal{U}(\cbaseV,\ccatC)$.
  There exists a unique strict effectful functor from $(\MON(\hyV) \to \EFF(\hyV,\hyG))$ to $(\cbaseV \to \ccatC)$ commuting with $F$ as an \polygraphCouple{} morphism.
\end{lem}
\begin{proof}
  By the same freeness result for monoidal categories, there already exists a unique strict monoidal functor $H_{0} \colon \MON(\hyV) \to \cbaseV$ that sends any object $A \in \obj{\hyV}$ to $F_{obj}(A)$.

  We will show there is a unique way to extend this functor together with the hypergraph assignment $\hyG \to \catC$ into a functor $H \colon \EFF(\hyV,\hyG) \to \ccatC$. Even when we know that the runtime can appear always on the left, we choose here to define the functor independently of the position of the runtime and later check that it preserves braidings. That is, giving such a functor amounts to give some mapping of morphisms containing the runtime $\R$ in some position in their input and output,
  \[f \colon A_{0} \tensor \dots \tensor \R \tensor \dots \tensor A_{n} \to B_{0} \tensor \dots \tensor \R \tensor \dots \tensor B_{m}\]
  to morphisms $H(f) \colon FA_{0} \tensor \dots  \tensor FA_{n} \to FB_{0} \tensor \dots  \tensor FB_{n}$ in $\ccatC$, in a way that preserves composition, whiskerings, inclusions from $\MON(\hyV)$, and that is invariant to composition with braidings.
  In order to define this mapping, we will perform structural induction over the monoidal terms of the runtime monoidal category of the form
  $\MONRUN(\hyV,\hyG)(A_{0} \tensor \dots \tensor \R^{(i)} \tensor \dots \tensor A_{n}, \R \tensor B_{0} \tensor \dots  \tensor \R^{(j)} \tensor  \dots \tensor B_{m})$ and show that it is the only mapping with these properties (\Cref{fig:assignment}).

  Monoidal terms in a strict, freely presented, monoidal category are formed by identities ($\im$), composition $(\comp)$, tensoring $(\otimes)$, and some generators (in this case, in \Cref{fig:rungen}). Here is where we apply the main result of Joyal and Street \cite{joyal91}: reasoning about monoidal terms is the same as reasoning about string diagrams; thus, we can prove a result about string diagrams reasoning inductively on monoidal terms. Monoidal terms are subject to \emph{(i)} functoriality of the tensor, $\im \tensor \im = \im$ and $(f \comp g) \tensor (h \comp k) = (f \tensor h) \comp (g \tensor k)$; \emph{(ii)} associativity and unitality of the tensor, $f \tensor \im_{I} = f$ and $f \tensor (g \tensor h) = (f \tensor g) \tensor h$; \emph{(iii)} the usual unitality, $f \comp \im = f$ and $\im \comp f = f$ and associativity $f \comp (g \comp h) = (f \comp g) \comp h$; \emph{(iv)} the axioms of our presentation (in this case, in \Cref{fig:runaxiom}).
  \begin{figure}[H]
    \centering
    \begin{align*}
      & H\left( \IconId{A} \right) = \mathrm{id}_A; \quad
      H\left( \IconRun{} \right) = \mathrm{id}_I; \quad
      H\left( \EmptyBox{} \right) = \mathrm{id}_I; \quad
      H\left( \IconComp{x}{y} \right) = H\left( \IconEff{x} \right) \comp H\left( \IconEff{y} \right);
      \\ & H\left( \IconEff{x} \IconPure{u}\right) =
      (\mathrm{id} \otimes H_0\left(\IconPure{u}\right)) \comp (H\left(\IconEff{x}\right) \otimes \mathrm{id}) =
      (H\left(\IconEff{x}\right) \otimes \mathrm{id}) \comp (\mathrm{id} \otimes H_0\left(\IconPure{u}\right)); \\
    & H\left(\IconPure{u} \IconEff{x}\right) =
    (H_0\left(\IconPure{u}\right) \otimes \mathrm{id}) \comp (\mathrm{id} \otimes H\left(\IconEff{x}\right)) =
    (\mathrm{id} \otimes H\left(\IconEff{x}\right)) \comp (H_0\left(\IconPure{u}\right) \otimes \mathrm{id}); \\
    & H\left(\IconEff{f}\right) = F(f); \quad
    H\left(\IconPure{v}\right) = F_0(v)^\circ; \quad
    H\left(\IconSwapOne{}\right) =
    H\left(\IconSwapTwo{}\right) = \mathrm{id};
    \end{align*}
    \caption{Assignment on morphisms, defined by structural induction on terms.}
    \label{fig:assignment}
  \end{figure}

  \begin{itemize}
    \item
      If the term is an identity, it can be \emph{(i)} an identity on an object of the \effectfulPolygraph{}, $A \in \obj{(\hyV,\hyG)}$, in which case it must be mapped to the same identity by functoriality, $H(\im_{A}) = \im_{A}$; \emph{(ii)} an identity on the runtime, in which case it must be mapped to the identity on the unit object, $H(\im_{\R}) = \im_{I}$; or \emph{(iii)} an identity on the unit object, in which case it must be mapped to the identity on the unit, $H(\im_{I}) = \im_{I}$.

    \item
      If the term is a composition, $(f \comp g) \colon A_{0} \tensor \dots \tensor \R \tensor \dots \tensor A_{n} \to C_{0} \tensor \dots \tensor \R \tensor \dots \tensor C_{k}$, it must be along a boundary of the form $B_{0} \tensor \dots \tensor \R \tensor \dots \tensor B_{m}$: this is because every generator leaves the number of runtimes, $\R$, invariant.
      Thus, each one of the components determines itself a braid clique morphism.
      We must preserve composition of braid clique morphisms, so we must map $H(f \comp g) = H(f) \comp H(g)$.

    \item
      If the term is a tensor of two terms,  $(x \tensor u) \colon A_{0} \tensor \dots \tensor \R \tensor \dots \tensor A_{n} \to B_{0} \tensor \dots \tensor \R \tensor \dots \tensor B_{m}$, then only one of them was a term taking $\R$ as input and output (without loss of generality, assume it to be the first one) and the other was not: again, by construction, there are no morphisms taking one $\R$ as input and producing none, or viceversa.
      We split this morphism into $x \colon A_{0} \tensor\dots\tensor \R \tensor \dots \tensor A_{i-1} \to B_{0} \tensor \dots \tensor \R \tensor \dots \tensor B_{j-1}$ and $u \colon A_{i} \tensor \dots \tensor A_{n} \to B_{j} \tensor \dots \tensor B_{m}$.

      Again by structural induction, this time over terms $u \colon A_{i} \tensor \dots \tensor A_{n} \to B_{j} \tensor \dots \tensor B_{m}$,
      we know that the morphism must be either a generator in $\hyV(A_{i},\dots,A_{n};B_{j},\dots,B_{n})$ or a composition and tensoring of them. That is, $u$ is a morphism in the image of $\MON(\hyV)$, and it must be mapped according to the functor $H_{0} \colon \MON(\hyV) \to \cbaseV$.

      By induction hypothesis, we know how to map the morphism $x \colon A_{0} \tensor \dots\tensor  \R \tensor \dots \tensor A_{i-1} \to B_{0} \tensor \dots \tensor \R \tensor \dots \tensor B_{j-1}$.
      This means that, given any tensoring $x \tensor u$, we must map it to $H(x \tensor u) = (H(x) \tensor \im) \comp (\im \tensor H_{0}(u)) =  (\im \tensor H_{0}(u)) \comp (H(x) \tensor \im)$, where $H_{0}(u)$ is central.

    \item
      If the string diagram consists of a single generator, $f \colon \R \tensor A \to \R \tensor B$, it can only come from a generator $$f \in \Run(\hyV,\hyG)(\R,A_{0},\dots,A_{n};\R, B_{0},\dots,B_{m}) = \hyG(A_{0},\dots,A_{n}; B_{0},\dots,B_{m}),$$
      which must be mapped to $H(f) = F(f) \in \ccatC(A_{0} \tensor \dots \tensor A_{n}, B_{0} \tensor \dots \tensor B_{m})$.
      If the string diagram consists of a single braiding, it must be mapped to the identity, because the want the assignment to be invariant to braidings.

  \end{itemize}

  Now, we need to prove that this assignment is well-defined with respect to the axioms of these monoidal terms. Our reasoning follows \Cref{fig:assignmentwell}.

  \begin{figure}
    \centering
    \textsc{Functoriality of the Tensor.}
    \begin{align*}
      & (H\left(\EmptyBox{}\right) \otimes \mathrm{id}) \comp (\mathrm{id} \otimes H_0\left(\EmptyBox{}\right)) =
      H\left(\EmptyBox{}\right); \\
      & (H\left(\IconEff{x} \otimes \mathrm{id}\right)) \comp
      (\mathrm{id} \otimes H_0\left( \IconPure{u} \right)) \comp
      (H\left(\IconEff{y} \right) \otimes \mathrm{id}) \comp
      (\mathrm{id} \otimes H_0\left( \IconPure{v} \right)) =\\
      & ((H\left(\IconEff{x}\right) \comp H\left(\IconEff{y}\right)) \otimes \mathrm{id}) \comp
      (\mathrm{id} \otimes (H_0\left( \IconPure{u} \right) \comp H_0\left( \IconPure{v} \right)));
    \end{align*}
    \textsc{Monoidality of the Tensor.}
    \begin{align*}
      & (H\left(\IconEff{x}\right) \otimes \mathrm{id}_X) \comp (\mathrm{id} \otimes H_0\left(\EmptyBox{}\right)) =
      H\left(\IconEff{x}\right); \\
      & (((H\left(\IconEff{x}\right) \otimes \mathrm{id}_X) \comp (\mathrm{id} \otimes H_0\left(\IconPure{u}\right))) \otimes \mathrm{id}) \comp (\mathrm{id} \otimes H_0\left(\IconPure{v}\right)) = \\
      & (H\left(\IconEff{x}\right) \otimes \mathrm{id}_X) \comp (\mathrm{id} \otimes H_0\left(\IconPure{u}\right) \otimes H_0\left(\IconPure{v}\right));
    \end{align*}
    \textsc{Category Axioms.}
    \begin{align*}
      & (H\left(\IconEff{x}\right) \comp
         H\left(\IconEff{y}\right) )  \comp H\left(\IconEff{z}\right) =
         H\left(\IconEff{x}\right) \comp (H\left(\IconEff{y}\right) \comp H\left(\IconEff{z}\right)); \\
      & H\left(\IconEff{x}\right) \comp H\left(\IconThree{}\right) =
      H\left(\IconThree{}\right) \comp H\left(\IconEff{x}\right) =
      H\left(\IconEff{x}\right);
    \end{align*}
    \textsc{Runtime Axioms.}
    \begin{align*}
      & H\left(\IconSwapOne{}\right) \comp H\left(\IconSwapTwo{}\right) = \mathrm{id} ; \quad
      H\left(\IconSwapTwo{}\right) \comp H\left(\IconSwapOne{}\right) = \mathrm{id} ; \\
      & (H\left( \IconRun{}\right) \otimes \mathrm{id}) \comp
      (\mathrm{id} \otimes H_0\left(\IconPure{v}\right)) \comp
      H\left(\IconOneToThree{}\right) =
      H\left(\IconOneToThree{}\right) \comp
      (\mathrm{id} \otimes H_0\left(\IconPure{v}\right)) \comp
      (H\left( \IconRun{}\right) \otimes \mathrm{id}) ; \\
      & (\mathrm{id} \otimes H\left( \IconRun{}\right)) \comp
      (H_0\left(\IconPure{v}\right) \otimes \mathrm{id})  \comp
      H\left(\IconThreeToOne{}\right) =
      H\left(\IconThreeToOne{}\right) \comp
      (\mathrm{id} \otimes H_0\left(\IconPure{v}\right)) \comp
      (H\left( \IconRun{}\right)\otimes \mathrm{id}) ;
    \end{align*}
    \caption{The assignment is well defined.}
    \label{fig:assignmentwell}
  \end{figure}

  \begin{itemize}
    \item
    The tensor is functorial. We know that $H(\im \tensor \im) = H(\im)$, both are identities and that can be formally proven by induction on the number of wires.
    Now, for the interchange law, consider a quartet of morphisms that can be composed or tensored first and such that, without loss of generality, we assume the runtime to be on the left side. Then, we can use centrality to argue that
    \begin{align*}
      H((x \tensor u) \comp (y \tensor v)) &= (H(x) \tensor \im) \comp (\im \tensor H_{0}(u)) \comp  (H(y) \tensor \im) \comp (\im \tensor H_{0}(v))
      \\&= ((H(x)\comp H(y)) \tensor \im) \comp (\im \tensor (H_{0}(u) \comp H_{0}(v)))
      \\&= H((x \comp y) \tensor (u \comp v)).
    \end{align*}

    \item
    The tensor is monoidal. We know that $H(x \tensor \im_{I}) = (H(x) \tensor \im_{I}) \comp (\im \tensor \im_{I}) = H(x)$.
    Now, for associativity, consider a triple of morphisms that can be tensored in two ways and such that, without loss of generality, we assume the runtime to be on the left side.
    Then, we can use centrality to argue that
    \begin{align*}
      H((x \tensor u) \tensor v)
      &=  (((H(x) \tensor \im)  \comp (\im \tensor H_{0}(u))) \tensor \im) \comp \im \tensor H_{0}(v)
      \\&= (H(x) \tensor \im) \comp (\im \tensor H_{0}(u) \tensor H_{0}(v))
      \\&= H(x \tensor (u \tensor v)).
    \end{align*}

    \item
    The terms form a category. And indeed, it is true by construction that $H(x \comp (y \comp z)) = H((x \comp y) \comp z)$ and also that $H(x \comp \im) = H(x)$, because $H$ preserves composition.

    \item
    The runtime category enforces some axioms. The composition of two braidings is mapped to the identity by the fact that $H$ preserves composition and sends both to the identity.  Both sides of the braid naturality over a morphism $v$ are mapped to $H_{0}(v)$; with the multiple braidings being mapped again to the identity.
  \end{itemize}
  Thus, $H$ is well-defined and it defines the only possible assignment and the only possible strict premonoidal functor.
\end{proof}

\begin{thm}[Runtime as a resource]
  \label{theorem:runtime-as-a-resource}
  The free strict \effectfulCategory{} over an \polygraphCouple{} $(\hyV,\hyG)$ is $\MON(\hyV) \to \EFF(\hyV,\hyG)$. Its morphisms $A \to B$ are in bijection with the morphisms $\R \tensor A \to \R \tensor B$ of the runtime monoidal category,
  \[\EFF(\hyV,\hyG)(A,B) \cong \MONRUN(\hyV,\hyG)(\R \tensor A, \R \tensor B).\]
\end{thm}
\begin{proof}
  We must first show that $\MON(\hyV) \to \EFF(\hyV,\hyG)$ is an \effectfulCategory{}.
  The first step is to see that $\EFF(\hyV,\hyG)$ forms a \premonoidalCategory{}; we apply \Cref{lemma:eff-is-premonoidal}.
  We also know that $\MON(\hyV)$ is a monoidal category: in fact, a strict, freely generated one.
  There exists an identity on objects functor, \(\MON(\hyV) \to \EFF(\hyV,\hyG)\), that strictly preserves the premonoidal structure and centrality, by \Cref{lemma:identity-mon-eff}.

  Let us now show that it is the free one over the \polygraphCouple{} $(\hyV,\hyG)$.
  Let $\cbaseV \to \ccatC$ be an effectful category, with an \polygraphCouple{} map $F \colon (\hyV,\hyG) \to \mathcal{U}(\cbaseV,\ccatC)$.
  We can construct a unique effectful functor from $(\MON(\hyV) \to \EFF(\hyV,\hyG))$ to $(\cbaseV \to \ccatC)$ giving its universal property; this is \Cref{lemma:freeness}, which concludes the proof.
\end{proof}

\begin{cor}[String diagrams for effectful categories]
  \label{cor:string-effectfuls}
  We can use string diagrams for \effectfulCategories{}, quotiented under the same isotopy as for monoidal categories, provided that we do represent the runtime as an extra wire that needs to be the input and output of every effectful morphism.
\end{cor}

\section{Example: Global State}%
\label{sec:exampleGlobalState}%

Let us provide an example of reasoning that uses string diagrams for
\effectfulCategories{}. Imperative programs are characterized by the presence of
a global state that can be mutated. Reading or writing to this global state
constitutes an effectful computation: the order of operations that affect some
global state cannot be changed. Let us propose a simple theory of global state
(\Cref{fig:axioms-noninterchange-state}) and let us show that it is enough to
capture the phenomenon of \emph{race conditions}.
These are an adaptation of the \emph{lens laws}
\cite{foster07combinators}.

\begin{figure}[!ht]

\tikzset{every picture/.style={line width=0.75pt}} %

\begin{tikzpicture}[x=0.75pt,y=0.75pt,yscale=-1,xscale=1]
\draw [color={rgb, 255:red, 0; green, 0; blue, 0 }  ,draw opacity=1 ]   (180,84) .. controls (192.22,83.71) and (191.56,86.38) .. (192,96) ;
\draw    (80,80) -- (80.01,96) ;
\draw    (236,12) -- (236,20) ;
\draw  [color={rgb, 255:red, 0; green, 0; blue, 0 }  ,draw opacity=1 ][fill={rgb, 255:red, 0; green, 0; blue, 0 }  ,fill opacity=1 ] (22,32) .. controls (22,30.9) and (22.9,30) .. (24,30) .. controls (25.11,30) and (26,30.9) .. (26,32) .. controls (26,33.1) and (25.11,34) .. (24,34) .. controls (22.9,34) and (22,33.1) .. (22,32) -- cycle ;
\draw    (24,32) -- (24.01,40) ;
\draw    (40,24) -- (40.01,40) ;
\draw  [color={rgb, 255:red, 0; green, 0; blue, 0 }  ,draw opacity=1 ][fill={rgb, 255:red, 0; green, 0; blue, 0 }  ,fill opacity=1 ] (38,24) .. controls (38,22.9) and (38.9,22) .. (40,22) .. controls (41.1,22) and (42,22.9) .. (42,24) .. controls (42,25.1) and (41.1,26) .. (40,26) .. controls (38.9,26) and (38,25.1) .. (38,24) -- cycle ;
\draw  [draw opacity=0] (44,4) -- (68,4) -- (68,44) -- (44,44) -- cycle ;
\draw  [color={rgb, 255:red, 0; green, 0; blue, 0 }  ,draw opacity=1 ][fill={rgb, 255:red, 0; green, 0; blue, 0 }  ,fill opacity=1 ] (78,20) .. controls (78,18.9) and (78.9,18) .. (80,18) .. controls (81.1,18) and (82,18.9) .. (82,20) .. controls (82,21.1) and (81.1,22) .. (80,22) .. controls (78.9,22) and (78,21.1) .. (78,20) -- cycle ;
\draw    (80,20) -- (80,28) ;
\draw  [color={rgb, 255:red, 0; green, 0; blue, 0 }  ,draw opacity=1 ][fill={rgb, 255:red, 0; green, 0; blue, 0 }  ,fill opacity=1 ] (78,28) .. controls (78,26.9) and (78.9,26) .. (80,26) .. controls (81.11,26) and (82,26.9) .. (82,28) .. controls (82,29.1) and (81.11,30) .. (80,30) .. controls (78.9,30) and (78,29.1) .. (78,28) -- cycle ;
\draw [color={rgb, 255:red, 0; green, 0; blue, 0 }  ,draw opacity=1 ]   (80,28) .. controls (91.71,28.79) and (91.43,28.79) .. (92,40) ;
\draw [color={rgb, 255:red, 0; green, 0; blue, 0 }  ,draw opacity=1 ]   (80,28) .. controls (68.57,27.65) and (68.29,29.36) .. (68,40) ;
\draw  [color={rgb, 255:red, 0; green, 0; blue, 0 }  ,draw opacity=1 ][fill={rgb, 255:red, 0; green, 0; blue, 0 }  ,fill opacity=1 ] (134.01,32) .. controls (134.01,30.9) and (134.9,30) .. (136.01,30) .. controls (137.11,30) and (138.01,30.9) .. (138.01,32) .. controls (138.01,33.1) and (137.11,34) .. (136.01,34) .. controls (134.9,34) and (134.01,33.1) .. (134.01,32) -- cycle ;
\draw    (136,16) -- (136.01,32) ;
\draw  [color={rgb, 255:red, 0; green, 0; blue, 0 }  ,draw opacity=1 ][fill={rgb, 255:red, 0; green, 0; blue, 0 }  ,fill opacity=1 ] (134,16) .. controls (134,14.9) and (134.9,14) .. (136,14) .. controls (137.1,14) and (138,14.9) .. (138,16) .. controls (138,17.1) and (137.1,18) .. (136,18) .. controls (134.9,18) and (134,17.1) .. (134,16) -- cycle ;
\draw  [draw opacity=0] (144,4) -- (168,4) -- (168,44) -- (144,44) -- cycle ;
\draw  [dash pattern={on 4.5pt off 4.5pt}] (168,12) -- (188,12) -- (188,36) -- (168,36) -- cycle ;
\draw  [color={rgb, 255:red, 0; green, 0; blue, 0 }  ,draw opacity=1 ][fill={rgb, 255:red, 255; green, 255; blue, 255 }  ,fill opacity=1 ] (234,20) .. controls (234,18.9) and (234.9,18) .. (236,18) .. controls (237.11,18) and (238,18.9) .. (238,20) .. controls (238,21.1) and (237.11,22) .. (236,22) .. controls (234.9,22) and (234,21.1) .. (234,20) -- cycle ;
\draw    (252,12) -- (252.01,28) ;
\draw  [color={rgb, 255:red, 0; green, 0; blue, 0 }  ,draw opacity=1 ][fill={rgb, 255:red, 255; green, 255; blue, 255 }  ,fill opacity=1 ] (250.01,28) .. controls (250.01,26.9) and (250.9,26) .. (252.01,26) .. controls (253.11,26) and (254.01,26.9) .. (254.01,28) .. controls (254.01,29.1) and (253.11,30) .. (252.01,30) .. controls (250.9,30) and (250.01,29.1) .. (250.01,28) -- cycle ;
\draw  [color={rgb, 255:red, 0; green, 0; blue, 0 }  ,draw opacity=1 ][fill={rgb, 255:red, 0; green, 0; blue, 0 }  ,fill opacity=1 ] (282,20) .. controls (282,18.9) and (282.9,18) .. (284,18) .. controls (285.11,18) and (286,18.9) .. (286,20) .. controls (286,21.1) and (285.11,22) .. (284,22) .. controls (282.9,22) and (282,21.1) .. (282,20) -- cycle ;
\draw    (284,12) -- (284,20) ;
\draw    (300,12) -- (300.01,28) ;
\draw  [color={rgb, 255:red, 0; green, 0; blue, 0 }  ,draw opacity=1 ][fill={rgb, 255:red, 255; green, 255; blue, 255 }  ,fill opacity=1 ] (298.01,28) .. controls (298.01,26.9) and (298.91,26) .. (300.01,26) .. controls (301.12,26) and (302.01,26.9) .. (302.01,28) .. controls (302.01,29.1) and (301.12,30) .. (300.01,30) .. controls (298.91,30) and (298.01,29.1) .. (298.01,28) -- cycle ;
\draw  [draw opacity=0] (256,4) -- (280,4) -- (280,44) -- (256,44) -- cycle ;
\draw  [color={rgb, 255:red, 0; green, 0; blue, 0 }  ,draw opacity=1 ][fill={rgb, 255:red, 255; green, 255; blue, 255 }  ,fill opacity=1 ] (78.01,96) .. controls (78.01,94.9) and (78.9,94) .. (80.01,94) .. controls (81.11,94) and (82.01,94.9) .. (82.01,96) .. controls (82.01,97.1) and (81.11,98) .. (80.01,98) .. controls (78.9,98) and (78.01,97.1) .. (78.01,96) -- cycle ;
\draw  [color={rgb, 255:red, 0; green, 0; blue, 0 }  ,draw opacity=1 ][fill={rgb, 255:red, 0; green, 0; blue, 0 }  ,fill opacity=1 ] (78,80) .. controls (78,78.9) and (78.9,78) .. (80,78) .. controls (81.1,78) and (82,78.9) .. (82,80) .. controls (82,81.1) and (81.1,82) .. (80,82) .. controls (78.9,82) and (78,81.1) .. (78,80) -- cycle ;
\draw  [draw opacity=0] (88,68) -- (112,68) -- (112,108) -- (88,108) -- cycle ;
\draw  [dash pattern={on 4.5pt off 4.5pt}] (112,76) -- (132,76) -- (132,100) -- (112,100) -- cycle ;
\draw    (220,76) -- (220,84) ;
\draw  [color={rgb, 255:red, 0; green, 0; blue, 0 }  ,draw opacity=1 ][fill={rgb, 255:red, 255; green, 255; blue, 255 }  ,fill opacity=1 ] (218,84) .. controls (218,82.9) and (218.9,82) .. (220,82) .. controls (221.1,82) and (222,82.9) .. (222,84) .. controls (222,85.1) and (221.1,86) .. (220,86) .. controls (218.9,86) and (218,85.1) .. (218,84) -- cycle ;
\draw  [color={rgb, 255:red, 0; green, 0; blue, 0 }  ,draw opacity=1 ][fill={rgb, 255:red, 0; green, 0; blue, 0 }  ,fill opacity=1 ] (218,96) .. controls (218,94.9) and (218.9,94) .. (220,94) .. controls (221.1,94) and (222,94.9) .. (222,96) .. controls (222,97.1) and (221.1,98) .. (220,98) .. controls (218.9,98) and (218,97.1) .. (218,96) -- cycle ;
\draw    (220,96) -- (220,104) ;
\draw    (180,76) -- (180,84) ;
\draw  [draw opacity=0] (192,68) -- (216,68) -- (216,108) -- (192,108) -- cycle ;
\draw [color={rgb, 255:red, 0; green, 0; blue, 0 }  ,draw opacity=1 ]   (180,84) .. controls (168.57,83.65) and (167.55,86.38) .. (168,96) ;
\draw  [color={rgb, 255:red, 0; green, 0; blue, 0 }  ,draw opacity=1 ][fill={rgb, 255:red, 255; green, 255; blue, 255 }  ,fill opacity=1 ] (166,96) .. controls (166,94.9) and (166.89,94) .. (168,94) .. controls (169.1,94) and (170,94.9) .. (170,96) .. controls (170,97.1) and (169.1,98) .. (168,98) .. controls (166.89,98) and (166,97.1) .. (166,96) -- cycle ;
\draw  [color={rgb, 255:red, 0; green, 0; blue, 0 }  ,draw opacity=1 ][fill={rgb, 255:red, 0; green, 0; blue, 0 }  ,fill opacity=1 ] (178,84) .. controls (178,82.9) and (178.9,82) .. (180,82) .. controls (181.1,82) and (182,82.9) .. (182,84) .. controls (182,85.1) and (181.1,86) .. (180,86) .. controls (178.9,86) and (178,85.1) .. (178,84) -- cycle ;
\draw    (192,96) -- (192,104) ;
\draw  [draw opacity=0] (96,4) -- (120,4) -- (120,44) -- (96,44) -- cycle ;
\draw  [draw opacity=0] (188,4) -- (212,4) -- (212,44) -- (188,44) -- cycle ;
\draw  [draw opacity=0] (304,4) -- (328,4) -- (328,44) -- (304,44) -- cycle ;
\draw  [draw opacity=0] (224,68) -- (248,68) -- (248,108) -- (224,108) -- cycle ;
\draw  [draw opacity=0] (132,68) -- (156,68) -- (156,108) -- (132,108) -- cycle ;

\draw (56,24) node  [font=\footnotesize]  {$\overset{(i)}{=}$};
\draw (29,46.4) node [anchor=north west][inner sep=0.75pt]  [font=\small]  {$get-get$};
\draw (156,24) node  [font=\footnotesize]  {$\overset{(ii)}{=}$};
\draw (117,46.4) node [anchor=north west][inner sep=0.75pt]  [font=\small]  {$get-discard$};
\draw (268,24) node  [font=\footnotesize]  {$\overset{(iii)}{=}$};
\draw (236,45.4) node [anchor=north west][inner sep=0.75pt]  [font=\small]  {$put-put$};
\draw (100,88) node  [font=\footnotesize]  {$\overset{(iv)}{=}$};
\draw (73,110.4) node [anchor=north west][inner sep=0.75pt]  [font=\small]  {$get-put$};
\draw (204,88) node  [font=\footnotesize]  {$\overset{(v)}{=}$};
\draw (177,109.4) node [anchor=north west][inner sep=0.75pt]  [font=\small]  {$put-get$};
\draw (104,24) node  [font=\footnotesize]  {$;$};
\draw (200,24) node  [font=\footnotesize]  {$;$};
\draw (316,24) node  [font=\footnotesize]  {$;$};
\draw (236,88) node  [font=\footnotesize]  {$;$};
\draw (144,88) node  [font=\footnotesize]  {$;$};

\end{tikzpicture}
   \caption{Non-interchanging diagrams for the axioms of global state.}%
  \label{fig:axioms-noninterchange-state}
\end{figure}

\begin{defi}
  The theory of \emph{global state} is generated by the \effectfulPolygraph{}
  with a single object $X$; two pure generators, $(\blackComultiplication{}) ፡ X
  → X ⊗ X$ and $(\blackComonoidUnit{}) ፡ X → I$, representing $\mathsf{copy}$
  and $\mathsf{discard}$, quotiented by the comonoid axioms; and two effectful
  generators, $(\blackUnit{}) ፡ I → X$ and  $(\whiteComonoidUnit) ፡ X → I$,
  representing $\Get$ and $\Put$ (in \Cref{fig:generators-global-state}),
  quotiented by the equations in \Cref{fig:axioms-global-state}.
\end{defi}

\begin{figure}[!ht]
  \centering

\tikzset{every picture/.style={line width=0.75pt}} %

\begin{tikzpicture}[x=0.75pt,y=0.75pt,yscale=-1,xscale=1]
\draw [color={rgb, 255:red, 191; green, 97; blue, 106 }  ,draw opacity=1 ]   (184,24) -- (184,36) ;
\draw [color={rgb, 255:red, 191; green, 97; blue, 106 }  ,draw opacity=1 ]   (128,24) .. controls (116.57,23.65) and (116.28,25.36) .. (116,36) ;
\draw [color={rgb, 255:red, 191; green, 97; blue, 106 }  ,draw opacity=1 ]   (128,12) -- (128,24) ;
\draw  [color={rgb, 255:red, 0; green, 0; blue, 0 }  ,draw opacity=1 ][fill={rgb, 255:red, 0; green, 0; blue, 0 }  ,fill opacity=1 ] (126,24) .. controls (126,22.9) and (126.9,22) .. (128,22) .. controls (129.1,22) and (130,22.9) .. (130,24) .. controls (130,25.1) and (129.1,26) .. (128,26) .. controls (126.9,26) and (126,25.1) .. (126,24) -- cycle ;
\draw  [draw opacity=0] (144,4) -- (168,4) -- (168,44) -- (144,44) -- cycle ;
\draw [color={rgb, 255:red, 0; green, 0; blue, 0 }  ,draw opacity=1 ]   (128,24) .. controls (139.71,24.79) and (139.43,24.79) .. (140,36) ;
\draw [color={rgb, 255:red, 0; green, 0; blue, 0 }  ,draw opacity=1 ]   (184,24) .. controls (195.71,23.21) and (195.43,23.21) .. (196,12) ;
\draw [color={rgb, 255:red, 191; green, 97; blue, 106 }  ,draw opacity=1 ]   (184,24) .. controls (172.57,24.35) and (172.29,22.64) .. (172,12) ;
\draw  [color={rgb, 255:red, 0; green, 0; blue, 0 }  ,draw opacity=1 ][fill={rgb, 255:red, 255; green, 255; blue, 255 }  ,fill opacity=1 ] (182,24) .. controls (182,25.1) and (182.9,26) .. (184,26) .. controls (185.11,26) and (186,25.1) .. (186,24) .. controls (186,22.9) and (185.11,22) .. (184,22) .. controls (182.9,22) and (182,22.9) .. (182,24) -- cycle ;
\draw  [draw opacity=0] (200,4) -- (224,4) -- (224,44) -- (200,44) -- cycle ;
\draw [color={rgb, 255:red, 0; green, 0; blue, 0 }  ,draw opacity=1 ][fill={rgb, 255:red, 0; green, 0; blue, 0 }  ,fill opacity=1 ]   (80,12) -- (80,24) ;
\draw [color={rgb, 255:red, 0; green, 0; blue, 0 }  ,draw opacity=1 ]   (32,24) .. controls (20.56,23.65) and (20.28,25.36) .. (19.99,36) ;
\draw [color={rgb, 255:red, 0; green, 0; blue, 0 }  ,draw opacity=1 ]   (32,12) -- (32,24) ;
\draw  [color={rgb, 255:red, 0; green, 0; blue, 0 }  ,draw opacity=1 ][fill={rgb, 255:red, 0; green, 0; blue, 0 }  ,fill opacity=1 ] (30,24) .. controls (30,22.9) and (30.89,22) .. (32,22) .. controls (33.1,22) and (34,22.9) .. (34,24) .. controls (34,25.1) and (33.1,26) .. (32,26) .. controls (30.89,26) and (30,25.1) .. (30,24) -- cycle ;
\draw  [draw opacity=0] (48,4) -- (72,4) -- (72,44) -- (48,44) -- cycle ;
\draw [color={rgb, 255:red, 0; green, 0; blue, 0 }  ,draw opacity=1 ]   (32,24) .. controls (43.71,24.79) and (43.42,24.79) .. (44,36) ;
\draw  [color={rgb, 255:red, 0; green, 0; blue, 0 }  ,draw opacity=1 ][fill={rgb, 255:red, 0; green, 0; blue, 0 }  ,fill opacity=1 ] (78,24) .. controls (78,25.1) and (78.9,26) .. (80,26) .. controls (81.11,26) and (82,25.1) .. (82,24) .. controls (82,22.9) and (81.11,22) .. (80,22) .. controls (78.9,22) and (78,22.9) .. (78,24) -- cycle ;
\draw  [draw opacity=0] (96,4) -- (120,4) -- (120,44) -- (96,44) -- cycle ;

\draw (156,24) node  [font=\footnotesize]  {$;$};
\draw (212,24) node  [font=\footnotesize]  {$;$};
\draw (60,24) node  [font=\footnotesize]  {$;$};
\draw (100,24) node  [font=\footnotesize]  {$;$};

\end{tikzpicture}
   \caption{Generators for global state.}
  \label{fig:generators-global-state}
\end{figure}

These two $\Put$ and $\Get$ generators, without extra axioms, represent what
happens when a process can \emph{send} or \emph{receive} resources from an
unrestricted environment. Right now, we are concerned with the theory of a
global state accessed by a single process: we will impose the following axioms,
which assert that the global state was not modified by anyone but this single
process (\Cref{fig:axioms-global-state}).

\begin{figure}[!ht]

\tikzset{every picture/.style={line width=0.75pt}} %

\begin{tikzpicture}[x=0.75pt,y=0.75pt,yscale=-1,xscale=1]
\draw [color={rgb, 255:red, 191; green, 97; blue, 106 }  ,draw opacity=1 ]   (324,80) -- (324,92) ;
\draw [color={rgb, 255:red, 191; green, 97; blue, 106 }  ,draw opacity=1 ]   (248,92) -- (248,104) ;
\draw [color={rgb, 255:red, 191; green, 97; blue, 106 }  ,draw opacity=1 ]   (248,92) .. controls (236.57,92.35) and (236.28,90.64) .. (236,80) ;
\draw [color={rgb, 255:red, 0; green, 0; blue, 0 }  ,draw opacity=1 ]   (248,92) .. controls (259.99,92.23) and (259.99,91.63) .. (259.99,84) ;
\draw [color={rgb, 255:red, 191; green, 97; blue, 106 }  ,draw opacity=1 ]   (100.01,96) -- (100.01,104) ;
\draw [color={rgb, 255:red, 191; green, 97; blue, 106 }  ,draw opacity=1 ]   (100,80) .. controls (88.57,79.65) and (88.4,80.23) .. (88,88) ;
\draw [color={rgb, 255:red, 191; green, 97; blue, 106 }  ,draw opacity=1 ]   (100,72) -- (100,80) ;
\draw    (388,24) -- (388,36) ;
\draw [color={rgb, 255:red, 0; green, 0; blue, 0 }  ,draw opacity=1 ]   (300.01,32) .. controls (311.72,31.21) and (311.43,31.21) .. (312,20) ;
\draw [color={rgb, 255:red, 0; green, 0; blue, 0 }  ,draw opacity=1 ]   (312,20) .. controls (300.57,20.35) and (300.29,18.64) .. (300,8) ;
\draw [color={rgb, 255:red, 0; green, 0; blue, 0 }  ,draw opacity=1 ]   (312,20) .. controls (323.71,19.21) and (323.43,19.21) .. (324,8) ;
\draw [color={rgb, 255:red, 0; green, 0; blue, 0 }  ,draw opacity=1 ]   (252.01,32) .. controls (263.72,31.21) and (263.44,31.21) .. (264.01,20) ;
\draw [color={rgb, 255:red, 0; green, 0; blue, 0 }  ,draw opacity=1 ]   (240,20) .. controls (251.71,19.21) and (251.43,19.21) .. (252,8) ;
\draw [color={rgb, 255:red, 191; green, 97; blue, 106 }  ,draw opacity=1 ]   (82,16) .. controls (70.57,15.65) and (68.29,17.36) .. (68,28) ;
\draw [color={rgb, 255:red, 191; green, 97; blue, 106 }  ,draw opacity=1 ]   (80,8) -- (80,16) ;
\draw [color={rgb, 255:red, 191; green, 97; blue, 106 }  ,draw opacity=1 ]   (20,28) .. controls (8.56,27.65) and (8.28,29.36) .. (7.99,40) ;
\draw [color={rgb, 255:red, 191; green, 97; blue, 106 }  ,draw opacity=1 ]   (32,16) .. controls (20.57,15.65) and (20.28,17.36) .. (20,28) ;
\draw [color={rgb, 255:red, 191; green, 97; blue, 106 }  ,draw opacity=1 ]   (32,8) -- (32,16) ;
\draw [color={rgb, 255:red, 0; green, 0; blue, 0 }  ,draw opacity=1 ]   (272.09,76) .. controls (284.31,75.71) and (283.64,78.38) .. (284.09,88) ;
\draw  [color={rgb, 255:red, 0; green, 0; blue, 0 }  ,draw opacity=1 ][fill={rgb, 255:red, 0; green, 0; blue, 0 }  ,fill opacity=1 ] (18,28) .. controls (18,26.9) and (18.89,26) .. (20,26) .. controls (21.1,26) and (22,26.9) .. (22,28) .. controls (22,29.1) and (21.1,30) .. (20,30) .. controls (18.89,30) and (18,29.1) .. (18,28) -- cycle ;
\draw    (44,28) -- (44,40) ;
\draw  [color={rgb, 255:red, 0; green, 0; blue, 0 }  ,draw opacity=1 ][fill={rgb, 255:red, 0; green, 0; blue, 0 }  ,fill opacity=1 ] (30,16) .. controls (30,14.9) and (30.9,14) .. (32,14) .. controls (33.1,14) and (34,14.9) .. (34,16) .. controls (34,17.1) and (33.1,18) .. (32,18) .. controls (30.9,18) and (30,17.1) .. (30,16) -- cycle ;
\draw  [draw opacity=0] (44,4) -- (68,4) -- (68,44) -- (44,44) -- cycle ;
\draw  [color={rgb, 255:red, 0; green, 0; blue, 0 }  ,draw opacity=1 ][fill={rgb, 255:red, 0; green, 0; blue, 0 }  ,fill opacity=1 ] (78,16) .. controls (78,14.9) and (78.9,14) .. (80,14) .. controls (81.1,14) and (82,14.9) .. (82,16) .. controls (82,17.1) and (81.1,18) .. (80,18) .. controls (78.9,18) and (78,17.1) .. (78,16) -- cycle ;
\draw  [color={rgb, 255:red, 0; green, 0; blue, 0 }  ,draw opacity=1 ][fill={rgb, 255:red, 0; green, 0; blue, 0 }  ,fill opacity=1 ] (90,28) .. controls (90,26.9) and (90.89,26) .. (92,26) .. controls (93.1,26) and (94,26.9) .. (94,28) .. controls (94,29.1) and (93.1,30) .. (92,30) .. controls (90.89,30) and (90,29.1) .. (90,28) -- cycle ;
\draw [color={rgb, 255:red, 0; green, 0; blue, 0 }  ,draw opacity=1 ]   (92,28) .. controls (103.71,28.79) and (103.42,28.79) .. (103.99,40) ;
\draw [color={rgb, 255:red, 0; green, 0; blue, 0 }  ,draw opacity=1 ]   (92,28) .. controls (80.56,27.65) and (80.28,29.36) .. (79.99,40) ;
\draw  [color={rgb, 255:red, 0; green, 0; blue, 0 }  ,draw opacity=1 ][fill={rgb, 255:red, 0; green, 0; blue, 0 }  ,fill opacity=1 ] (154,32) .. controls (154,30.9) and (154.89,30) .. (156,30) .. controls (157.1,30) and (158,30.9) .. (158,32) .. controls (158,33.1) and (157.1,34) .. (156,34) .. controls (154.89,34) and (154,33.1) .. (154,32) -- cycle ;
\draw  [color={rgb, 255:red, 0; green, 0; blue, 0 }  ,draw opacity=1 ][fill={rgb, 255:red, 0; green, 0; blue, 0 }  ,fill opacity=1 ] (142,20) .. controls (142,18.9) and (142.9,18) .. (144,18) .. controls (145.1,18) and (146,18.9) .. (146,20) .. controls (146,21.1) and (145.1,22) .. (144,22) .. controls (142.9,22) and (142,21.1) .. (142,20) -- cycle ;
\draw  [draw opacity=0] (160,4) -- (184,4) -- (184,44) -- (160,44) -- cycle ;
\draw  [color={rgb, 255:red, 0; green, 0; blue, 0 }  ,draw opacity=1 ][fill={rgb, 255:red, 0; green, 0; blue, 0 }  ,fill opacity=1 ] (426,20) .. controls (426,18.9) and (426.9,18) .. (428,18) .. controls (429.11,18) and (430,18.9) .. (430,20) .. controls (430,21.1) and (429.11,22) .. (428,22) .. controls (426.9,22) and (426,21.1) .. (426,20) -- cycle ;
\draw    (428,12) -- (428,20) ;
\draw    (443.99,12) -- (444,16) ;
\draw  [color={rgb, 255:red, 0; green, 0; blue, 0 }  ,draw opacity=1 ][fill={rgb, 255:red, 0; green, 0; blue, 0 }  ,fill opacity=1 ] (98,80) .. controls (98,78.9) and (98.9,78) .. (100,78) .. controls (101.1,78) and (102,78.9) .. (102,80) .. controls (102,81.1) and (101.1,82) .. (100,82) .. controls (98.9,82) and (98,81.1) .. (98,80) -- cycle ;
\draw  [draw opacity=0] (112,68) -- (136,68) -- (136,108) -- (112,108) -- cycle ;
\draw    (272.09,68) -- (272.09,76) ;
\draw  [draw opacity=0] (286,69) -- (310,69) -- (310,109) -- (286,109) -- cycle ;
\draw [color={rgb, 255:red, 0; green, 0; blue, 0 }  ,draw opacity=1 ]   (272.09,76) .. controls (260.65,75.65) and (259.5,77.81) .. (260,84) ;
\draw  [color={rgb, 255:red, 0; green, 0; blue, 0 }  ,draw opacity=1 ][fill={rgb, 255:red, 255; green, 255; blue, 255 }  ,fill opacity=1 ] (246,92) .. controls (246,90.9) and (246.9,90) .. (248,90) .. controls (249.1,90) and (250,90.9) .. (250,92) .. controls (250,93.1) and (249.1,94) .. (248,94) .. controls (246.9,94) and (246,93.1) .. (246,92) -- cycle ;
\draw  [color={rgb, 255:red, 0; green, 0; blue, 0 }  ,draw opacity=1 ][fill={rgb, 255:red, 0; green, 0; blue, 0 }  ,fill opacity=1 ] (270.09,76) .. controls (270.09,74.9) and (270.98,74) .. (272.09,74) .. controls (273.19,74) and (274.09,74.9) .. (274.09,76) .. controls (274.09,77.1) and (273.19,78) .. (272.09,78) .. controls (270.98,78) and (270.09,77.1) .. (270.09,76) -- cycle ;
\draw    (284,88) -- (284,104) ;
\draw  [draw opacity=0] (104,4) -- (128,4) -- (128,44) -- (104,44) -- cycle ;
\draw  [draw opacity=0] (196,4) -- (220,4) -- (220,44) -- (196,44) -- cycle ;
\draw  [draw opacity=0] (344,68) -- (368,68) -- (368,108) -- (344,108) -- cycle ;
\draw  [draw opacity=0] (144,68) -- (168,68) -- (168,108) -- (144,108) -- cycle ;
\draw [color={rgb, 255:red, 0; green, 0; blue, 0 }  ,draw opacity=1 ]   (32,16) .. controls (43.71,16.79) and (43.43,16.79) .. (44,28) ;
\draw [color={rgb, 255:red, 0; green, 0; blue, 0 }  ,draw opacity=1 ]   (20,28) .. controls (31.71,28.79) and (31.42,28.79) .. (31.99,40) ;
\draw [color={rgb, 255:red, 0; green, 0; blue, 0 }  ,draw opacity=1 ]   (80,16) .. controls (91.71,16.79) and (91.42,16.79) .. (92,28) ;
\draw [color={rgb, 255:red, 191; green, 97; blue, 106 }  ,draw opacity=1 ]   (68,28) -- (68,40) ;
\draw [color={rgb, 255:red, 191; green, 97; blue, 106 }  ,draw opacity=1 ]   (146,20) .. controls (134.57,19.65) and (132.29,21.36) .. (132,32) ;
\draw [color={rgb, 255:red, 191; green, 97; blue, 106 }  ,draw opacity=1 ]   (144,8) -- (144,20) ;
\draw  [color={rgb, 255:red, 0; green, 0; blue, 0 }  ,draw opacity=1 ][fill={rgb, 255:red, 0; green, 0; blue, 0 }  ,fill opacity=1 ] (142,20) .. controls (142,18.9) and (142.9,18) .. (144,18) .. controls (145.1,18) and (146,18.9) .. (146,20) .. controls (146,21.1) and (145.1,22) .. (144,22) .. controls (142.9,22) and (142,21.1) .. (142,20) -- cycle ;
\draw [color={rgb, 255:red, 0; green, 0; blue, 0 }  ,draw opacity=1 ]   (144,20) .. controls (155.71,20.79) and (155.42,20.79) .. (156,32) ;
\draw [color={rgb, 255:red, 191; green, 97; blue, 106 }  ,draw opacity=1 ]   (132,32) -- (132,40) ;
\draw [color={rgb, 255:red, 191; green, 97; blue, 106 }  ,draw opacity=1 ]   (192,8) -- (192,40) ;
\draw [color={rgb, 255:red, 191; green, 97; blue, 106 }  ,draw opacity=1 ]   (302.01,32) .. controls (290.58,32.35) and (288.29,30.64) .. (288.01,20) ;
\draw [color={rgb, 255:red, 191; green, 97; blue, 106 }  ,draw opacity=1 ]   (300.01,40) -- (300.01,32) ;
\draw [color={rgb, 255:red, 191; green, 97; blue, 106 }  ,draw opacity=1 ]   (240,20) .. controls (228.57,20.35) and (228.29,18.64) .. (228,8) ;
\draw [color={rgb, 255:red, 191; green, 97; blue, 106 }  ,draw opacity=1 ]   (252.01,32) .. controls (240.58,32.35) and (240.29,30.64) .. (240,20) ;
\draw [color={rgb, 255:red, 191; green, 97; blue, 106 }  ,draw opacity=1 ]   (252.01,40) -- (252.01,32) ;
\draw  [color={rgb, 255:red, 0; green, 0; blue, 0 }  ,draw opacity=1 ][fill={rgb, 255:red, 255; green, 255; blue, 255 }  ,fill opacity=1 ] (238,20) .. controls (238,21.1) and (238.9,22) .. (240,22) .. controls (241.11,22) and (242,21.1) .. (242,20) .. controls (242,18.9) and (241.11,18) .. (240,18) .. controls (238.9,18) and (238,18.9) .. (238,20) -- cycle ;
\draw    (264.01,20) -- (264.01,8) ;
\draw  [color={rgb, 255:red, 0; green, 0; blue, 0 }  ,draw opacity=1 ][fill={rgb, 255:red, 255; green, 255; blue, 255 }  ,fill opacity=1 ] (250.01,32) .. controls (250.01,33.1) and (250.9,34) .. (252.01,34) .. controls (253.11,34) and (254.01,33.1) .. (254.01,32) .. controls (254.01,30.9) and (253.11,30) .. (252.01,30) .. controls (250.9,30) and (250.01,30.9) .. (250.01,32) -- cycle ;
\draw  [draw opacity=0] (264.01,44) -- (288.01,44) -- (288.01,4) -- (264.01,4) -- cycle ;
\draw  [color={rgb, 255:red, 0; green, 0; blue, 0 }  ,draw opacity=1 ][fill={rgb, 255:red, 255; green, 255; blue, 255 }  ,fill opacity=1 ] (298.01,32) .. controls (298.01,33.1) and (298.9,34) .. (300.01,34) .. controls (301.11,34) and (302.01,33.1) .. (302.01,32) .. controls (302.01,30.9) and (301.11,30) .. (300.01,30) .. controls (298.9,30) and (298.01,30.9) .. (298.01,32) -- cycle ;
\draw  [color={rgb, 255:red, 0; green, 0; blue, 0 }  ,draw opacity=1 ][fill={rgb, 255:red, 255; green, 255; blue, 255 }  ,fill opacity=1 ] (310,20) .. controls (310,21.1) and (310.9,22) .. (312,22) .. controls (313.11,22) and (314,21.1) .. (314,20) .. controls (314,18.9) and (313.11,18) .. (312,18) .. controls (310.9,18) and (310,18.9) .. (310,20) -- cycle ;
\draw [color={rgb, 255:red, 191; green, 97; blue, 106 }  ,draw opacity=1 ]   (288.01,20) -- (288.01,8) ;
\draw [color={rgb, 255:red, 0; green, 0; blue, 0 }  ,draw opacity=1 ]   (388,24) .. controls (376.57,24.35) and (376.29,22.64) .. (376,12) ;
\draw [color={rgb, 255:red, 0; green, 0; blue, 0 }  ,draw opacity=1 ]   (388,24) .. controls (399.71,23.21) and (399.43,23.21) .. (400,12) ;
\draw  [color={rgb, 255:red, 0; green, 0; blue, 0 }  ,draw opacity=1 ][fill={rgb, 255:red, 255; green, 255; blue, 255 }  ,fill opacity=1 ] (386,24) .. controls (386,25.1) and (386.9,26) .. (388,26) .. controls (389.11,26) and (390,25.1) .. (390,24) .. controls (390,22.9) and (389.11,22) .. (388,22) .. controls (386.9,22) and (386,22.9) .. (386,24) -- cycle ;
\draw  [draw opacity=0] (400,4) -- (424,4) -- (424,44) -- (400,44) -- cycle ;
\draw    (444,16) .. controls (444.5,27.56) and (436.25,22.31) .. (436,36) ;
\draw  [draw opacity=0] (448,4) -- (472,4) -- (472,44) -- (448,44) -- cycle ;
\draw [color={rgb, 255:red, 0; green, 0; blue, 0 }  ,draw opacity=1 ]   (100,80) .. controls (111.71,80.79) and (111.8,79.43) .. (112,88) ;
\draw [color={rgb, 255:red, 0; green, 0; blue, 0 }  ,draw opacity=1 ]   (100.01,96) .. controls (112,96.23) and (112,95.63) .. (112,88) ;
\draw [color={rgb, 255:red, 191; green, 97; blue, 106 }  ,draw opacity=1 ]   (88,88) .. controls (88,96.63) and (88.2,95.63) .. (100.01,96) ;
\draw  [color={rgb, 255:red, 0; green, 0; blue, 0 }  ,draw opacity=1 ][fill={rgb, 255:red, 255; green, 255; blue, 255 }  ,fill opacity=1 ] (98.01,96) .. controls (98.01,94.9) and (98.9,94) .. (100.01,94) .. controls (101.11,94) and (102.01,94.9) .. (102.01,96) .. controls (102.01,97.1) and (101.11,98) .. (100.01,98) .. controls (98.9,98) and (98.01,97.1) .. (98.01,96) -- cycle ;
\draw [color={rgb, 255:red, 191; green, 97; blue, 106 }  ,draw opacity=1 ]   (140,72) -- (140,104) ;
\draw [color={rgb, 255:red, 191; green, 97; blue, 106 }  ,draw opacity=1 ]   (236,68) -- (236,80) ;
\draw [color={rgb, 255:red, 0; green, 0; blue, 0 }  ,draw opacity=1 ]   (324,80) .. controls (335.99,80.23) and (335.99,75.63) .. (335.99,68) ;
\draw [color={rgb, 255:red, 191; green, 97; blue, 106 }  ,draw opacity=1 ]   (324,80) .. controls (312.57,80.35) and (312.28,78.64) .. (312,68) ;
\draw  [color={rgb, 255:red, 0; green, 0; blue, 0 }  ,draw opacity=1 ][fill={rgb, 255:red, 255; green, 255; blue, 255 }  ,fill opacity=1 ] (322,80) .. controls (322,78.9) and (322.9,78) .. (324,78) .. controls (325.1,78) and (326,78.9) .. (326,80) .. controls (326,81.1) and (325.1,82) .. (324,82) .. controls (322.9,82) and (322,81.1) .. (322,80) -- cycle ;
\draw [color={rgb, 255:red, 0; green, 0; blue, 0 }  ,draw opacity=1 ]   (324,92) .. controls (336.22,91.71) and (335.56,94.38) .. (336,104) ;
\draw [color={rgb, 255:red, 191; green, 97; blue, 106 }  ,draw opacity=1 ]   (324,92) .. controls (312.57,91.65) and (312,94.31) .. (312,104) ;
\draw  [color={rgb, 255:red, 0; green, 0; blue, 0 }  ,draw opacity=1 ][fill={rgb, 255:red, 0; green, 0; blue, 0 }  ,fill opacity=1 ] (322,92) .. controls (322,90.9) and (322.9,90) .. (324,90) .. controls (325.1,90) and (326,90.9) .. (326,92) .. controls (326,93.1) and (325.1,94) .. (324,94) .. controls (322.9,94) and (322,93.1) .. (322,92) -- cycle ;

\draw (56,24) node  [font=\footnotesize]  {$\overset{(i)}{=}$};
\draw (17,45.4) node [anchor=north west][inner sep=0.75pt]  [font=\small]  {$comonoid\ ( runtime\ coaction)$};
\draw (170,23.5) node  [font=\footnotesize]  {$\overset{(ii)}{=}$};
\draw (231,45.4) node [anchor=north west][inner sep=0.75pt]  [font=\small]  {$semimonoid\ ( runtime\ semiaction)$};
\draw (124,88) node  [font=\footnotesize]  {$\overset{(iv)}{=}$};
\draw (101,109.4) node [anchor=north west][inner sep=0.75pt]  [font=\small]  {$special$};
\draw (298,89) node  [font=\footnotesize]  {$\overset{(v)}{=}$};
\draw (265,110.4) node [anchor=north west][inner sep=0.75pt]  [font=\small]  {$Frobenius$};
\draw (116,24) node  [font=\footnotesize]  {$;$};
\draw (208,24) node  [font=\footnotesize]  {$;$};
\draw (356,88) node  [font=\footnotesize]  {$;$};
\draw (156,88) node  [font=\footnotesize]  {$;$};
\draw (276.01,24) node  [font=\footnotesize]  {$\overset{(iii)}{=}$};
\draw (333,18.4) node [anchor=north west][inner sep=0.75pt]  [font=\small]  {$with$};
\draw (412,24) node  [font=\footnotesize]  {${=}$};
\draw (460,24) node  [font=\footnotesize]  {$;$};

\end{tikzpicture}
   \caption{Effectful string diagrams for the axioms of global state.}
  \label{fig:axioms-global-state}
\end{figure}

The equations in \Cref{fig:axioms-global-state} (and
\Cref{fig:axioms-noninterchange-state}) say that: (\emph{i}) reading the global
state twice gets us the same result twice, (\emph{ii}) reading the global state
and discarding the result is the same as doing nothing, (\emph{iii}) writing
twice to the global state keeps only the last thing that was written,
(\emph{iv}) reading something and immediately writing it to the global state is
the same as doing nothing, (\emph{v}) keeping a copy of something that we write
to global state is the same as writing it to the global state and then reading
it.

\begin{rem}[Global state is semi-Frobenius]
  The theory in \Cref{fig:axioms-global-state} can be summarized as a
  \emph{semi-Frobenius action} on the runtime $\R$, where the semimonoid is the
  \emph{left-absorbing semigroup}\footnote{Left-absorbing semigroups are those
  satisfying the equation $x \cdot y = x$ \cite{kilp11:monoids}; we use the name
  \emph{semimonoid} and interpret them more generally in a symmetric monoidal
  category.}. The \emph{get-get} and \emph{get-discard} laws correspond to a
  comonoid action on the runtime; the \emph{put-put} law correspond to a
  semimonoid action of a left-absorbing semigroup; the \emph{get-put} law
  correspond to the special axiom; and the \emph{put-get} law corresponds to the
  Frobenius axiom.

  While this algebra exists also in \Cref{fig:axioms-noninterchange-state}, it
  only becomes apparent when we use the string diagrams of
  \effectfulCategories{}, in \Cref{fig:axioms-global-state}.
\end{rem}

\begin{prop}[Race conditions]
  Concurrently mixing two processes that share a global state, $f$ and $g$, can
  produce four possible results: (i) only the result of the first one is
  preserved, (ii) only the result of the second one is preserved, or
  (iii,iv) the composition of both is preserved, in any order.
\end{prop}
\begin{proof}
  Formally, we work in the theory of global state adding two processes, $f ፡ X →
  X$ and $g ፡ X → X$, that can moreover be discarded, meaning $f ⨾
  \blackComonoidUnit{} = g ⨾ \blackComonoidUnit{} = \blackComonoidUnit{}$. We
  employ string diagrams for \effectfulCategories{}. The following three
  diagrams in \Cref{fig:raceConditions} correspond to the first three cases of
  race conditions; the last one is analogous to the third.
\end{proof}

\begin{figure}[!ht]

\tikzset{every picture/.style={line width=0.75pt}} %

\begin{tikzpicture}[x=0.75pt,y=0.75pt,yscale=-1,xscale=1]
\draw [color={rgb, 255:red, 191; green, 97; blue, 106 }  ,draw opacity=1 ]   (224.01,148) -- (224.01,156) ;
\draw [color={rgb, 255:red, 191; green, 97; blue, 106 }  ,draw opacity=1 ]   (151.95,144) .. controls (151.95,152.63) and (152.15,151.63) .. (163.96,152) ;
\draw [color={rgb, 255:red, 191; green, 97; blue, 106 }  ,draw opacity=1 ]   (248,72) -- (248,80) ;
\draw [color={rgb, 255:red, 191; green, 97; blue, 106 }  ,draw opacity=1 ]   (236.01,64) .. controls (236.01,72.63) and (236.21,71.63) .. (248.01,72) ;
\draw [color={rgb, 255:red, 191; green, 97; blue, 106 }  ,draw opacity=1 ]   (224,56) .. controls (224,64.63) and (224.2,63.63) .. (236.01,64) ;
\draw [color={rgb, 255:red, 191; green, 97; blue, 106 }  ,draw opacity=1 ]   (28.01,68) -- (28.01,76) ;
\draw [color={rgb, 255:red, 0; green, 0; blue, 0 }  ,draw opacity=1 ]   (28.01,68) .. controls (40.01,68.23) and (40.01,67.63) .. (40.01,60) ;
\draw [color={rgb, 255:red, 191; green, 97; blue, 106 }  ,draw opacity=1 ]   (16.01,60) .. controls (16.01,68.63) and (16.21,67.63) .. (28.01,68) ;
\draw  [color={rgb, 255:red, 0; green, 0; blue, 0 }  ,draw opacity=1 ][fill={rgb, 255:red, 255; green, 255; blue, 255 }  ,fill opacity=1 ] (26.01,68) .. controls (26.01,66.9) and (26.91,66) .. (28.01,66) .. controls (29.12,66) and (30.01,66.9) .. (30.01,68) .. controls (30.01,69.1) and (29.12,70) .. (28.01,70) .. controls (26.91,70) and (26.01,69.1) .. (26.01,68) -- cycle ;
\draw [color={rgb, 255:red, 191; green, 97; blue, 106 }  ,draw opacity=1 ]   (16,28) .. controls (4.57,27.65) and (4.29,29.36) .. (4,40) ;
\draw [color={rgb, 255:red, 191; green, 97; blue, 106 }  ,draw opacity=1 ]   (28.01,16) .. controls (16.58,15.65) and (16.29,17.36) .. (16,28) ;
\draw [color={rgb, 255:red, 191; green, 97; blue, 106 }  ,draw opacity=1 ]   (28.01,8) -- (28.01,16) ;
\draw  [color={rgb, 255:red, 0; green, 0; blue, 0 }  ,draw opacity=1 ][fill={rgb, 255:red, 0; green, 0; blue, 0 }  ,fill opacity=1 ] (14,28) .. controls (14,26.9) and (14.9,26) .. (16,26) .. controls (17.11,26) and (18,26.9) .. (18,28) .. controls (18,29.1) and (17.11,30) .. (16,30) .. controls (14.9,30) and (14,29.1) .. (14,28) -- cycle ;
\draw  [color={rgb, 255:red, 0; green, 0; blue, 0 }  ,draw opacity=1 ][fill={rgb, 255:red, 0; green, 0; blue, 0 }  ,fill opacity=1 ] (26.01,16) .. controls (26.01,14.9) and (26.9,14) .. (28.01,14) .. controls (29.11,14) and (30.01,14.9) .. (30.01,16) .. controls (30.01,17.1) and (29.11,18) .. (28.01,18) .. controls (26.9,18) and (26.01,17.1) .. (26.01,16) -- cycle ;
\draw [color={rgb, 255:red, 0; green, 0; blue, 0 }  ,draw opacity=1 ]   (28.01,16) .. controls (40.23,16.38) and (40.23,14.6) .. (40.01,24) ;
\draw [color={rgb, 255:red, 0; green, 0; blue, 0 }  ,draw opacity=1 ]   (16,28) .. controls (27.71,28.79) and (23.44,28.79) .. (24.01,40) ;
\draw  [draw opacity=0] (176,20) -- (200,20) -- (200,60) -- (176,60) -- cycle ;
\draw   (12.01,40) -- (36.01,40) -- (36.01,52) -- (12.01,52) -- cycle ;
\draw   (28.01,24) -- (52.01,24) -- (52.01,36) -- (28.01,36) -- cycle ;
\draw [color={rgb, 255:red, 191; green, 97; blue, 106 }  ,draw opacity=1 ]   (4.01,40) -- (4.01,52) ;
\draw [color={rgb, 255:red, 0; green, 0; blue, 0 }  ,draw opacity=1 ]   (16.01,60) .. controls (28.01,60.23) and (24.01,59.63) .. (24.01,52) ;
\draw [color={rgb, 255:red, 191; green, 97; blue, 106 }  ,draw opacity=1 ]   (4.01,52) .. controls (4.01,60.63) and (4.21,59.63) .. (16.01,60) ;
\draw  [color={rgb, 255:red, 0; green, 0; blue, 0 }  ,draw opacity=1 ][fill={rgb, 255:red, 255; green, 255; blue, 255 }  ,fill opacity=1 ] (14.01,60) .. controls (14.01,58.9) and (14.91,58) .. (16.01,58) .. controls (17.12,58) and (18.01,58.9) .. (18.01,60) .. controls (18.01,61.1) and (17.12,62) .. (16.01,62) .. controls (14.91,62) and (14.01,61.1) .. (14.01,60) -- cycle ;
\draw    (40.01,36) -- (40.01,60) ;
\draw [color={rgb, 255:red, 191; green, 97; blue, 106 }  ,draw opacity=1 ]   (96.01,68) -- (96.01,76) ;
\draw [color={rgb, 255:red, 0; green, 0; blue, 0 }  ,draw opacity=1 ]   (96.01,68) .. controls (108.01,68.23) and (108.01,67.63) .. (108.01,60) ;
\draw [color={rgb, 255:red, 191; green, 97; blue, 106 }  ,draw opacity=1 ]   (80,60) .. controls (80,68.63) and (84.21,67.63) .. (96.01,68) ;
\draw  [color={rgb, 255:red, 0; green, 0; blue, 0 }  ,draw opacity=1 ][fill={rgb, 255:red, 255; green, 255; blue, 255 }  ,fill opacity=1 ] (94.01,68) .. controls (94.01,66.9) and (94.91,66) .. (96.01,66) .. controls (97.12,66) and (98.01,66.9) .. (98.01,68) .. controls (98.01,69.1) and (97.12,70) .. (96.01,70) .. controls (94.91,70) and (94.01,69.1) .. (94.01,68) -- cycle ;
\draw [color={rgb, 255:red, 191; green, 97; blue, 106 }  ,draw opacity=1 ]   (84,28) .. controls (72.57,27.65) and (72.29,29.36) .. (72,40) ;
\draw [color={rgb, 255:red, 191; green, 97; blue, 106 }  ,draw opacity=1 ]   (96.01,16) .. controls (84.58,15.65) and (84.29,17.36) .. (84,28) ;
\draw [color={rgb, 255:red, 191; green, 97; blue, 106 }  ,draw opacity=1 ]   (96.01,8) -- (96.01,16) ;
\draw  [color={rgb, 255:red, 0; green, 0; blue, 0 }  ,draw opacity=1 ][fill={rgb, 255:red, 0; green, 0; blue, 0 }  ,fill opacity=1 ] (82,28) .. controls (82,26.9) and (82.9,26) .. (84,26) .. controls (85.11,26) and (86,26.9) .. (86,28) .. controls (86,29.1) and (85.11,30) .. (84,30) .. controls (82.9,30) and (82,29.1) .. (82,28) -- cycle ;
\draw  [color={rgb, 255:red, 0; green, 0; blue, 0 }  ,draw opacity=1 ][fill={rgb, 255:red, 0; green, 0; blue, 0 }  ,fill opacity=1 ] (94.01,16) .. controls (94.01,14.9) and (94.9,14) .. (96.01,14) .. controls (97.11,14) and (98.01,14.9) .. (98.01,16) .. controls (98.01,17.1) and (97.11,18) .. (96.01,18) .. controls (94.9,18) and (94.01,17.1) .. (94.01,16) -- cycle ;
\draw [color={rgb, 255:red, 0; green, 0; blue, 0 }  ,draw opacity=1 ]   (96.01,16) .. controls (108.23,16.38) and (108.23,14.6) .. (108.01,24) ;
\draw [color={rgb, 255:red, 0; green, 0; blue, 0 }  ,draw opacity=1 ]   (84,28) .. controls (95.71,28.79) and (91.44,28.79) .. (92.01,40) ;
\draw   (80.01,40) -- (104.01,40) -- (104.01,52) -- (80.01,52) -- cycle ;
\draw   (96.01,24) -- (120.01,24) -- (120.01,36) -- (96.01,36) -- cycle ;
\draw    (108.01,36) -- (108.01,60) ;
\draw [color={rgb, 255:red, 0; green, 0; blue, 0 }  ,draw opacity=1 ]   (92,52) -- (92,60) ;
\draw  [color={rgb, 255:red, 0; green, 0; blue, 0 }  ,draw opacity=1 ][fill={rgb, 255:red, 0; green, 0; blue, 0 }  ,fill opacity=1 ] (90,60) .. controls (90,58.9) and (90.9,58) .. (92,58) .. controls (93.1,58) and (94,58.9) .. (94,60) .. controls (94,61.1) and (93.1,62) .. (92,62) .. controls (90.9,62) and (90,61.1) .. (90,60) -- cycle ;
\draw [color={rgb, 255:red, 191; green, 97; blue, 106 }  ,draw opacity=1 ]   (72,40) .. controls (72.5,51.85) and (80,52.35) .. (80,60) ;
\draw [color={rgb, 255:red, 191; green, 97; blue, 106 }  ,draw opacity=1 ]   (152.01,56) -- (152.01,64) ;
\draw [color={rgb, 255:red, 0; green, 0; blue, 0 }  ,draw opacity=1 ]   (152.01,56) .. controls (164,56.23) and (164,55.63) .. (164,48) ;
\draw [color={rgb, 255:red, 191; green, 97; blue, 106 }  ,draw opacity=1 ]   (140,48) .. controls (140,56.63) and (140.2,55.63) .. (152.01,56) ;
\draw  [color={rgb, 255:red, 0; green, 0; blue, 0 }  ,draw opacity=1 ][fill={rgb, 255:red, 255; green, 255; blue, 255 }  ,fill opacity=1 ] (150.01,56) .. controls (150.01,54.9) and (150.9,54) .. (152.01,54) .. controls (153.11,54) and (154.01,54.9) .. (154.01,56) .. controls (154.01,57.1) and (153.11,58) .. (152.01,58) .. controls (150.9,58) and (150.01,57.1) .. (150.01,56) -- cycle ;
\draw [color={rgb, 255:red, 191; green, 97; blue, 106 }  ,draw opacity=1 ]   (152,28) .. controls (140.57,27.65) and (140.29,27.93) .. (140,36) ;
\draw [color={rgb, 255:red, 191; green, 97; blue, 106 }  ,draw opacity=1 ]   (152,20) -- (152,28) ;
\draw  [color={rgb, 255:red, 0; green, 0; blue, 0 }  ,draw opacity=1 ][fill={rgb, 255:red, 0; green, 0; blue, 0 }  ,fill opacity=1 ] (150,28) .. controls (150,26.9) and (150.9,26) .. (152,26) .. controls (153.1,26) and (154,26.9) .. (154,28) .. controls (154,29.1) and (153.1,30) .. (152,30) .. controls (150.9,30) and (150,29.1) .. (150,28) -- cycle ;
\draw [color={rgb, 255:red, 0; green, 0; blue, 0 }  ,draw opacity=1 ]   (152,28) .. controls (164.22,28.38) and (164.22,26.6) .. (164,36) ;
\draw   (152,36) -- (176,36) -- (176,48) -- (152,48) -- cycle ;
\draw [color={rgb, 255:red, 191; green, 97; blue, 106 }  ,draw opacity=1 ]   (140,36) -- (140,48) ;
\draw [color={rgb, 255:red, 0; green, 0; blue, 0 }  ,draw opacity=1 ]   (236.01,64) .. controls (248,64.23) and (248,64.7) .. (248,60) ;
\draw  [color={rgb, 255:red, 0; green, 0; blue, 0 }  ,draw opacity=1 ][fill={rgb, 255:red, 255; green, 255; blue, 255 }  ,fill opacity=1 ] (234.01,64) .. controls (234.01,62.9) and (234.9,62) .. (236.01,62) .. controls (237.11,62) and (238.01,62.9) .. (238.01,64) .. controls (238.01,65.1) and (237.11,66) .. (236.01,66) .. controls (234.9,66) and (234.01,65.1) .. (234.01,64) -- cycle ;
\draw [color={rgb, 255:red, 191; green, 97; blue, 106 }  ,draw opacity=1 ]   (232,28) .. controls (220.56,27.65) and (220.28,29.36) .. (219.99,40) ;
\draw [color={rgb, 255:red, 191; green, 97; blue, 106 }  ,draw opacity=1 ]   (244,16) .. controls (232.57,15.65) and (232.28,17.36) .. (232,28) ;
\draw [color={rgb, 255:red, 191; green, 97; blue, 106 }  ,draw opacity=1 ]   (244,8) -- (244,16) ;
\draw  [color={rgb, 255:red, 0; green, 0; blue, 0 }  ,draw opacity=1 ][fill={rgb, 255:red, 0; green, 0; blue, 0 }  ,fill opacity=1 ] (230,28) .. controls (230,26.9) and (230.89,26) .. (232,26) .. controls (233.1,26) and (234,26.9) .. (234,28) .. controls (234,29.1) and (233.1,30) .. (232,30) .. controls (230.89,30) and (230,29.1) .. (230,28) -- cycle ;
\draw  [color={rgb, 255:red, 0; green, 0; blue, 0 }  ,draw opacity=1 ][fill={rgb, 255:red, 0; green, 0; blue, 0 }  ,fill opacity=1 ] (242,16) .. controls (242,14.9) and (242.9,14) .. (244,14) .. controls (245.1,14) and (246,14.9) .. (246,16) .. controls (246,17.1) and (245.1,18) .. (244,18) .. controls (242.9,18) and (242,17.1) .. (242,16) -- cycle ;
\draw [color={rgb, 255:red, 0; green, 0; blue, 0 }  ,draw opacity=1 ]   (244,16) .. controls (256.22,16.38) and (256.22,14.6) .. (256,24) ;
\draw [color={rgb, 255:red, 0; green, 0; blue, 0 }  ,draw opacity=1 ]   (232,28) .. controls (243.71,28.79) and (239.43,28.79) .. (240,40) ;
\draw   (228,40) -- (252,40) -- (252,52) -- (228,52) -- cycle ;
\draw   (244,24) -- (268,24) -- (268,36) -- (244,36) -- cycle ;
\draw    (256,36) -- (256,44) ;
\draw [color={rgb, 255:red, 191; green, 97; blue, 106 }  ,draw opacity=1 ]   (364.01,56) -- (364.01,64) ;
\draw [color={rgb, 255:red, 0; green, 0; blue, 0 }  ,draw opacity=1 ]   (364.01,56) .. controls (376,56.23) and (376,55.63) .. (376,48) ;
\draw [color={rgb, 255:red, 191; green, 97; blue, 106 }  ,draw opacity=1 ]   (352,48) .. controls (352,56.63) and (352.2,55.63) .. (364.01,56) ;
\draw  [color={rgb, 255:red, 0; green, 0; blue, 0 }  ,draw opacity=1 ][fill={rgb, 255:red, 255; green, 255; blue, 255 }  ,fill opacity=1 ] (362.01,56) .. controls (362.01,54.9) and (362.9,54) .. (364.01,54) .. controls (365.11,54) and (366.01,54.9) .. (366.01,56) .. controls (366.01,57.1) and (365.11,58) .. (364.01,58) .. controls (362.9,58) and (362.01,57.1) .. (362.01,56) -- cycle ;
\draw [color={rgb, 255:red, 191; green, 97; blue, 106 }  ,draw opacity=1 ]   (364,28) .. controls (352.57,27.65) and (352.29,27.93) .. (352,36) ;
\draw [color={rgb, 255:red, 191; green, 97; blue, 106 }  ,draw opacity=1 ]   (364,20) -- (364,28) ;
\draw  [color={rgb, 255:red, 0; green, 0; blue, 0 }  ,draw opacity=1 ][fill={rgb, 255:red, 0; green, 0; blue, 0 }  ,fill opacity=1 ] (362,28) .. controls (362,26.9) and (362.9,26) .. (364,26) .. controls (365.1,26) and (366,26.9) .. (366,28) .. controls (366,29.1) and (365.1,30) .. (364,30) .. controls (362.9,30) and (362,29.1) .. (362,28) -- cycle ;
\draw [color={rgb, 255:red, 0; green, 0; blue, 0 }  ,draw opacity=1 ]   (364,28) .. controls (376.22,28.38) and (376.22,26.6) .. (376,36) ;
\draw   (364,36) -- (388,36) -- (388,48) -- (364,48) -- cycle ;
\draw [color={rgb, 255:red, 191; green, 97; blue, 106 }  ,draw opacity=1 ]   (352,36) -- (352,48) ;
\draw [color={rgb, 255:red, 191; green, 97; blue, 106 }  ,draw opacity=1 ]   (220,40) .. controls (220.5,51.85) and (224,48.35) .. (224,56) ;
\draw [color={rgb, 255:red, 0; green, 0; blue, 0 }  ,draw opacity=1 ]   (256,44) .. controls (256.5,55.85) and (248.2,53.1) .. (248,60) ;
\draw [color={rgb, 255:red, 0; green, 0; blue, 0 }  ,draw opacity=1 ]   (240,52) .. controls (239.8,63.7) and (260.01,56.35) .. (260.01,64) ;
\draw [color={rgb, 255:red, 0; green, 0; blue, 0 }  ,draw opacity=1 ]   (248.01,72) .. controls (260.01,72.23) and (260.01,71.63) .. (260.01,64) ;
\draw  [color={rgb, 255:red, 0; green, 0; blue, 0 }  ,draw opacity=1 ][fill={rgb, 255:red, 255; green, 255; blue, 255 }  ,fill opacity=1 ] (246.01,72) .. controls (246.01,70.9) and (246.91,70) .. (248.01,70) .. controls (249.12,70) and (250.01,70.9) .. (250.01,72) .. controls (250.01,73.1) and (249.12,74) .. (248.01,74) .. controls (246.91,74) and (246.01,73.1) .. (246.01,72) -- cycle ;
\draw [color={rgb, 255:red, 191; green, 97; blue, 106 }  ,draw opacity=1 ]   (296.01,64) -- (296.01,72) ;
\draw [color={rgb, 255:red, 191; green, 97; blue, 106 }  ,draw opacity=1 ]   (284,56) .. controls (284,64.63) and (284.2,63.63) .. (296.01,64) ;
\draw [color={rgb, 255:red, 191; green, 97; blue, 106 }  ,draw opacity=1 ]   (300,32) .. controls (288.56,31.65) and (288.28,33.36) .. (287.99,44) ;
\draw [color={rgb, 255:red, 191; green, 97; blue, 106 }  ,draw opacity=1 ]   (312,20) .. controls (300.57,19.65) and (300.28,21.36) .. (300,32) ;
\draw [color={rgb, 255:red, 191; green, 97; blue, 106 }  ,draw opacity=1 ]   (312,12) -- (312,20) ;
\draw  [color={rgb, 255:red, 0; green, 0; blue, 0 }  ,draw opacity=1 ][fill={rgb, 255:red, 0; green, 0; blue, 0 }  ,fill opacity=1 ] (298,32) .. controls (298,30.9) and (298.89,30) .. (300,30) .. controls (301.1,30) and (302,30.9) .. (302,32) .. controls (302,33.1) and (301.1,34) .. (300,34) .. controls (298.89,34) and (298,33.1) .. (298,32) -- cycle ;
\draw  [color={rgb, 255:red, 0; green, 0; blue, 0 }  ,draw opacity=1 ][fill={rgb, 255:red, 0; green, 0; blue, 0 }  ,fill opacity=1 ] (310,20) .. controls (310,18.9) and (310.9,18) .. (312,18) .. controls (313.1,18) and (314,18.9) .. (314,20) .. controls (314,21.1) and (313.1,22) .. (312,22) .. controls (310.9,22) and (310,21.1) .. (310,20) -- cycle ;
\draw [color={rgb, 255:red, 0; green, 0; blue, 0 }  ,draw opacity=1 ]   (312,20) .. controls (324.22,20.38) and (324.22,18.6) .. (324,28) ;
\draw [color={rgb, 255:red, 0; green, 0; blue, 0 }  ,draw opacity=1 ]   (300,32) .. controls (311.71,32.79) and (307.43,32.79) .. (308,44) ;
\draw   (296,44) -- (320,44) -- (320,56) -- (296,56) -- cycle ;
\draw   (312,28) -- (336,28) -- (336,40) -- (312,40) -- cycle ;
\draw    (324,40) -- (324,48) ;
\draw [color={rgb, 255:red, 191; green, 97; blue, 106 }  ,draw opacity=1 ]   (287.99,44) .. controls (288.44,49.49) and (284,48.35) .. (284,56) ;
\draw [color={rgb, 255:red, 0; green, 0; blue, 0 }  ,draw opacity=1 ]   (296.01,64) .. controls (308,64.23) and (308,63.63) .. (308,56) ;
\draw  [color={rgb, 255:red, 0; green, 0; blue, 0 }  ,draw opacity=1 ][fill={rgb, 255:red, 255; green, 255; blue, 255 }  ,fill opacity=1 ] (294.01,64) .. controls (294.01,62.9) and (294.9,62) .. (296.01,62) .. controls (297.11,62) and (298.01,62.9) .. (298.01,64) .. controls (298.01,65.1) and (297.11,66) .. (296.01,66) .. controls (294.9,66) and (294.01,65.1) .. (294.01,64) -- cycle ;
\draw  [color={rgb, 255:red, 0; green, 0; blue, 0 }  ,draw opacity=1 ][fill={rgb, 255:red, 0; green, 0; blue, 0 }  ,fill opacity=1 ] (322,48) .. controls (322,46.9) and (322.9,46) .. (324,46) .. controls (325.1,46) and (326,46.9) .. (326,48) .. controls (326,49.1) and (325.1,50) .. (324,50) .. controls (322.9,50) and (322,49.1) .. (322,48) -- cycle ;
\draw  [draw opacity=0] (116,24) -- (140,24) -- (140,64) -- (116,64) -- cycle ;
\draw  [draw opacity=0] (48,24) -- (72,24) -- (72,64) -- (48,64) -- cycle ;
\draw  [draw opacity=0] (388,20) -- (412,20) -- (412,60) -- (388,60) -- cycle ;
\draw [color={rgb, 255:red, 191; green, 97; blue, 106 }  ,draw opacity=1 ]   (96.01,116) -- (96.01,124) ;
\draw [color={rgb, 255:red, 0; green, 0; blue, 0 }  ,draw opacity=1 ]   (96.01,116) .. controls (108,116.23) and (108,115.63) .. (108,108) ;
\draw [color={rgb, 255:red, 191; green, 97; blue, 106 }  ,draw opacity=1 ]   (84,108) .. controls (84,116.63) and (84.2,115.63) .. (96.01,116) ;
\draw  [color={rgb, 255:red, 0; green, 0; blue, 0 }  ,draw opacity=1 ][fill={rgb, 255:red, 255; green, 255; blue, 255 }  ,fill opacity=1 ] (94.01,116) .. controls (94.01,114.9) and (94.9,114) .. (96.01,114) .. controls (97.11,114) and (98.01,114.9) .. (98.01,116) .. controls (98.01,117.1) and (97.11,118) .. (96.01,118) .. controls (94.9,118) and (94.01,117.1) .. (94.01,116) -- cycle ;
\draw [color={rgb, 255:red, 191; green, 97; blue, 106 }  ,draw opacity=1 ]   (96,88) .. controls (84.57,87.65) and (84.29,87.93) .. (84,96) ;
\draw [color={rgb, 255:red, 191; green, 97; blue, 106 }  ,draw opacity=1 ]   (96,80) -- (96,88) ;
\draw  [color={rgb, 255:red, 0; green, 0; blue, 0 }  ,draw opacity=1 ][fill={rgb, 255:red, 0; green, 0; blue, 0 }  ,fill opacity=1 ] (94,88) .. controls (94,86.9) and (94.9,86) .. (96,86) .. controls (97.1,86) and (98,86.9) .. (98,88) .. controls (98,89.1) and (97.1,90) .. (96,90) .. controls (94.9,90) and (94,89.1) .. (94,88) -- cycle ;
\draw [color={rgb, 255:red, 0; green, 0; blue, 0 }  ,draw opacity=1 ]   (96,88) .. controls (108.22,88.38) and (108.22,86.6) .. (108,96) ;
\draw   (96,96) -- (120,96) -- (120,108) -- (96,108) -- cycle ;
\draw [color={rgb, 255:red, 191; green, 97; blue, 106 }  ,draw opacity=1 ]   (84,96) -- (84,108) ;
\draw [color={rgb, 255:red, 191; green, 97; blue, 106 }  ,draw opacity=1 ]   (96.01,152) -- (96.01,160) ;
\draw [color={rgb, 255:red, 0; green, 0; blue, 0 }  ,draw opacity=1 ]   (96.01,152) .. controls (108,152.23) and (108,151.63) .. (108,144) ;
\draw [color={rgb, 255:red, 191; green, 97; blue, 106 }  ,draw opacity=1 ]   (84,144) .. controls (84,152.63) and (84.2,151.63) .. (96.01,152) ;
\draw  [color={rgb, 255:red, 0; green, 0; blue, 0 }  ,draw opacity=1 ][fill={rgb, 255:red, 255; green, 255; blue, 255 }  ,fill opacity=1 ] (94.01,152) .. controls (94.01,150.9) and (94.9,150) .. (96.01,150) .. controls (97.11,150) and (98.01,150.9) .. (98.01,152) .. controls (98.01,153.1) and (97.11,154) .. (96.01,154) .. controls (94.9,154) and (94.01,153.1) .. (94.01,152) -- cycle ;
\draw [color={rgb, 255:red, 191; green, 97; blue, 106 }  ,draw opacity=1 ]   (96,124) .. controls (84.57,123.65) and (84.29,123.93) .. (84,132) ;
\draw  [color={rgb, 255:red, 0; green, 0; blue, 0 }  ,draw opacity=1 ][fill={rgb, 255:red, 0; green, 0; blue, 0 }  ,fill opacity=1 ] (94,124) .. controls (94,122.9) and (94.9,122) .. (96,122) .. controls (97.1,122) and (98,122.9) .. (98,124) .. controls (98,125.1) and (97.1,126) .. (96,126) .. controls (94.9,126) and (94,125.1) .. (94,124) -- cycle ;
\draw [color={rgb, 255:red, 0; green, 0; blue, 0 }  ,draw opacity=1 ]   (96,124) .. controls (108.22,124.38) and (108.22,122.6) .. (108,132) ;
\draw   (96,132) -- (120,132) -- (120,144) -- (96,144) -- cycle ;
\draw [color={rgb, 255:red, 191; green, 97; blue, 106 }  ,draw opacity=1 ]   (84,132) -- (84,144) ;
\draw [color={rgb, 255:red, 191; green, 97; blue, 106 }  ,draw opacity=1 ]   (139.95,128) -- (139.95,136) ;
\draw [color={rgb, 255:red, 191; green, 97; blue, 106 }  ,draw opacity=1 ]   (163.95,88) .. controls (152.52,87.65) and (152.24,87.93) .. (151.95,96) ;
\draw [color={rgb, 255:red, 191; green, 97; blue, 106 }  ,draw opacity=1 ]   (163.95,80) -- (163.95,88) ;
\draw  [color={rgb, 255:red, 0; green, 0; blue, 0 }  ,draw opacity=1 ][fill={rgb, 255:red, 0; green, 0; blue, 0 }  ,fill opacity=1 ] (161.95,88) .. controls (161.95,86.9) and (162.85,86) .. (163.95,86) .. controls (165.05,86) and (165.95,86.9) .. (165.95,88) .. controls (165.95,89.1) and (165.05,90) .. (163.95,90) .. controls (162.85,90) and (161.95,89.1) .. (161.95,88) -- cycle ;
\draw [color={rgb, 255:red, 0; green, 0; blue, 0 }  ,draw opacity=1 ]   (163.95,88) .. controls (176.17,88.38) and (176.17,86.6) .. (175.95,96) ;
\draw   (163.95,96) -- (187.95,96) -- (187.95,108) -- (163.95,108) -- cycle ;
\draw [color={rgb, 255:red, 191; green, 97; blue, 106 }  ,draw opacity=1 ]   (151.95,96) -- (151.95,108) ;
\draw [color={rgb, 255:red, 0; green, 0; blue, 0 }  ,draw opacity=1 ]   (151.96,144) .. controls (163.95,144.23) and (163.95,143.63) .. (163.95,136) ;
\draw [color={rgb, 255:red, 191; green, 97; blue, 106 }  ,draw opacity=1 ]   (139.95,136) .. controls (139.95,144.63) and (140.15,143.63) .. (151.96,144) ;
\draw  [color={rgb, 255:red, 0; green, 0; blue, 0 }  ,draw opacity=1 ][fill={rgb, 255:red, 255; green, 255; blue, 255 }  ,fill opacity=1 ] (149.96,144) .. controls (149.96,142.9) and (150.85,142) .. (151.96,142) .. controls (153.06,142) and (153.96,142.9) .. (153.96,144) .. controls (153.96,145.1) and (153.06,146) .. (151.96,146) .. controls (150.85,146) and (149.96,145.1) .. (149.96,144) -- cycle ;
\draw   (151.95,124) -- (175.95,124) -- (175.95,136) -- (151.95,136) -- cycle ;
\draw [color={rgb, 255:red, 0; green, 0; blue, 0 }  ,draw opacity=1 ]   (175.96,108) -- (175.96,116) ;
\draw [color={rgb, 255:red, 0; green, 0; blue, 0 }  ,draw opacity=1 ]   (175.95,116) .. controls (164.52,115.65) and (164.24,115.93) .. (163.95,124) ;
\draw  [color={rgb, 255:red, 0; green, 0; blue, 0 }  ,draw opacity=1 ][fill={rgb, 255:red, 0; green, 0; blue, 0 }  ,fill opacity=1 ] (173.95,116) .. controls (173.95,114.9) and (174.85,114) .. (175.95,114) .. controls (177.05,114) and (177.95,114.9) .. (177.95,116) .. controls (177.95,117.1) and (177.05,118) .. (175.95,118) .. controls (174.85,118) and (173.95,117.1) .. (173.95,116) -- cycle ;
\draw [color={rgb, 255:red, 0; green, 0; blue, 0 }  ,draw opacity=1 ]   (175.95,116) .. controls (188.17,116.38) and (188.17,114.6) .. (187.95,124) ;
\draw [color={rgb, 255:red, 191; green, 97; blue, 106 }  ,draw opacity=1 ]   (151.95,108) .. controls (152.45,119.85) and (139.78,116.38) .. (139.95,128) ;
\draw [color={rgb, 255:red, 191; green, 97; blue, 106 }  ,draw opacity=1 ]   (163.96,152) -- (163.96,160) ;
\draw [color={rgb, 255:red, 0; green, 0; blue, 0 }  ,draw opacity=1 ]   (163.96,152) .. controls (175.95,152.23) and (175.95,151.63) .. (175.95,144) ;
\draw  [color={rgb, 255:red, 0; green, 0; blue, 0 }  ,draw opacity=1 ][fill={rgb, 255:red, 255; green, 255; blue, 255 }  ,fill opacity=1 ] (161.96,152) .. controls (161.96,150.9) and (162.85,150) .. (163.96,150) .. controls (165.06,150) and (165.96,150.9) .. (165.96,152) .. controls (165.96,153.1) and (165.06,154) .. (163.96,154) .. controls (162.85,154) and (161.96,153.1) .. (161.96,152) -- cycle ;
\draw [color={rgb, 255:red, 0; green, 0; blue, 0 }  ,draw opacity=1 ]   (187.95,124) .. controls (188.45,135.85) and (175.95,136.35) .. (175.95,144) ;
\draw  [draw opacity=0] (188,100) -- (212,100) -- (212,140) -- (188,140) -- cycle ;
\draw [color={rgb, 255:red, 191; green, 97; blue, 106 }  ,draw opacity=1 ]   (236,92) .. controls (224.57,91.65) and (224.29,91.93) .. (224,100) ;
\draw [color={rgb, 255:red, 191; green, 97; blue, 106 }  ,draw opacity=1 ]   (236,84) -- (236,92) ;
\draw  [color={rgb, 255:red, 0; green, 0; blue, 0 }  ,draw opacity=1 ][fill={rgb, 255:red, 0; green, 0; blue, 0 }  ,fill opacity=1 ] (234,92) .. controls (234,90.9) and (234.9,90) .. (236,90) .. controls (237.11,90) and (238,90.9) .. (238,92) .. controls (238,93.1) and (237.11,94) .. (236,94) .. controls (234.9,94) and (234,93.1) .. (234,92) -- cycle ;
\draw [color={rgb, 255:red, 0; green, 0; blue, 0 }  ,draw opacity=1 ]   (236.01,92) .. controls (248.22,92.38) and (248.22,90.6) .. (248,100) ;
\draw   (236,100) -- (260,100) -- (260,112) -- (236,112) -- cycle ;
\draw [color={rgb, 255:red, 191; green, 97; blue, 106 }  ,draw opacity=1 ]   (224,100) -- (224.01,112) ;
\draw [color={rgb, 255:red, 0; green, 0; blue, 0 }  ,draw opacity=1 ]   (224.01,148) .. controls (236,148.23) and (236,147.63) .. (236,140) ;
\draw [color={rgb, 255:red, 191; green, 97; blue, 106 }  ,draw opacity=1 ]   (212,140) .. controls (212,148.63) and (212.2,147.63) .. (224.01,148) ;
\draw  [color={rgb, 255:red, 0; green, 0; blue, 0 }  ,draw opacity=1 ][fill={rgb, 255:red, 255; green, 255; blue, 255 }  ,fill opacity=1 ] (222.01,148) .. controls (222.01,146.9) and (222.9,146) .. (224.01,146) .. controls (225.11,146) and (226.01,146.9) .. (226.01,148) .. controls (226.01,149.1) and (225.11,150) .. (224.01,150) .. controls (222.9,150) and (222.01,149.1) .. (222.01,148) -- cycle ;
\draw   (224,128) -- (248,128) -- (248,140) -- (224,140) -- cycle ;
\draw [color={rgb, 255:red, 0; green, 0; blue, 0 }  ,draw opacity=1 ]   (248.01,112) -- (248.01,120) ;
\draw [color={rgb, 255:red, 0; green, 0; blue, 0 }  ,draw opacity=1 ]   (248,120) .. controls (236.57,119.65) and (236.29,119.93) .. (236,128) ;
\draw  [color={rgb, 255:red, 0; green, 0; blue, 0 }  ,draw opacity=1 ][fill={rgb, 255:red, 0; green, 0; blue, 0 }  ,fill opacity=1 ] (246,120) .. controls (246,118.9) and (246.9,118) .. (248,118) .. controls (249.11,118) and (250,118.9) .. (250,120) .. controls (250,121.1) and (249.11,122) .. (248,122) .. controls (246.9,122) and (246,121.1) .. (246,120) -- cycle ;
\draw [color={rgb, 255:red, 0; green, 0; blue, 0 }  ,draw opacity=1 ]   (248.01,120) .. controls (260.22,120.38) and (260.22,118.6) .. (260,128) ;
\draw [color={rgb, 255:red, 191; green, 97; blue, 106 }  ,draw opacity=1 ]   (224,112) .. controls (224.5,123.85) and (211.83,120.38) .. (212,132) ;
\draw  [draw opacity=0] (188,32) -- (212,32) -- (212,72) -- (188,72) -- cycle ;
\draw  [color={rgb, 255:red, 0; green, 0; blue, 0 }  ,draw opacity=1 ][fill={rgb, 255:red, 0; green, 0; blue, 0 }  ,fill opacity=1 ] (258,128) .. controls (258,126.9) and (258.9,126) .. (260,126) .. controls (261.11,126) and (262,126.9) .. (262,128) .. controls (262,129.1) and (261.11,130) .. (260,130) .. controls (258.9,130) and (258,129.1) .. (258,128) -- cycle ;
\draw [color={rgb, 255:red, 191; green, 97; blue, 106 }  ,draw opacity=1 ]   (212,132) -- (212,140) ;
\draw [color={rgb, 255:red, 191; green, 97; blue, 106 }  ,draw opacity=1 ]   (296.01,140) -- (296.01,148) ;
\draw [color={rgb, 255:red, 0; green, 0; blue, 0 }  ,draw opacity=1 ]   (296.01,140) .. controls (308,140.23) and (308,139.63) .. (308,132) ;
\draw [color={rgb, 255:red, 191; green, 97; blue, 106 }  ,draw opacity=1 ]   (284,132) .. controls (284,140.63) and (284.2,139.63) .. (296.01,140) ;
\draw  [color={rgb, 255:red, 0; green, 0; blue, 0 }  ,draw opacity=1 ][fill={rgb, 255:red, 255; green, 255; blue, 255 }  ,fill opacity=1 ] (294.01,140) .. controls (294.01,138.9) and (294.9,138) .. (296.01,138) .. controls (297.11,138) and (298.01,138.9) .. (298.01,140) .. controls (298.01,141.1) and (297.11,142) .. (296.01,142) .. controls (294.9,142) and (294.01,141.1) .. (294.01,140) -- cycle ;
\draw [color={rgb, 255:red, 191; green, 97; blue, 106 }  ,draw opacity=1 ]   (296,96) .. controls (284.57,95.65) and (284.29,95.93) .. (284,104) ;
\draw [color={rgb, 255:red, 191; green, 97; blue, 106 }  ,draw opacity=1 ]   (296,88) -- (296,96) ;
\draw  [color={rgb, 255:red, 0; green, 0; blue, 0 }  ,draw opacity=1 ][fill={rgb, 255:red, 0; green, 0; blue, 0 }  ,fill opacity=1 ] (294,96) .. controls (294,94.9) and (294.9,94) .. (296,94) .. controls (297.1,94) and (298,94.9) .. (298,96) .. controls (298,97.1) and (297.1,98) .. (296,98) .. controls (294.9,98) and (294,97.1) .. (294,96) -- cycle ;
\draw [color={rgb, 255:red, 0; green, 0; blue, 0 }  ,draw opacity=1 ]   (296,96) .. controls (308.22,96.38) and (308.22,94.6) .. (308,104) ;
\draw   (296,104) -- (320,104) -- (320,116) -- (296,116) -- cycle ;
\draw   (296,120) -- (320,120) -- (320,132) -- (296,132) -- cycle ;
\draw [color={rgb, 255:red, 0; green, 0; blue, 0 }  ,draw opacity=1 ]   (308,116) -- (308,120) ;
\draw [color={rgb, 255:red, 191; green, 97; blue, 106 }  ,draw opacity=1 ]   (284,104) -- (284,132) ;
\draw  [draw opacity=0] (260,100) -- (284,100) -- (284,140) -- (260,140) -- cycle ;
\draw  [draw opacity=0] (320,100) -- (344,100) -- (344,140) -- (320,140) -- cycle ;

\draw (188,40) node  [font=\footnotesize]  {$;$};
\draw (40.01,30) node  [font=\footnotesize]  {$f$};
\draw (24.01,46) node  [font=\footnotesize]  {$g$};
\draw (108.01,30) node  [font=\footnotesize]  {$f$};
\draw (92.01,46) node  [font=\footnotesize]  {$g$};
\draw (164,42) node  [font=\footnotesize]  {$f$};
\draw (256,30) node  [font=\footnotesize]  {$f$};
\draw (240,46) node  [font=\footnotesize]  {$g$};
\draw (376,42) node  [font=\footnotesize]  {$g$};
\draw (324,34) node  [font=\footnotesize]  {$f$};
\draw (308,50) node  [font=\footnotesize]  {$g$};
\draw (128,44) node  [font=\footnotesize]  {$=$};
\draw (60,44) node  [font=\footnotesize]  {$=$};
\draw (400,40) node  [font=\footnotesize]  {$;$};
\draw (108,102) node  [font=\footnotesize]  {$f$};
\draw (108,138) node  [font=\footnotesize]  {$g$};
\draw (175.95,102) node  [font=\footnotesize]  {$f$};
\draw (163.95,130) node  [font=\footnotesize]  {$g$};
\draw (128,120) node  [font=\footnotesize]  {$=$};
\draw (248,106) node  [font=\footnotesize]  {$f$};
\draw (236,134) node  [font=\footnotesize]  {$g$};
\draw (200,120) node  [font=\footnotesize]  {$=$};
\draw (308,110) node  [font=\footnotesize]  {$f$};
\draw (308,126) node  [font=\footnotesize]  {$g$};
\draw (272,120) node  [font=\footnotesize]  {$=$};
\draw (332,120) node  [font=\footnotesize]  {$;$};

\end{tikzpicture}
   \caption{Possible results of a race condition.}%
  \label{fig:raceConditions}
\end{figure}

\section{String Diagrams for Premonoidal Categories}
\label{sec:stringPremonoidal}
We have constructed string diagrams for \effectfulCategories{} via an
adjunction (\Cref{theorem:runtime-as-a-resource}); and as a consequence, we will also have constructed string diagrams
for \premonoidalCategories{}. At the beginning of this text, we left open the
question of what were the correct functors between \premonoidalCategories{}: we
now choose to assemble them into a category where functors preserve the central
morphisms. This is not as satisfactory as working directly with
\effectfulCategories{}, but we do so for completeness; we refer the reader to
the work of Staton and Levy for a detailed discussion of why
\effectfulCategories{} may be a better notion than \premonoidalCategories{}
\cite{staton13}.

\begin{defi}%
  \defining{linkStrictPremon}{}%
  Strict \premonoidalCategories{} and central functors between them form a
  category, $\PremonStr$, endowed with a fully faithful functor $\mathsf{J}
  \colon \PremonStr \to \EffCatStr$ that acts by sending a
  \premonoidalCategory{} $ℂ$ to the \effectfulCategory{} given by the inclusion
  of its center, $\zentre(ℂ) → ℂ$.
\end{defi}

On the syntax side, we will use the fact that every \polygraph{}---every
signature of a \premonoidalCategory{}---can be read as an \effectfulPolygraph{}
with no pure generators.

\begin{prop}%
  \label{lemma:adjointpolygraphcouple}%
  There exists an adjunction between \polygraphs{} and \polygraphCouples{}. The
  left adjoint is given by the functor interpreting every edge as an effectful
  one, $\mathsf{I} \colon \PolyGraph \to \EffPolyGraph$. The right adjoint
  functor forgets the pure edges, $\mathsf{V} \colon \EffPolyGraph \to
  \PolyGraph$.
\end{prop}

On the semantics side, we may observe that the free \effectfulCategory{} over a
\polygraph{} constructed by the previous runtime construction, $\mathsf{R}
\colon \EffPolyGraph \to \EffCatStr$, is actually a \premonoidalCategory{} that
has a trivial centre consisting only of identities.

\begin{lem}%
  \label{lemma:freepremonoidalfactors}%
  The functor $\mathsf{I} \comp \mathsf{R} ፡ \PolyGraph → \EffCatStr$ factors
  through the inclusion of \premonoidalCategories{} as $\mathsf{I} \comp
  \mathsf{R}^\ast \comp \mathsf{J}$ for a codomain restriction $\mathsf{R}^\ast
  \colon \PolyGraph → \PremonStr$.
\end{lem}

\begin{cor}%
  \label{cor:adjunctioneffectful}%
  Finally, from the previous section, we extract that there is an adjunction
  between \polygraphCouples{} and \effectfulCategories{} with effectful functors
  given by the runtime construction of the free effectful category $\mathsf{R}
  \colon \EffPolyGraph \to \EffCatStr$ and the forgetful functor $\mathsf{U}
  \colon \EffCatStr \to \EffPolyGraph$.
\end{cor}

\begin{thm}%
  \label{thm:adjunction-string-premonoidals}%
  There is an adjunction between \premonoidalCategories{} and \polygraphs{}
  given by the string diagrams of \effectfulCategories{} with every node needing
  the runtime.
\end{thm}
\begin{proof}%
  We show the adjunction by a natural bijection of morphism sets. Let
  $\mathcal{G}$ be any \polygraph{} and let $\mathbb{P}$ be any
  \premonoidalCategory{}. The right adjoint will be $\mathsf{U}_\bullet =
  \mathsf{V} \comp \mathsf{U} \comp \mathsf{J}$; the left adjoint will be
  $\mathsf{I} \comp \mathsf{R^\ast}$, as obtained in Lemma
  \ref{lemma:freepremonoidalfactors}.
  \begin{align*}
    \PolyGraph(\mathcal{G}; \mathsf{U}_\bullet \mathbb{P}) & \cong
    \PolyGraph(\mathcal{G}; \mathsf{VUJ} \mathbb{P}) \overset{(i)}\cong
    \EffPolyGraph(\mathsf{I}\mathcal{G}; \mathsf{UJ} \mathbb{P}) \overset{(ii)}\cong
    \EffCatStr(\mathsf{RI}\mathcal{G}; \mathsf{J} \mathbb{P}) \\ & \overset{(iii)}\cong
    \EffCatStr(\mathsf{J}\mathsf{R}^\ast\mathsf{I}\mathcal{G}; \mathsf{J} \mathbb{P}) \overset{(iv)}\cong
    \PremonStr(\mathsf{R}^\ast\mathsf{I}\mathcal{G}; \mathbb{P}).
  \end{align*}%
  Here, we have used that \emph{(i)} the adjunction between \polygraphs{} and
  \polygraphCouples{} by Lemma \ref{lemma:adjointpolygraphcouple}; \emph{(ii)}
  the adjunction between \polygraphCouples{} and \effectfulCategories{} in
  Corollary \ref{cor:adjunctioneffectful}; \emph{(iii)} the factorization in \Cref{lemma:freepremonoidalfactors}; and \emph{(iv)} that
  \premonoidalCategories{} embed fully-faithfully into \effectfulCategories{}.
\end{proof}

\section{Conclusions}

\PremonoidalCategories{} are monoidal categories with runtime, and we can still use monoidal string diagrams and unrestricted topological deformations to reason about them.
Instead of dealing directly with \premonoidalCategories{}, we employ the better-behaved notion of non-cartesian Freyd categories, \effectfulCategories{}. There exists a more fine-grained notion of ``Cartesian effect category'' \cite{dumas11}, which generalizes Freyd categories
and justifies using “\effectfulCategory{}” for the general case.

Ultimately, this is a first step towards our more ambitious project of presenting the categorical structure of programming languages in a purely diagrammatic way, revisiting Alan Jeffrey's work \cite{jeffrey97,jeffrey1997premonoidal,marionotes2022}.
The internal language of premonoidal categories and effectful categories is given by the \emph{arrow do-notation} \cite{paterson01}; at the same time, we have shown that it is given by suitable string diagrams.
This correspondence allows us to translate between programs and string diagrams (\Cref{fig:premonoidalprogram}).

\begin{exa}
  This example illustrates a ``Hello world'', in diagrams and do-notation.
\begin{figure}[H]
  \centering
  \begin{minipage}{0.60\textwidth}
      \begin{Verbatim}[commandchars=\\\{\}]
  \colorPre{proc} () \colorPre{-> do}
    question \colorMon{<- "What's your name?"}
    () \colorPre{<-} \colorPre{print} \colorPre{-<} question
    name \colorPre{<-} \colorPre{get} \colorPre{-<} ()
    greeting \colorMon{<-} \colorMon{"Hello"}
    output \colorMon{<- concatenate}(greeting, name)
    () \colorPre{<- print -<} output
  \colorPre{return} ()
      \end{Verbatim}
  \end{minipage}
  \begin{minipage}{0.35\textwidth}

\tikzset{every picture/.style={line width=0.75pt}} %

\begin{tikzpicture}[x=0.75pt,y=0.75pt,yscale=-1,xscale=1]
\draw [color={rgb, 255:red, 0; green, 0; blue, 0 }  ,draw opacity=1 ]   (100,40) .. controls (100.17,48.17) and (91.83,43.83) .. (92,52) ;
\draw  [draw opacity=1 ] (80,24) -- (120,24) -- (120,40) -- (80,40) -- cycle ;
\draw  [color={rgb, 255:red, 191; green, 97; blue, 106 }  ,draw opacity=1 ] (64,52) -- (104,52) -- (104,68) -- (64,68) -- cycle ;
\draw [color={rgb, 255:red, 191; green, 97; blue, 106 }  ,draw opacity=1 ]   (64,8) -- (64,32) ;
\draw [color={rgb, 255:red, 191; green, 97; blue, 106 }  ,draw opacity=1 ]   (64,32) .. controls (64.17,43.83) and (75.83,39.83) .. (76,52) ;
\draw [color={rgb, 255:red, 191; green, 97; blue, 106 }  ,draw opacity=1 ]   (84,68) -- (84,76) ;
\draw  [draw opacity=1 ] (26,24) -- (62,24) -- (62,40) -- (26,40) -- cycle ;
\draw  [color={rgb, 255:red, 191; green, 97; blue, 106 }  ,draw opacity=1 ] (64,76) -- (104,76) -- (104,92) -- (64,92) -- cycle ;
\draw [color={rgb, 255:red, 0; green, 0; blue, 0 }  ,draw opacity=1 ]   (44,40) -- (44,92) ;
\draw  [draw opacity=1 ] (48,108) -- (112,108) -- (112,124) -- (48,124) -- cycle ;
\draw [color={rgb, 255:red, 0; green, 0; blue, 0 }  ,draw opacity=1 ]   (92,92) .. controls (92.17,100.17) and (103.83,99.83) .. (104,108) ;
\draw [color={rgb, 255:red, 191; green, 97; blue, 106 }  ,draw opacity=1 ]   (76,92) .. controls (76,108.27) and (32,97.07) .. (32,112) ;
\draw [color={rgb, 255:red, 191; green, 97; blue, 106 }  ,draw opacity=1 ]   (32,112) -- (32,120) ;
\draw [color={rgb, 255:red, 0; green, 0; blue, 0 }  ,draw opacity=1 ]   (80,124) .. controls (80.17,132.17) and (71.83,127.83) .. (72,136) ;
\draw  [color={rgb, 255:red, 191; green, 97; blue, 106 }  ,draw opacity=1 ] (44,136) -- (84,136) -- (84,152) -- (44,152) -- cycle ;
\draw [color={rgb, 255:red, 191; green, 97; blue, 106 }  ,draw opacity=1 ]   (32,120) .. controls (32.17,131.83) and (55.83,123.83) .. (56,136) ;
\draw [color={rgb, 255:red, 191; green, 97; blue, 106 }  ,draw opacity=1 ]   (64,152) -- (64,172) ;
\draw  [color={rgb, 255:red, 255; green, 255; blue, 255 }  ,draw opacity=1 ][fill={rgb, 255:red, 255; green, 255; blue, 255 }  ,fill opacity=1 ] (54.4,102.15) .. controls (54.4,101.05) and (55.3,100.15) .. (56.4,100.15) .. controls (57.5,100.15) and (58.4,101.05) .. (58.4,102.15) .. controls (58.4,103.26) and (57.5,104.15) .. (56.4,104.15) .. controls (55.3,104.15) and (54.4,103.26) .. (54.4,102.15) -- cycle ;
\draw [color={rgb, 255:red, 0; green, 0; blue, 0 }  ,draw opacity=1 ]   (44,92) .. controls (44.17,100.17) and (59.83,99.83) .. (60,108) ;

\draw (100,32) node  [font=\scriptsize,opacity=1 ]  {\emph{what}};
\draw (84,60) node  [font=\scriptsize,color={rgb, 255:red, 191; green, 97; blue, 106 }  ,opacity=1 ]  {\emph{print}};
\draw (84,84) node  [font=\scriptsize,color={rgb, 255:red, 191; green, 97; blue, 106 }  ,opacity=1 ]  {\emph{get}};
\draw (80,116) node  [font=\scriptsize,opacity=1 ]  {\emph{concat}};
\draw (44,32) node  [font=\scriptsize,opacity=1 ]  {\emph{hello}};
\draw (64,144) node  [font=\scriptsize,color={rgb, 255:red, 191; green, 97; blue, 106 }  ,opacity=1 ]  {\emph{print}};
\draw (19.5,60.5) node  [font=\tiny,color={rgb, 255:red, 0; green, 0; blue, 0 }  ,opacity=1 ]  {\emph{greeting}};
\draw (123.5,48.5) node  [font=\tiny,color={rgb, 255:red, 0; green, 0; blue, 0 }  ,opacity=1 ]  {\emph{question}};
\draw (118.5,96.5) node  [font=\tiny,color={rgb, 255:red, 0; green, 0; blue, 0 }  ,opacity=1 ]  {\emph{name}};
\draw (98.5,131.5) node  [font=\tiny,color={rgb, 255:red, 0; green, 0; blue, 0 }  ,opacity=1 ]  {\emph{output}};

\end{tikzpicture}
   \end{minipage}
  \caption{Premonoidal program in arrow do-notation and string diagrams.}
  \label{fig:premonoidalprogram}
\end{figure}
\end{exa}

\section{Related and Further work.}

This manuscript is an extended version of ``Promonads and String Diagrams for
Effectful Categories'' by the first-named author, presented at Applied Category
Theory 2022 \cite{roman:promonadsStringDiagrams}. It also extends and takes
examples from the PhD thesis of the first-named author, which was supervised
by the second-named author \cite{romanThesis}. The present version restructures
the presentation for clarity, includes fully detailed proofs of all results, and
adds a new section on \premonoidalCategories{} (\Cref{sec:stringPremonoidal}).

Staton and Møgelberg \cite{mogelberg14} propose a formalization of Jeffrey's graphical calculus for effectful categories that arise as the Kleisli category of a strong monad.
They prove that \emph{'every strong monad is a linear-use state monad'}, that is, a state monad of the form $\R \multimap !(\bullet) \otimes \R$, where the state $\R$ is an object that cannot be copied nor discarded.
Levy's Call-By-Push-Value \cite{levy2004} is the closest translation of Freyd and \effectfulCategories{} into a programming language that synthesizes the functional and imperative paradigms. Our work can be regarded as providing a string diagrammatic counterpart preserving effects but without higher-order functions.

Relatedly, the ``Functional Machine Calculus'' is a model of higher-order effectful computation developed in the work of Heijltjes, Barrett and McCusker~\cite{heijltjes22:functionalmachine,heijltjes22:functionalmachineii,barrett23:thesis}. The functional machine calculus uses a string diagrammatic language reminiscent of the string diagrams for \effectfulCategories{} that we introduce here; further work formalizing this connection is warranted.
One of the aspects that separates other graphical languages like this one from the string diagrams for \effectfulCategories{} is the occurrence of multiple independent runtimes, corresponding to multiple independent effects. Earnshaw and Nester \cite{earnshaw:devicesEA}, in joint work with this author, explore how multiple runtimes may permit a more fine-grained presentation of \effectfulCategories{}; this opens a future direction for generalizing this work.

Finally, Earnshaw and this second author \cite{earnshaw23:tracetheory} have developed a string diagrammatic trace theory for monoidal formal languages and monoidal automata that uses string diagrams of \effectfulCategories{}.

\newpage
\noindent\textbf{\emph{Acknowledgements.}} The authors want to thank the
anonymous reviewers at LMCS and the reviewers for ``Promonads and String
Diagrams for Effectful Categories'' at Applied Category Theory 2022
\cite{roman:promonadsStringDiagrams}; especially their suggestion of separating
and highlighting the first part of that work, which considerably improved the
narrative. The authors want to thank Matt Earnshaw, Sam Staton, Tarmo Uustalu,
and Niels Voorneveld for many helpful comments and discussions.

\bibliographystyle{alphaurl}
\bibliography{bibliography}

\newcommand{\etalchar}[1]{$^{#1}$}
\begin{thebibliography}{FGM{\etalchar{+}}07}

\bibitem[AC09]{abramsky2009categorical}
Samson Abramsky and Bob Coecke.
\newblock Categorical quantum mechanics.
\newblock In Kurt Engesser, Dov~M. Gabbay, and Daniel Lehmann, editors, {\em Handbook of Quantum Logic and Quantum Structures}, pages 261--323. Elsevier, Amsterdam, 2009.
\newblock URL: \url{https://www.sciencedirect.com/science/article/pii/B9780444528698500104}, \href {https://doi.org/10.1016/B978-0-444-52869-8.50010-4} {\path{doi:10.1016/B978-0-444-52869-8.50010-4}}.

\bibitem[AHS02]{abramsky02}
Samson Abramsky, Esfandiar Haghverdi, and Philip~J. Scott.
\newblock Geometry of interaction and linear combinatory algebras.
\newblock {\em Math. Struct. Comput. Sci.}, 12(5):625--665, 2002.
\newblock \href {https://doi.org/10.1017/S0960129502003730} {\path{doi:10.1017/S0960129502003730}}.

\bibitem[Bar23]{barrett23:thesis}
Chris Barrett.
\newblock On the simply-typed functional machine calculus: Categorical semantics and strong normalisation.
\newblock {\em CoRR}, abs/2305.16073, 2023.
\newblock URL: \url{https://doi.org/10.48550/arXiv.2305.16073}, \href {https://arxiv.org/abs/2305.16073} {\path{arXiv:2305.16073}}, \href {https://doi.org/10.48550/ARXIV.2305.16073} {\path{doi:10.48550/ARXIV.2305.16073}}.

\bibitem[BCST96]{blute96}
Richard~F. Blute, J.~Robin~B. Cockett, Robert~A.G. Seely, and Todd~H. Trimble.
\newblock Natural deduction and coherence for weakly distributive categories.
\newblock {\em Journal of Pure and Applied Algebra}, 113(3):229--296, 1996.
\newblock URL: \url{https://www.sciencedirect.com/science/article/pii/002240499500159X}, \href {https://doi.org/10.1016/0022-4049(95)00159-X} {\path{doi:10.1016/0022-4049(95)00159-X}}.

\bibitem[BGK{\etalchar{+}}22]{bonchi22:rewritingStrings}
Filippo Bonchi, Fabio Gadducci, Aleks Kissinger, Pawel Sobocinski, and Fabio Zanasi.
\newblock String diagram rewrite theory {II:} rewriting with symmetric monoidal structure.
\newblock {\em Math. Struct. Comput. Sci.}, 32(4):511--541, 2022.
\newblock \href {https://doi.org/10.1017/S0960129522000317} {\path{doi:10.1017/S0960129522000317}}.

\bibitem[BHM23]{heijltjes22:functionalmachineii}
Chris Barrett, Willem Heijltjes, and Guy McCusker.
\newblock The functional machine calculus {II:} semantics.
\newblock In Bartek Klin and Elaine Pimentel, editors, {\em 31st {EACSL} Annual Conference on Computer Science Logic, {CSL} 2023, February 13-16, 2023, Warsaw, Poland}, volume 252 of {\em LIPIcs}, pages 10:1--10:18. Schloss Dagstuhl - Leibniz-Zentrum f{\"{u}}r Informatik, 2023.
\newblock URL: \url{https://doi.org/10.4230/LIPIcs.CSL.2023.10}, \href {https://doi.org/10.4230/LIPICS.CSL.2023.10} {\path{doi:10.4230/LIPICS.CSL.2023.10}}.

\bibitem[BSS18]{bonchi18}
Filippo Bonchi, Jens Seeber, and Pawel Sobocinski.
\newblock Graphical conjunctive queries.
\newblock In Dan~R. Ghica and Achim Jung, editors, {\em 27th {EACSL} Annual Conference on Computer Science Logic, {CSL} 2018, September 4-7, 2018, Birmingham, {UK}}, volume 119 of {\em LIPIcs}, pages 13:1--13:23. Schloss Dagstuhl - Leibniz-Zentrum f{\"{u}}r Informatik, 2018.
\newblock \href {https://doi.org/10.4230/LIPIcs.CSL.2018.13} {\path{doi:10.4230/LIPIcs.CSL.2018.13}}.

\bibitem[DDR11]{dumas11}
Jean{-}Guillaume Dumas, Dominique Duval, and Jean{-}Claude Reynaud.
\newblock Cartesian effect categories are {F}reyd-categories.
\newblock {\em Journal of Symbolic Computation}, 46(3):272--293, 2011.
\newblock \href {https://doi.org/10.1016/j.jsc.2010.09.008} {\path{doi:10.1016/j.jsc.2010.09.008}}.

\bibitem[DGNO10]{drinfeld10}
Vladimir Drinfeld, Shlomo Gelaki, Dmitri Nikshych, and Victor Ostrik.
\newblock On braided fusion categories {I}.
\newblock {\em Selecta Mathematica}, 16(1):1--119, 2010.
\newblock \href {https://doi.org/10.1007/s00029-010-0017-z} {\path{doi:10.1007/s00029-010-0017-z}}.

\bibitem[ENR23]{earnshaw:devicesEA}
Matt Earnshaw, Chad Nester, and Mario Román.
\newblock Presentations of {Premonoidal} {Categories} by {Devices}, {Extended} {Abstract}.
\newblock {\em Nordic Workshop on Programming Languages (NWPT'23). Online, \url{https://mroman42.github.io/notes/papers/presentations-of-premonoidals-by-devices.pdf}}, 2023.

\bibitem[ES23]{earnshaw23:tracetheory}
Matthew Earnshaw and Pawel Sobocinski.
\newblock String diagrammatic trace theory.
\newblock In J{\'{e}}r{\^{o}}me Leroux, Sylvain Lombardy, and David Peleg, editors, {\em 48th International Symposium on Mathematical Foundations of Computer Science, {MFCS} 2023, August 28 to September 1, 2023, Bordeaux, France}, volume 272 of {\em LIPIcs}, pages 43:1--43:15. Schloss Dagstuhl - Leibniz-Zentrum f{\"{u}}r Informatik, 2023.
\newblock URL: \url{https://doi.org/10.4230/LIPIcs.MFCS.2023.43}, \href {https://doi.org/10.4230/LIPICS.MFCS.2023.43} {\path{doi:10.4230/LIPICS.MFCS.2023.43}}.

\bibitem[FGM{\etalchar{+}}07]{foster07combinators}
J~Nathan Foster, Michael~B Greenwald, Jonathan~T Moore, Benjamin~C Pierce, and Alan Schmitt.
\newblock Combinators for bidirectional tree transformations: A linguistic approach to the view-update problem.
\newblock {\em ACM Transactions on Programming Languages and Systems (TOPLAS)}, 29(3):17--es, 2007.

\bibitem[Gir89]{girard89}
Jean-Yves Girard.
\newblock Geometry of interaction 1: Interpretation of system f.
\newblock In R.~Ferro, C.~Bonotto, S.~Valentini, and A.~Zanardo, editors, {\em Logic Colloquium '88}, volume 127 of {\em Studies in Logic and the Foundations of Mathematics}, pages 221--260. Elsevier, 1989.
\newblock URL: \url{https://www.sciencedirect.com/science/article/pii/S0049237X08702714}, \href {https://doi.org/10.1016/S0049-237X(08)70271-4} {\path{doi:10.1016/S0049-237X(08)70271-4}}.

\bibitem[Gui80]{guitart1980tenseurs}
Ren{\'e} Guitart.
\newblock Tenseurs et machines.
\newblock {\em Cahiers de topologie et g{\'e}om{\'e}trie diff{\'e}rentielle cat{\'e}goriques}, 21(1):5--62, 1980.
\newblock URL: \url{http://www.numdam.org/item/CTGDC_1980__21_1_5_0/}.

\bibitem[Hei22]{heijltjes22:functionalmachine}
Willem Heijltjes.
\newblock The functional machine calculus.
\newblock {\em CoRR}, abs/2212.08177, 2022.
\newblock URL: \url{https://doi.org/10.48550/arXiv.2212.08177}, \href {https://arxiv.org/abs/2212.08177} {\path{arXiv:2212.08177}}, \href {https://doi.org/10.48550/ARXIV.2212.08177} {\path{doi:10.48550/ARXIV.2212.08177}}.

\bibitem[HMH14]{hoshino14}
Naohiko Hoshino, Koko Muroya, and Ichiro Hasuo.
\newblock Memoryful geometry of interaction: from coalgebraic components to algebraic effects.
\newblock In Thomas~A. Henzinger and Dale Miller, editors, {\em Joint Meeting of the Twenty-Third {EACSL} Annual Conference on Computer Science Logic {(CSL)} and the Twenty-Ninth Annual {ACM/IEEE} Symposium on Logic in Computer Science (LICS), {CSL-LICS} '14, Vienna, Austria, July 14 - 18, 2014}, pages 52:1--52:10. {ACM}, 2014.
\newblock \href {https://doi.org/10.1145/2603088.2603124} {\path{doi:10.1145/2603088.2603124}}.

\bibitem[Hug00]{hughes00}
John Hughes.
\newblock Generalising monads to arrows.
\newblock {\em Science of Computer Programming}, 37(1-3):67--111, 2000.
\newblock \href {https://doi.org/10.1016/S0167-6423(99)00023-4} {\path{doi:10.1016/S0167-6423(99)00023-4}}.

\bibitem[Jef97a]{jeffrey1997premonoidal}
Alan Jeffrey.
\newblock Premonoidal categories and a graphical view of programs.
\newblock {\em Preprint at ResearchGate}, 1997.
\newblock URL: \url{https://www.researchgate.net/profile/Alan-Jeffrey/publication/228639836_Premonoidal_categories_and_a_graphical_view_of_programs/links/00b495182cd648a874000000/Premonoidal-categories-and-a-graphical-view-of-programs.pdf}.

\bibitem[Jef97b]{jeffrey97}
Alan Jeffrey.
\newblock Premonoidal categories and flow graphs.
\newblock {\em Electron. Notes Theor. Comput. Sci.}, 10:51, 1997.
\newblock \href {https://doi.org/10.1016/S1571-0661(05)80688-7} {\path{doi:10.1016/S1571-0661(05)80688-7}}.

\bibitem[JHH09]{jacobs09}
Bart Jacobs, Chris Heunen, and Ichiro Hasuo.
\newblock Categorical semantics for arrows.
\newblock {\em J. Funct. Program.}, 19(3-4):403--438, 2009.
\newblock \href {https://doi.org/10.1017/S0956796809007308} {\path{doi:10.1017/S0956796809007308}}.

\bibitem[JS91]{joyal91}
André Joyal and Ross Street.
\newblock The geometry of tensor calculus, {I}.
\newblock {\em Advances in Mathematics}, 88(1):55--112, 1991.
\newblock URL: \url{https://www.sciencedirect.com/science/article/pii/000187089190003P}, \href {https://doi.org/10.1016/0001-8708(91)90003-P} {\path{doi:10.1016/0001-8708(91)90003-P}}.

\bibitem[KFHB21]{kuhail21:visual}
Mohammad~Amin Kuhail, Shahbano Farooq, Rawad Hammad, and Mohammed Bahja.
\newblock Characterizing visual programming approaches for end-user developers: A systematic review.
\newblock {\em IEEE Access}, 9:14181--14202, 2021.

\bibitem[KKM11]{kilp11:monoids}
Mati Kilp, Ulrich Knauer, and Alexander~V Mikhalev.
\newblock {\em Monoids, Acts and Categories: With Applications to Wreath Products and Graphs. A Handbook for Students and Researchers}, volume~29.
\newblock Walter de Gruyter, 2011.

\bibitem[Lev22]{levy2004}
Paul~Blain Levy.
\newblock Call-by-push-value.
\newblock {\em ACM SIGLOG News}, 9(2):7–29, may 2022.
\newblock \href {https://doi.org/10.1145/3537668.3537670} {\path{doi:10.1145/3537668.3537670}}.

\bibitem[Mog91]{moggi91}
Eugenio Moggi.
\newblock Notions of computation and monads.
\newblock {\em Information and Computation}, 93(1):55--92, 1991.
\newblock \href {https://doi.org/10.1016/0890-5401(91)90052-4} {\path{doi:10.1016/0890-5401(91)90052-4}}.

\bibitem[MS14]{mogelberg14}
Rasmus~Ejlers M{\o}gelberg and Sam Staton.
\newblock Linear usage of state.
\newblock {\em Log. Methods Comput. Sci.}, 10(1), 2014.
\newblock \href {https://doi.org/10.2168/LMCS-10(1:17)2014} {\path{doi:10.2168/LMCS-10(1:17)2014}}.

\bibitem[Pat01]{paterson01}
Ross Paterson.
\newblock A new notation for arrows.
\newblock In Benjamin~C. Pierce, editor, {\em Proceedings of the Sixth {ACM} {SIGPLAN} International Conference on Functional Programming {(ICFP} '01), Firenze (Florence), Italy, September 3-5, 2001}, pages 229--240. {ACM}, 2001.
\newblock \href {https://doi.org/10.1145/507635.507664} {\path{doi:10.1145/507635.507664}}.

\bibitem[Pav13]{pavlovic13}
Dusko Pavlovic.
\newblock Monoidal computer {I:} basic computability by string diagrams.
\newblock {\em Inf. Comput.}, 226:94--116, 2013.
\newblock \href {https://doi.org/10.1016/j.ic.2013.03.007} {\path{doi:10.1016/j.ic.2013.03.007}}.

\bibitem[Pow02]{power02}
John Power.
\newblock Premonoidal categories as categories with algebraic structure.
\newblock {\em Theor. Comput. Sci.}, 278(1-2):303--321, 2002.
\newblock \href {https://doi.org/10.1016/S0304-3975(00)00340-6} {\path{doi:10.1016/S0304-3975(00)00340-6}}.

\bibitem[PR97]{power97}
John Power and Edmund Robinson.
\newblock Premonoidal categories and notions of computation.
\newblock {\em Math. Struct. Comput. Sci.}, 7(5):453--468, 1997.
\newblock \href {https://doi.org/10.1017/S0960129597002375} {\path{doi:10.1017/S0960129597002375}}.

\bibitem[PT99]{power99:freyd}
John Power and Hayo Thielecke.
\newblock Closed freyd- and kappa-categories.
\newblock In Jir{\'{\i}} Wiedermann, Peter van Emde~Boas, and Mogens Nielsen, editors, {\em Automata, Languages and Programming, 26th International Colloquium, ICALP'99, Prague, Czech Republic, July 11-15, 1999, Proceedings}, volume 1644 of {\em Lecture Notes in Computer Science}, pages 625--634. Springer, 1999.
\newblock \href {https://doi.org/10.1007/3-540-48523-6\_59} {\path{doi:10.1007/3-540-48523-6\_59}}.

\bibitem[Rom22a]{marionotes2022}
Mario Rom\'an.
\newblock Notes on {J}effrey's {A} {G}raphical {V}iew of {P}rograms.
\newblock Online, \url{https://www.ioc.ee/~mroman/data/talks/premonoidalgraphicalview.pdf}, March 2022.

\bibitem[Rom22b]{roman:promonadsStringDiagrams}
Mario Rom{\'{a}}n.
\newblock Promonads and string diagrams for effectful categories.
\newblock In Jade Master and Martha Lewis, editors, {\em Proceedings Fifth International Conference on Applied Category Theory, {ACT} 2022, Glasgow, United Kingdom, 18-22 July 2022}, volume 380 of {\em {EPTCS}}, pages 344--361, 2022.
\newblock \href {https://doi.org/10.4204/EPTCS.380.20} {\path{doi:10.4204/EPTCS.380.20}}.

\bibitem[Rom23]{romanThesis}
Mario Román.
\newblock Monoidal {C}ontext {T}heory. {PhD} {T}hesis.
\newblock {\em Tallinn University of Technology, available at \url{https://mroman42.github.io/notes/papers/monoidal-context-theory.pdf}}, 2023.
\newblock Supervised by Pawe{\l} Soboci{\'n}ski.

\bibitem[Sch01]{schweimeier01}
Ralf Schweimeier.
\newblock {\em Categorical and graphical models of programming languages}.
\newblock PhD thesis, University of Sussex, {UK}, 2001.
\newblock URL: \url{https://ethos.bl.uk/OrderDetails.do?uin=uk.bl.ethos.366059}.

\bibitem[Sel10]{selinger2010survey}
Peter Selinger.
\newblock A survey of graphical languages for monoidal categories.
\newblock In {\em New structures for physics}, pages 289--355. Springer, 2010.

\bibitem[Shu16]{shulman2016categorical}
Michael Shulman.
\newblock Categorical logic from a categorical point of view.
\newblock {\em Available on the web}, 2016.
\newblock URL: \url{https://mikeshulman.github.io/catlog/catlog.pdf}.

\bibitem[Shu18]{shulman20182}
Michael Shulman.
\newblock The {2-Chu-Dialectica} construction and the polycategory of multivariable adjunctions.
\newblock {\em arXiv preprint arXiv:1806.06082}, 2018.
\newblock \href {https://doi.org/10.48550/arxiv.1806.06082} {\path{doi:10.48550/arxiv.1806.06082}}.

\bibitem[SL13]{staton13}
Sam Staton and Paul~Blain Levy.
\newblock Universal properties of impure programming languages.
\newblock In Roberto Giacobazzi and Radhia Cousot, editors, {\em The 40th Annual {ACM} {SIGPLAN-SIGACT} Symposium on Principles of Programming Languages, {POPL} '13, Rome, Italy - January 23 - 25, 2013}, pages 179--192. {ACM}, 2013.
\newblock \href {https://doi.org/10.1145/2429069.2429091} {\path{doi:10.1145/2429069.2429091}}.

\bibitem[{Tod}10]{trimble:coherence}
{Todd Trimble}.
\newblock Coherence theorem for monoidal categories (nlab entry), section 3. discussion, 2010.
\newblock \url{https://ncatlab.org/nlab/show/coherence+theorem+for+monoidal+categories}, Last accessed on 2022-05-10.

\bibitem[UV08]{uustalu2008comonadic}
Tarmo Uustalu and Varmo Vene.
\newblock Comonadic notions of computation.
\newblock In Jiří Ad{\'{a}}mek and Clemens Kupke, editors, {\em Proceedings of the Ninth Workshop on Coalgebraic Methods in Computer Science, {CMCS} 2008, Budapest, Hungary, April 4-6, 2008}, volume 203 of {\em Electronic Notes in Theoretical Computer Science}, pages 263--284. Elsevier, 2008.
\newblock \href {https://doi.org/10.1016/j.entcs.2008.05.029} {\path{doi:10.1016/j.entcs.2008.05.029}}.

\end{thebibliography}

\newpage
\appendix

\end{document}